\documentclass[main, final]{article}

    \PassOptionsToPackage{numbers}{natbib}

\usepackage{neurips_2026}
\usepackage{microtype}
\usepackage{graphicx}
\usepackage{subcaption} %
\usepackage{booktabs} %

\usepackage{hyperref}

\usepackage{xurl}

\usepackage{amsmath}
\usepackage{amssymb}
\usepackage{mathtools}
\usepackage{amsthm}
\usepackage{tabularx}
\usepackage{makecell}

\usepackage[capitalize,noabbrev]{cleveref}

\usepackage[dvipsnames]{xcolor}
\usepackage{xspace}
\usepackage{thmtools}
\usepackage{thm-restate}
\declaretheorem[name=Theorem,numberwithin=section]{theorem}%
\declaretheorem[name=Lemma,sibling=theorem]{lem}

\usepackage{algorithm}%

\usepackage{algpseudocode}

\usepackage{enumitem}
\usepackage{aligned-overset}
\newcommand{\grad}{\nabla}

\newcommand{\norm}[2]{\left\lVert #1\right\rVert_{#2}}

\DeclareMathOperator{\dom}{dom}

\definecolor{mydarkgreen}{RGB}{39,130,67}
\newcommand{\green}{\color{mydarkgreen}}

\newcommand{\algname}[1]{{\green \small \sf #1}\xspace}

\newcommand{\gladssn}{\algname{GLAd-SSN}}
\newcommand{\adam}{\algname{Adam}}
\newcommand{\leapssn}{\algname{LeAP-SSN}}

\newcommand{\lzsymb}{m}

\newcommand{\HH}{\mathcal{H}}

\newcommand{\F}{\mathcal{F}}
\newcommand\so{
  \mathchoice
    {{\scriptstyle\mathcal{O}}}%
    {{\scriptstyle\mathcal{O}}}%
    {{\scriptscriptstyle\mathcal{O}}}%
    {\scalebox{.7}{$\scriptscriptstyle\mathcal{O}$}}%
  }

\newcommand\N{\mathbb{N}}
\newcommand\R{\mathbb{R}}

\definecolor{darkorange}{rgb}{8.,.4,0.}

\newcommand{\rev}[1]{#1} 

\usepackage{pifont}
\newcommand{\greencheck}{{\color{green}\checkmark}}
\newcommand{\redcross}{{\color{red}\ding{55}}}

\usepackage{array}

\theoremstyle{plain}

\theoremstyle{definition}

\newtheorem{ass}{Assumption}
\theoremstyle{remark}
\newtheorem{remark}[theorem]{Remark}
\crefname{ass}{Assumption}{Assumptions}

\usepackage[textsize=tiny]{todonotes}
 \usepackage{multirow} 

\hypersetup{
  colorlinks=true,
  citecolor=MidnightBlue,
  linkcolor=purple!80
  }

\title{Skip the Hessian, Keep the Rates: Globalized Semismooth Newton with Lazy Hessian Updates}
\author{%
  Amal Alphonse%
  \\
  Faculty of Mathematics\\
  University of Duisburg-Essen\\
  Thea-Leymann-Straße 9\\
  45127 Essen, Germany \\
  \texttt{amal.alphonse@uni-due.de} \\
  \And
  Pavel Dvurechensky \\
  Weierstrass Institute \\
Anton-Wilhelm-Amo-Straße 39\\
10117 Berlin, Germany \\
  \texttt{pavel.dvurechensky@wias-berlin.de} \\
  \AND
  Clemens Sirotenko \\
  Weierstrass Institute \\
Anton-Wilhelm-Amo-Straße 39\\
10117 Berlin, Germany \\
  \texttt{clemens.sirotenko@wias-berlin.de} \\
}

\begin{document}

\maketitle

\begin{abstract}
Second-order methods are provably faster than first-order methods, and their efficient implementations for large-scale optimization problems have attracted significant  attention. Yet, optimization problems in ML often have nonsmooth derivatives, which makes the existing convergence rate theory of second-order methods inapplicable. In this paper, we propose a new semismooth Newton method (SSN) that enjoys both global convergence rates and asymptotic superlinear convergence without requiring second-order differentiability. Crucially, our method does not require (generalized) Hessians to be evaluated at each iteration but only periodically, and it reuses stale Hessians otherwise (i.e., it performs lazy Hessian updates), saving compute cost and often leading to significant speedups in time, whilst still maintaining strong global and local convergence rate guarantees. We develop our theory in an infinite-dimensional setting and illustrate it with numerical experiments on matrix factorization and neural networks with Lipschitz constraints.

\end{abstract}

\section{Introduction}
Recent developments in second-order optimization methods showcase their superior convergence rates compared to first-order methods \rev{\cite{cartis2011adaptive,nesterov2018lectures,nesterov2006cubic}} and propose efficient implementations that are scalable enough to solve large problems arising in ML \rev{\cite{liu2024sophia,gupta2018shampoo}}.
The efficiency of second-order methods is theoretically well understood for smooth problems with Lipschitz Hessians where they provably find second-order stationary points of non-convex functions and have fast convergence rates for minimizing convex functions. 
At the same time, optimization problems in ML often have nonsmooth elements, e.g., training DNNs with ReLU, \rev{Huberized ReLU, or ReLU$^2$} nonlinearities, (kernel) support vector machines \cite{jiang2020semismooth,yin2019}, kernel-based optimal transport problems \cite{lin2024specialized}, partial correlation network estimation \cite{kim2025partial}, learning problems with LASSO and other nonsmooth (group) sparsity-inducing penalizations  \cite{shi2020semismooth,Luo2019Solving,Hu2024projected}, square-root regression \cite{Tang2020Sparse}.
Nonsmooth derivatives naturally appear in optimization with inequality constraints that are treated by penalization, e.g., in nonnegative PCA \cite{Hu2024projected}, nonnegative matrix factorization \cite{Chow2016Analysis}, nonnegative least squares, nonparametric maximum likelihood estimation in mixture models \cite{Zhang2024Efficient}. 
Moreover, \rev{several well-established} losses such as Huber, squared hinge, or \rev{the sparsemax loss \cite{sparsemax} have piecewise-defined second derivatives that exhibit jump discontinuities at certain transition points}. The majority of the global results on second-order methods are inapplicable in these situations since the Hessians are not Lipschitz. 

An important property of classical Newton methods in the smooth world is their local superlinear convergence. This property can be ported to the nonsmooth world based on the notion of semismooth derivatives and the celebrated semismooth Newton method (SSN) \rev{\cite{MR1786137,MR1250115,MR1216791,MR1972217,hintermuller2002primal}}. This method is widely used for ML applications as witnessed by the above-mentioned references, where the objectives have semismooth and even $\gamma$-order/strongly semismooth derivatives. Nevertheless, there is a gap in the theory of global rates for second-order methods in the $\gamma$-order/strongly semismooth setting.%

A drawback of Newton methods is their potentially expensive iteration that requires evaluating the Hessian and solving a linear system. Assuming Lipschitz Hessians, it was shown in \cite{doikov2023second} that Hessian updates can be skipped for some iterations (i.e., one can perform lazy Hessian updates) in regularized Newton methods; such results are not known in the nonsmooth setting.

To sum up, existing second-order methods for optimization involving nonsmooth derivatives suffer from expensive iterations and lack the full picture of theoretical convergence guarantees that covers both global rates and local asymptotic behavior in all the semismooth and $\gamma$-order semismooth settings. The goal of this paper is to close these gaps. 
 
We consider the composite minimization problem
\begin{equation}
   \min_{x \in \HH} \{F(x)\coloneqq f(x) + \psi(x) \}, \tag{P}\label{eq:opt}
\end{equation}
where $\HH$ is a Hilbert space,  $\psi\colon \HH \to \overline{\mathbb{R}}$ is a potentially nonsmooth but simple convex function, and $f \in C^{1,1}(\HH; \mathbb{R})$ is a (possibly nonconvex) function such that the Fréchet derivative $f'$ is not necessarily smooth.  
\renewcommand{\theHequation}{unique.\arabic{equation}}

\paragraph{Related works.} 

Our paper complements several strands of existing literature on second-order methods. First, in finite dimensions, global rates are obtained in many settings, including convex and non-convex smooth problems, see, e.g.,  \cite{nesterov2006cubic,cartis2011adaptive,nesterov2008accelerating,monteiro2013accelerated,gasnikov2019near,carmon2022optimal,kovalev2022first,mishchenko2023regularized,doikov2024gradient,gratton2025yet}, with the main assumption that the Hessian is Lipschitz. 
Nonsmooth problems and global rates under the umbrella of a H\"older Hessian assumption were considered, e.g., in \cite{grapiglia2017regularized,grapiglia2019accelerated,CarGouToi19,doikov2024super}.
These papers use global assumptions on the Hessian and do not consider asymptotic rates under semismoothness at the solution or local convergence.

On the other hand, classical works \cite{MR1972649, MR1786137, MR1250115, MR1216791, MR1972217, hintermuller2002primal} show including in infinite-dimensional settings, local superlinear convergence of SSN -- a generalization of the Newton method that uses generalized second derivatives instead of the Hessian for nonsmooth problems. Globalized convergence results are also obtained, e.g., in \cite{pang1995globally,qi1995trust,facchinei2003finite,xiao2016semismooth,cui2021modern,potzl2022second,potzl2024inexact,khanh2024globally,mordukhovich2024second,wachsmuth2025globalized,ouyang2025trust,Gfrerer2025}, yet the question of global convergence rates is not answered in these papers. 
\cite{LeapSSN} obtain global rates, but consider neither the $\gamma$-order semismooth setting nor allow for lazy Hessian updates.

Reusing Hessians from previous iterates dates back to Shamanskii \cite{Shamanskii1967AMO}, who established local convergence for nonlinear systems. More recently, \cite{doikov2023second,chen2025computationally} showed that modern globalized Newton methods can reuse Hessians with only minor complexity degradation while reducing per-iteration cost, and established global rates. However, these works rely on Lipschitz Hessian assumptions, and it is not known whether the same can be achieved for nonsmooth problems.

\paragraph{Contributions.} 
To address the outlined gaps, we propose the first SSN combining all four desiderata: lazy Hessian updates, global convergence rates, local superlinear convergence under semismoothness, and improved local rates under \(\gamma\)-order semismoothness, satisfied by many ML applications. Our analysis is carried out in possibly infinite-dimensional Hilbert spaces and, in finite dimensions, yields superlinear convergence for \(f\in C^2\), improving existing results. Importantly, \textit{all guarantees are achieved simultaneously by a single class- and parameter-agnostic method} without requiring knowledge of the objective class or problem parameters\footnote{\rev{\cref{sec:app_discussion} contains a table summarizing assumptions and results, providing also additional discussions, and \cref{sec:limitations} discusses limitations of the results.}}. We also present experiments on constrained NN training, matrix factorization, and support vector classification.

\paragraph{Notation.} 

The dual to the space $\HH$ is denoted by $\HH^*$; $\mathcal{R}\colon\HH \to \HH^*$ denotes the Riesz map. We use $\|\cdot\|$ to denote the norm on $\HH$ and $\|\cdot\|_*$ to denote the norm on $\HH^*$. We write $\langle g, x \rangle$ to denote the duality pairing of $g\in \HH^*$ and $x \in \HH$. For $x\in \HH$, we set $B_R(x)\coloneqq \{u\in \HH:\|x-u\|\leq R\}$. For a point $x\in \HH$ and a closed set $A \subseteq \HH$, we denote ${\rm dist}(x,A)\coloneqq \inf_{u\in A}\|x-u\|_{\HH}$. For  $G \colon \HH \to \overline{\mathbb{R}}\coloneqq \R \cup \{+\infty\}$, we denote $\dom G\coloneqq \{x\in \HH: G(x)<+\infty\}$. For a proper convex function $\psi \colon \HH \to \overline{\mathbb{R}}$, we denote by $\partial \psi$ its convex subdifferential. We also use $\psi'(x)$ to denote an element of $\partial \psi(x)$. For a function $f\colon \HH \to \overline{\mathbb{R}}$ that is $C^{1}$ on $\dom f$, i.e., continuously Fréchet differentiable, we denote by $f'(x)\in \HH^*$ its Fréchet derivative at $x\in \dom f$. $C^2$ denotes the class of twice continuously Fréchet differentiable functions. For an operator $H \in \mathcal{L}(\HH,\HH^*)$, $H  \succeq 0$ means that $H$ is positive semidefinite.

\section{Semismooth Newton Method}
\label{sec:algorithm} 

\rev{We start by making the following standard assumption regarding the problem \eqref{eq:opt}.}
\begin{ass}%
\label{ass:basic}
\begin{enumerate}[label=(\roman*)]\itemsep=-0.5em
    \item The function $f\colon \HH \to \mathbb{R}$ is continuously Fréchet differentiable and $f'$ is $\frac{L}{2}$-Lipschitz. 
    \item The function $\psi\colon \HH \to \overline{\mathbb{R}}$ is convex, proper, and lower semicontinuous.
    \item The function $F\coloneqq f+\psi$ is bounded from below by the global minimal value $F^*$ and problem \eqref{eq:opt} has a nonempty solution set $S\coloneqq \{x^* \in \HH: F(x^*)=F^*\}$.
    \item Given a point $x \in \HH$, we can evaluate a linear self-adjoint operator\footnote{To have a flexible algorithm, we do not specify at this moment how $H$ is related to $f$. This can be the second derivative or Hessian if $f \in C^2$, a generalized second-order derivative if $f'$ is nonsmooth, or a \rev{Gauss--Newton} or quasi-Newton approximation for the (generalized) second-order derivative. \rev{Positive semidefinite approximations are possible by projecting self-adjoint second-order approximations onto the cone of positive semidefinite operators.} We also use the same constant $L$ as for the Lipschitz constant of $f'$ since it is w.l.o.g. For simplicity, we just refer to $H$ as the Hessian (or as second-order information). See also the discussion in \cref{S:Newton_derivative}. \rev{Finally, $H$ is taken to be self-adjoint w.l.o.g. since if it is not the case, we can simply replace $H$ by $\frac{1}{2}(H+H^*)$.}} $H(x)$ such that the operator norm is bounded, i.e., $\norm{H(x)}{\mathrm{op}}\coloneqq \sup_{h\in \HH, h \ne 0} \frac{\|H(x)h\|_*}{\|h\|}\leq \frac{L}{2}$ for all $x \in \dom f$.

\end{enumerate} 
\end{ass}

\begin{algorithm}[t]
\caption{\gladssn\ (\underline{G}lobalized \underline{L}azy \underline{Ad}aptive
\underline{S}emi\underline{s}mooth \underline{N}ewton Method)}
\label{alg:proximal_newton}
\begin{algorithmic}[1]
\State \textbf{input:} Parameters $p \in [0,1]$, $\lzsymb \in \mathbb{N}$,
$x_0\in \dom F$, $\psi'(x_0) \in \partial \psi(x_0)$, and $\Lambda_0 >0$.
\For{$k=0,1,\ldots$}\label{step:outer_loop}
    \For{$j_k=0,1,\ldots$} \label{step:inner_loop}
        \State Set $\lambda=4^{j_k}\Lambda_k\norm{F'(x_k)}{*}^p$.
        \State \parbox[t]{0.9\linewidth}{Using $H$ evaluated at iteration $\pi(k) := k-k\bmod \lzsymb$  (cf. \eqref{eq:def_m}), attempt to %
        compute $x_+$ as a stationary point of}
        \begin{equation}
        \min_{y \in \mathcal{H}}
        f(x_k)+\langle f'(x_k), y-x_k \rangle
        + \frac12 \langle H(x_{\pi(k)})(y-x_k), y-x_k \rangle
        + \frac{\lambda}{2}\norm{y-x_k}{}^2
        + \psi(y). \label{eq:prox_step_alg}
        \end{equation}
\Statex \hspace*{\dimexpr\algorithmicindent+\algorithmicindent\relax}%
\parbox[t]{\dimexpr\linewidth-\algorithmicindent-\algorithmicindent\relax}{%
    \rev{If not possible, go to next iteration of the inner loop
    (line \ref{step:inner_loop}).}%
}
        \State Set
        $
        \psi'(x_+) =
        -f'(x_k)
        - H(x_{\pi(k)})(x_+-x_k)
        - \lambda \mathcal{R}(x_+-x_k).
        $
        \State Set $F'(x_+) = f'(x_+) + \psi'(x_+)$.
        \If{ 
        \begin{align}
         \langle F'(x_+),x_k -x_+\rangle \geq \frac{1}{2\lambda}\norm{F'(x_+)}{*}^2 \quad\text{and}\quad
            F(x_k)-F(x_+)\geq \frac{\lambda}{4} \norm{x_+-x_k}{}^2,\label{eq:stop_cond_alg}
        \end{align}
        }
            \State Set $x_{k+1}=x_+$, $\lambda_k=\lambda$, and
            $\Lambda_{k+1}=4^{j_k}\Lambda_k/4$.
            \State Set $\psi'(x_{k+1})=\psi'(x_+)$ and
            $F'(x_{k+1})=f'(x_{k+1})+\psi'(x_{k+1})$.
            \State \textbf{break} and go to next iteration of the outer loop (line \ref{step:outer_loop}).
        \EndIf
    \EndFor
\EndFor
\end{algorithmic}
\end{algorithm}

Throughout, \(\partial F(x)\) denotes the limiting subdifferential of \(F\), which under the above assumption coincides with the Fr\'echet subdifferential \cite{RockafellarWets,Mordukhovich1}. We write \(F'(x)\in \partial F(x)\).

Our algorithm, \cref{alg:proximal_newton}, is based on quadratic regularization of the proximal Newton step with lazy Hessian updates, see \eqref{eq:PLMSN_step}. As in \cite{doikov2023second}, to implement lazy updates of the second-order information, we (re)use the value of the operator $H$ evaluated at iteration 
\begin{equation}
\pi(k) := k-k\bmod \lzsymb\label{eq:def_m},
\end{equation}
meaning that we only evaluate $H$ every $\lzsymb$ iterations. In particular, when $\lzsymb=1$, $\pi(k)=k$ and we get the standard Newton step. We also use a simple backtracking strategy 
and adaptively find the regularization parameter $\lambda$ in each iteration so that it guarantees sufficient progress of the algorithm, \rev{see the inner loop of \cref{alg:proximal_newton}}. At the same time, inspired by the universality of the method in \cite{doikov2024super}, $\lambda$ depends on the subgradient norm of $F$ at the current iterate, that is,  we use \textit{gradient regularization}. The resulting algorithm, which we christen \gladssn, is listed as \cref{alg:proximal_newton}. 
Importantly, our algorithm does not require knowledge of the objective class (convex, non-convex, satisfying PL condition, ($\gamma$-order) semismoothness), nor does it require knowledge of the parameters of the objective such as $L$.

The main step of the algorithm is defined through the following map %
\begin{equation}\label{eq:PLMSN_step}
\begin{aligned}
&x_+(\lambda,x,z)
\coloneqq \arg\min_{y \in \HH} \Bigl\{ \langle f'(x),y-x\rangle %
+\frac12\langle H(z)(y-x),y-x\rangle  +\frac{\lambda}{2}\|y-x\|^2 + \psi(y)
\Bigr\}. 
\end{aligned}
\end{equation}

By the stationarity condition for \eqref{eq:PLMSN_step}, for \(x_+ := x_+(\lambda,x,z)\), there exists \(\psi'(x_+) \in \partial \psi(x_+)\) s.t.
\begin{align}
\hspace{-1em}\psi'(x_+)
&= - f'(x) - H(z)(x_+ - x) - \lambda \mathcal{R}(x_+ - x),
\label{eq:PLMSN_step_optimality}\\
\hspace{-1em}F'(x_+)
&\coloneqq f'(x_+) + \psi'(x_+)  = f'(x_+) - f'(x) - H(z)(x_+ - x) - \lambda \mathcal{R}(x_+ - x)\in \partial F(x_+) .\label{eq:F_div_def}
\end{align}
Here $\mathcal{R}\colon\HH \to \HH^*$ is the Riesz map; in the finite dimensional setting, if the norm $\norm{x}{} = \sqrt{\langle Bx, x \rangle}$ is used for a self-adjoint positive-definite linear map $B$, $\mathcal{R}$ can be replaced with $B$ in \eqref{eq:PLMSN_step_optimality} and \eqref{eq:F_div_def}.

To streamline the presentation, we introduce the following notation for $k\geq 0$:
\begin{equation}
\label{eq:F_r_g_notaion}
\begin{alignedat}{2}
F_k      &\coloneqq F(x_k)-F^*,            &\qquad
r_k      &\coloneqq \norm{x_k-x_{k+1}}{}, \qquad r_{k,+}  \coloneqq \norm{x_k-x_{k,+}}{}\\
g_k      &\coloneqq \norm{F'(x_k)}{*},     &
x_{k,+}  &\coloneqq x_+(\lambda_k/4,x_k\rev{, x_{\pi(k)}}).
\end{alignedat}
\end{equation}
 
\begin{remark}
    The inner loop of \cref{alg:proximal_newton} is well defined
    and eventually terminates even if $H$ is indefinite.
    Indeed, by \cref{ass:basic},
    $\langle H(z)d,d\rangle\geq -\frac{L}{2}\|d\|^2$,
    so \eqref{eq:prox_step_alg} admits a unique minimizer
    whenever $\lambda>L/2$.     More generally, a stationary point of \eqref{eq:prox_step_alg}
    exists under any of the following conditions:
    \vspace{-.15cm}
    \begin{enumerate}[label=(\roman*)]\itemsep=-0.5em
        \item $H(x_{\pi(k)})+\lambda\mathcal R
        \succeq\delta\mathcal R$ for some $\delta>0$;
        \item $\psi$ is $\rho$-strongly convex and
        $\lambda+\rho>L/2$;
        \item $\psi\equiv0$ and
        $H(x_{\pi(k)})+\lambda\mathcal R$
        is boundedly invertible;
        \item $\dom\psi$ is nonempty and compact
        (equivalently, closed and bounded in finite dimensions).
    \end{enumerate}
\end{remark}

\section{Global convergence rates}\label{sec:global}
As will be shown, the convergence behavior \rev{of \gladssn} is largely governed by \((\lambda_k)_{k\in \N}\). We therefore first show that the inner loop in \cref{alg:proximal_newton} terminates finitely, implying also boundedness of \((\lambda_k)_{k\in \N}\). \rev{The proof is given in \cref{sec:global_proofs}.}

 \begin{restatable}{lem}{lemacceptanceinnernonconv}\label{lem:acceptance_inner_nonconv}
Let \cref{ass:basic} hold %
Consider iteration $k \geq 0$ of \cref{alg:proximal_newton} applied to \eqref{eq:opt}.
If $g_k>0$ and $\lambda \geq L$, the subproblems in \eqref{eq:prox_step_alg} are well posed, and both inequalities  in \eqref{eq:stop_cond_alg} of \cref{alg:proximal_newton} hold, i.e., %
all acceptance conditions of the inner loop of \cref{alg:proximal_newton} hold and the inner loop ends after a finite number of trials.
Furthermore, if $g_{k-1}>0$ and $H(x)\succeq0$ for all $x \in \mathcal{H}$, we have for every $k \geq 0$,
\begin{equation}
    \label{eq:lambda_bound}
    \lambda_k\leq \overline{\lambda}_{} \coloneqq \max\{4L,\Lambda_{0}g_{0}^p\}, \quad \Lambda_{k+1}\leq \overline{\lambda}_{}/(4g_k^p).
\end{equation}
\end{restatable}
Since \(g_k=\|F'(x_k)\|_* = 0\) implies that \(x_k\) solves \eqref{eq:opt}, we assume \(g_k>0\) for all \(k\ge0\).

\subsection{Non-convex problems}
In this subsection, we focus on the setting where $F$ in \eqref{eq:opt} is not necessarily convex. We show that under  \cref{ass:basic},  \cref{alg:proximal_newton} has a global nonasymptotic convergence rate $\mathcal{O}(1\slash \sqrt{k})$ in terms of the minimal subgradient norm on the trajectory. The formal result, \rev{proved in \cref{sec:global_proofs_nonconv},} is as follows.

\begin{restatable}{theorem}{thmglobalconvergencesublinearnonconv}
    \label{thm:global_convergence_sublinear_nonconv}
    Let  \cref{ass:basic} hold \rev{and assume in addition that $H(x)\succeq0$ for all $x \in \mathcal{H}$}. %
    Then, the sequence $(x_k)_{k \in \mathbb{N}}$ is well-defined with $(x_k)_{k \in \mathbb{N}} \subset \F_0$ %
    where $\F_0\coloneqq \{ x \in \dom F : F(x) \leq F(x_0)\}$ is the sublevel set, and 
    $\norm{F'(x_{k})}{*} \to 0$ as $k\to \infty$. Moreover, the following global nonasymptotic convergence rate holds for $k \geq 1$:
    \begin{equation}
    \label{eq:nonconv_rate}
        \min_{0\leq i \leq k-1}\norm{F'(x_{i+1})}{*}\leq 4 \sqrt{\overline{\lambda}F_0/k}.%
    \end{equation}
    If in addition $\lambda_k \to 0$  as $k\to \infty$, we obtain that $\min_{0\leq i \leq k-1}\norm{F'(x_{i+1})}{*} = \so(1\slash \sqrt{k})$.
    Finally, the number of Newton steps \eqref{eq:prox_step_alg}
    up to the end of iteration $k$ does not exceed $\rev{2}(k+1)+\log_4\frac{\overline{\lambda}} {4g_k^p\Lambda_0}$ and the number of Hessian evaluations does not exceed $\lceil \frac{k\rev{+1}}{m} \rceil$.
\end{restatable}
\rev{
\begin{remark}[On indefinite operators $H(x)$]\label{rem:generalHnonconvex}
If the operator $H(x)$ is not positive semidefinite, establishing the boundedness of $(\lambda_k)_{k \in \mathbb{N}}$ appears to be challenging without e.g. the availability of a PL condition (see the next section), which hinders a convergence result of the form above. Nevertheless, even under these relaxed assumptions, we can show the weaker convergence result  $\liminf_{k\to\infty}\norm{F'(x_k)}{*} = 0$.
A proof of this nontrivial result is given in \Cref{thm:liminf_nonconv}.
\end{remark}
}

\subsection{Non-convex problems under PL condition} 
\label{subsec:global_linear_under_global_PL}
In this subsection, we still allow that $F$ in \eqref{eq:opt} may be nonconvex, but we additionally assume that $F$ satisfies the Polyak--\L{}ojasiewicz (PL) condition -- a widely accepted condition in the ML community, see, e.g., \cite{karimi2016linear,Wang2024Monotone}, especially in connection to over-parameterized models \cite{liu2022loss,chen2023pl}. 
This allows us to obtain a global nonasymptotic \emph{linear} convergence rate for \cref{alg:proximal_newton} in terms of the objective functional values and the iterates. %
\rev{The proof is given in \cref{sec:global_proofs_PL}.}
\begin{restatable}[Global linear convergence]{theorem}{thmglobalconvergencenonconvPL}
    \label{thm:global_convergence_nonconv_PL}
    Let \cref{ass:basic} hold and $F$ satisfy the Polyak--\L{}ojasiewicz (PL) condition for some $\mu >0$:
\begin{equation}
    \label{eq:PL}
        \frac{1}{2\mu} {\rm dist}(0, \partial F(x))^2 \geq F(x)-F^*, \quad \forall x \in \dom F.
\end{equation}
    Then, the following statements hold. %
    \begin{enumerate}[label=(\roman*)]\itemsep=-0.5em
    \item \rev{The sequence $(\lambda_k)_{k \in \mathbb{N}}$ is bounded. More precisely, for every $k \in \mathbb{N}$, we have 
    \begin{equation} 
    \lambda_k \leq \overline{\lambda}_{\rm PL} := \overline{\lambda} \max\left\{1, \frac{1}{4}\left(\frac{2\overline{\lambda}}{\mu}\right)^{\frac{p}{2-p}}\right\}, \qquad \overline{\lambda} := \max\{4L,\Lambda_0 g_0^p\}.    \label{eq:boundlambdakunderPL}
    \end{equation} }
    \item The objective values converge linearly, i.e., for $k \geq 0$,%
        \begin{equation}
        \label{eq:nonconv_PL_rate}
        F(x_{k})-F^* \leq \exp\left(-\frac{\mu}{\mu+8 \rev{\overline{\lambda}_{\rm PL}}} \cdot k\right)\cdot(F({x}_0)-F^*).%
        \end{equation}
        If  $\lambda_k \to 0$ as $k\to \infty$, then $F(x_k)-F^*$ converges to $0$ superlinearly.
        \item %
        If $p<1$ and the sublevel set $\F_0$ is bounded, i.e.,
        \begin{equation}
            \label{eq:bounded_sublevel_set}
             D_0\coloneqq  \sup_{x, y \in \mathcal{F}_0} \norm{x-y}{} < +\infty,  
        \end{equation}
        then $x_k$ converges strongly and linearly to a global minimum  $x^*$, and $\norm{F'(x_k)}{*}$ converges linearly to $0$.
        If $\lambda_k \to 0$ as $k\to \infty$, then $x_k \to x^*$ and $\norm{F'(x_k)}{*} \to 0$ superlinearly. %
    \end{enumerate}
    \rev{Finally, the number of Newton steps \eqref{eq:prox_step_alg}
    up to the end of iteration $k$ does not exceed $\rev{2}(k+1)+\log_4\frac{\overline{\lambda}_{\rm PL}} {4g_k^p\Lambda_0}$ and the number of Hessian evaluations does not exceed $\lceil \frac{k\rev{+1}}{m} \rceil$.}
    \end{restatable}

\subsection{Convex problems}
\label{sec:global_convex_functions}
In this subsection, we focus on the setting where $F$ in \eqref{eq:opt} is convex, which allows for faster sublinear rates.  %
The main result of this subsection is as follows.\footnote{ Note that all the results in \cref{thm:global_convergence_sublinear_nonconv} hold also in this case.} \rev{The proof can be found in \cref{sec:global_proofs_convex}.}

\begin{restatable}{theorem}{thmglobalconvergencesublinear}
\label{thm:global_convergence_sublinear}
Let  \cref{ass:basic} hold \rev{and assume that $H(x)\succeq0$ for all $x \in \mathcal{H}$.} %
Additionally let $F$  have a bounded sublevel set $\mathcal{F}_0$ in the sense of \eqref{eq:bounded_sublevel_set} and be convex. 
Then, 
\begin{equation}
    \label{eq:conv_rate}
        F(x_k) - F^* \leq g_0D_0\exp\left(-\frac{k}{4}\right) + \frac{32\overline{\lambda}D_0^2 }{  k}, \quad k \geq 0.
\end{equation}
If in addition $\lambda_k \to 0$ as $k\to \infty$ (which holds, e.g., if $f$ is locally $C^2$, $F$ is strictly convex, $H:=\nabla^2 f(x)$, and $\dim(\HH) <  \infty$), we obtain that $F(x_k)-F^*=\so(1\slash k)$. \rev{Finally, the number of Newton steps \eqref{eq:prox_step_alg}
    up to the end of iteration $k$ does not exceed $\rev{2}(k+1)+\log_4\frac{\overline{\lambda}} {4g_k^p\Lambda_0}$ and the number of Hessian evaluations does not exceed $\lceil \frac{k\rev{+1}}{m} \rceil$.}
\end{restatable}

\section{Superlinear convergence}
\label{sec:fast}
In this section, we show that, beyond the global nonasymptotic rates above, \cref{alg:proximal_newton} enjoys improved asymptotic rates near a solution under semismoothness of \(f'\). We assume throughout that \(x_k \to x^*\) for some local minimum \(x^*\), as ensured, e.g., by \cref{thm:global_convergence_nonconv_PL}.
Following the semismooth Newton literature, we also impose the following local nondegeneracy assumption. 
\begin{ass}\label{ass:strong_convexity}
    For sufficiently large $k$, $H(x_k) \succeq 0$. $F$ is locally $\mu$-strongly convex around $x^*$ (which is the limit of $x_k$), i.e., there exists a ball $B_R(x^*)$ such that $\forall x,y \in B_R(x^*), F'(x) \in \partial F(x)$%
    \begin{align*}
        F(y) \geq F(x) + \langle F'(x), y-x\rangle &+ \frac{\mu}{2}\norm{y-x}{}^2.
    \end{align*}
\end{ass}
Since we focus on asymptotic results, all conditions are only required to hold for sufficiently large \(k\). W.l.o.g., let \(K\) denote an iteration index from which all assumptions hold; in particular, \(x_k \in B_R(x^*)\) for all \(k \ge K\).

\subsection{Technical preliminaries}
\label{sec:further_tools}
A crucial step in establishing asymptotic convergence properties of \cref{alg:proximal_newton} is to refine the upper bound \eqref{eq:lambda_bound} on the regularization parameter $\lambda_k$ in such a way that we are able to prove that $\lambda_k\to0$ as $k\to\infty$ under the semismoothness assumption. 
Let us define, recalling the notation \eqref{eq:F_r_g_notaion}, %
\begin{align*}
&M(x_{k,+}, x_k) := f'(x_{k,+}) - f'(x_{k}) - H(x_{\pi(k)})(x_{k, +} - x_k).
\end{align*}
\rev{Our proofs are based on establishing, for some appropriate sequence $(\eta_k)\subset \R_+$, the inequality}
\begin{equation}
\label{eq:gradient_variation}
    \norm{M(x_{k,+}, x_k)}{*}\leq \eta_k \norm{x_{k,+} - x_k}{}.
\end{equation}
\rev{Under \cref{ass:basic} it is clear that it holds
with $\eta_k\equiv L$}, the same bound that is used to prove \eqref{eq:lambda_bound}. If $f'$ is additionally semismooth, then we may expect that $\eta_k\to0$ as $k\to\infty$, which will give us the refined version of \eqref{eq:lambda_bound} with $\lambda_k\to0$. Thus, in the next result we obtain a (even more general) relationship between $\eta_k$ and $\lambda_k$, \rev{see \cref{sec:lambda_refined_proof} for the proof}.

\begin{restatable}{prop}{proplambdaktozero}
\label{prop:lambda_k_to_zero} 
Let $(x_k)_{k \in \mathbb{N}}$ be the iterates of \cref{alg:proximal_newton}. Let $x_k\to x^*$, $F$ be locally convex around $x^*$, and $H(x_{k}) \succeq 0$  for $k\geq K$, and suppose that there exist $\theta \in[0,1]$ and a sequence  $(\eta_k)_{k \geq K} \subset \R_+$ which is bounded by $C\geq 0$  s.t. for $k\geq K$,   
\begin{align}\label{eq:eps_condition_1}
\norm{ M(x_{k,+}, x_k)  }{*} 
&\leq \eta_k \norm{x_k - x_{k,+}}{}^{1+\theta}.
\end{align}
Then, for $k\geq K$, we have 
\begin{align}
        \lambda_k \leq
        \begin{cases}
            4 \left(\sqrt{2}\eta_k\right)^{\frac{1}{1+\theta}} g_k^{\frac{\theta}{1+\theta}} &: j_k > 0,\\
            \frac{\lambda_{k-1}}{2} &: j_k = 0.
         \end{cases}\label{eq:lambda_k_iterative_estimate}
\end{align}
 If, in addition, we assume $p \in [1/2,1]$ and
\begin{equation}
\label{eq:Lambda_0_assumpt}
    \Lambda_{K} \leq \left(\sqrt{2}C\right)^{\frac{1}{1+\theta}} \left(\frac{1}{g_K}\right)^{p-\frac{\theta}{1+\theta}},
\end{equation}
then, for $k\geq K$,
\begin{align}
    \Lambda_{k+1} &\leq \left(\sqrt{2}C\right)^{\frac{1}{1+\theta}} \left(\frac{1}{g_k}\right)^{p-\frac{\theta}{1+\theta}} \quad \text{and} \quad   \lambda_k \leq 4 \left(\sqrt{2}C\right)^{\frac{1}{1+\theta}} g_k^{\frac{\theta}{1+\theta}}.\label{eq:lambda_k_rate_bounds}
\end{align}
\end{restatable}
Assumption \eqref{eq:Lambda_0_assumpt} is nonrestrictive since it can be guaranteed at iteration $k=0$ by a burn-in phase of logarithmic length, see \cite{doikov2024super}. Thus, for the rest of this section, we assume that it is satisfied.

We can observe from \eqref{eq:lambda_k_iterative_estimate} that if $\theta =0$ and $\eta_k\to0$ as $k\to\infty$, then also $\lambda_k\to0$ as $k\to\infty$. This will be our key step to obtain improved rates under semismoothness. At the same time, \eqref{eq:lambda_k_rate_bounds} provides a rate with which $\lambda_k$ goes to 0 depending on %
$F'(x_k)$, which will allow us to quantify the rate of superlinear convergence under $\gamma$-order semismoothness.

\subsection{Superlinear convergence under semismoothness}
We now present our superlinear convergence result for \cref{alg:proximal_newton} under the following standard semismoothness assumption; see, e.g., \cite{ulbrich2011semismooth,facchinei2003finite} and the discussion in \cref{S:Newton_derivative}. 
\begin{ass}\label{ass:semismoothness}
$f'$ is semismooth at $x^*$ with respect to $H$, i.e., as $d\to 0$,
\begin{equation}
\norm{f'(x^*+d)-f'(x^*)-H(x^*+d)d}{*} = \so(\norm{d}{}).\label{eq:semismooth_at_x*}
\end{equation}
\end{ass}
This assumption is often satisfied in ML, e.g., in learning with the Huber loss \rev{or sparsemax loss}, or training NNs with piecewise-smooth activations such as squared ReLU, Huberized ReLU, softplus with truncation.
\rev{%
We prove in \cref{sec:superlin_proof} the following improved convergence result.}
\begin{restatable}[Superlinear convergence under semismoothness]{theorem}{thmfastlocal}
\label{thm:fast_local}  %
Let \cref{ass:basic,ass:strong_convexity,ass:semismoothness} hold,  $x_k \to x^*$, and, as $k \to \infty$,  
          \begin{align}
&\norm{(H(x_{k,+}) - H(x_k))(x_{k,+} - x^*)}{*}  = \so(\norm{ x_{k,+} - x_{k}}{}),\label{eq:DM_condition}\\
 &\norm{(H(x_k)-H(x_{\pi(k)}))(x_{k,+} - x_k)}{*} = \so(\| x_{k,+} -x_k\|).
        \label{eq:lazy_DM_condition}
        \end{align}
        Then,  as $k \to \infty$, we have $\lambda_k\to0$ and, for $k$ sufficiently large,%
        \begin{align*}
         \|x_{k+1} - x^*\|  \leq  \frac{4\lambda_{k}\|x_k -x^*\|}{ \mu}   = \so(\|x_k -x^*\|), %
         \norm{ F'(x_{k+1})}{*}  \leq  \frac{4\lambda_k\norm{F'(x_k)}{*}}{\mu}  = \so(\norm{F'(x_k)}{*}).%
        \end{align*}
\end{restatable}

Conditions \eqref{eq:DM_condition} and \eqref{eq:lazy_DM_condition} are in the spirit of the celebrated Dennis--Moré (DM) condition \cite{MR343581}, which essentially gives a necessary and sufficient condition for superlinear convergence of quasi-Newton methods for smooth functions (in the nonsmooth setting, we refer to, e.g., \cite{cibulka2015inexact, MR3023752}). Similar conditions are used in previous works on SSNs \cite{potzl2022second,potzl2024inexact} to show superlinear convergence. The next lemma provides a sufficient condition for \eqref{eq:DM_condition} and \eqref{eq:lazy_DM_condition}, \rev{see \cref{sec:superlin_proof} for the proof}.

\begin{restatable}{lem}{lemlazyDMHolds}
\label{lm:lemlazyDMHolds}
        Let \cref{ass:strong_convexity} hold and $x_k \to x^*$. If $H$ is continuous, then  \eqref{eq:DM_condition} and \eqref{eq:lazy_DM_condition} hold.
\end{restatable}

Combining \cref{thm:global_convergence_nonconv_PL} and \cref{thm:fast_local}, we now have the full convergence picture of \cref{alg:proximal_newton}: global linear convergence under PL condition and acceleration to superlinear convergence under additional local regularity. \rev{The proof is given in \cref{sec:superlin_proof}.}

\begin{restatable}[Main theorem, semismoothness]{theorem}{thmglobalsuperlinear}
\label{thm:global_superlinear}
Let \cref{ass:basic} hold, $F$ satisfy the Polyak--\L{}ojasiewicz condition \eqref{eq:PL}, the sublevel set $\F_0$ be bounded, and $p < 1$. Then, 
$F(x_k)-F^* \to 0$, $x_k \to x^*$, and $\norm{F'(x_k)}{*} \to 0$ linearly as $k\to \infty$. If in addition the local  conditions of  \cref{ass:strong_convexity,ass:semismoothness}, 
\eqref{eq:DM_condition} and \eqref{eq:lazy_DM_condition} hold, then
the following convergences are all superlinear:
\begin{align*}
&x_k \to x^*, \quad F(x_k)\to F^* , \quad \norm{F'(x_k)}{*} \to 0.
\end{align*}
\end{restatable}

\subsection{Faster rate under $\gamma$-order semismoothness}
In this section, we impose a stronger regularity assumption standard in the SSN literature; cf.\ \cite{ulbrich2011semismooth,MR2421297}.%
\begin{ass}\label{ass:gamma_order_semismoothness}
 $f'$ is $\gamma$-order semismooth at $x^*$ for some $\gamma \in (0,1]$ with respect to $H$, i.e., as $d \to 0$,
\begin{equation}
\norm{f'(x^*+d)-f'(x^*)-H(x^*+d)d}{*} = \mathcal{O}(\norm{d}{}^{1+\gamma}).\label{eq:gamma_order_semismooth_at_x*}
\end{equation}
\end{ass}
Higher-order semismoothness is the key regularity behind accelerated local convergence of semismooth Newton methods, and it dovetails with our gradient regularization framework. It is satisfied by many ML primitives, including piecewise-affine maps such as the ReLU activation, Euclidean projections onto polyhedral sets, and spectral thresholding/projection operators (e.g., projection onto the PSD cone) 
\cite{sun_sun_2002,cui2021modern}. We further highlight residual maps arising in kernel optimal transport as a recent application \cite{lin_cuturi_jordan_2024} and the classical squared hinge loss in SVM.

This assumption quantifies the r.h.s. of \eqref{eq:semismooth_at_x*}, giving the rate with which it goes to 0. This stronger regularity allows us to improve the result of \cref{thm:fast_local}. Note that \cref{ass:gamma_order_semismoothness}  implies \cref{ass:semismoothness} and thus, \cref{thm:fast_local} is still in force. \rev{We prove the next results in \cref{sec:superlin_gamma_proof}.}

\begin{restatable}[Superlinear convergence under $\gamma$-order semismoothness, $\lzsymb=1$]{theorem}{thmfastlocalgamma}
\label{thm:fast_local_gamma}  %
Let \cref{ass:basic,ass:strong_convexity,ass:gamma_order_semismoothness} hold, $x_k \to x^*$, $\lzsymb=1$ (i.e., $\pi(k)=k$), $p \in [1/2,1]$, and as $k \to \infty$,
\begin{equation}
\norm{(H(x_{k,+}) - H(x_k))(x_{k,+} - x^*)}{*}
= \mathcal{O}\!\left(\norm{x_{k,+} - x_k}{}^{1+\gamma}\right).
\label{eq:gamma_order_DM_condition}
\end{equation}
Then as $k \to \infty$, we have $\lambda_k = \mathcal{O}(g_k^{\frac{\gamma}{1+\gamma}})$ and, for $k$ sufficiently large, 
\begin{align}
\|x_{k+1} - x^*\|  = \mathcal{O}(\norm{x_{k-1}-x^*}{}^{\frac{1+2\gamma}{1+\gamma}}), \;\;
           \norm{ F'(x_{k+1})}{*}    = \mathcal{O}(\norm{F'(x_k)}{*}^{\frac{1+2\gamma}{1+\gamma}}). \label{eq:gamma_two_step_rate}   %
\end{align}  
\end{restatable}

 \begin{remark}\label{rem:improvement}
     If $\psi \equiv 0$, we can improve \eqref{eq:gamma_two_step_rate} by simply using $g_k \leq (L\slash 2)\norm{x_k-x^*}{}$ to achieve
     \(    \|x_{k+1} - x^*\| = \mathcal{O}(\norm{x_{k}-x^*}{}^{\frac{1+2\gamma}{1+\gamma}}).\)
 \end{remark}

\begin{restatable}{lem}{lemctwoimpliesdm}
 \label{lem:c2_implies_dm}
Let \cref{ass:strong_convexity} hold and $x_k \to x^*$. If $H$ is H\"older continuous with exponent $\gamma$ at $x^*$, 
then  \eqref{eq:gamma_order_DM_condition} holds. %

\end{restatable}

\begin{restatable}[Main theorem, $\gamma$-order semismoothness]{theorem}{thmglobalsuperlineargamma}
\label{thm:global_superlinear_gamma}
Let \cref{ass:basic} hold, $F$ satisfy the Polyak--\L{}ojasiewicz condition \eqref{eq:PL}, the sublevel set $\F_0$ be bounded, and $p < 1$. Then, 
$F(x_k)-F^* \to 0$, $x_k \to x^*$, and $\norm{F'(x_k)}{*} \to 0$ linearly as $k\to \infty$. If in addition the local conditions of  \cref{ass:strong_convexity,ass:gamma_order_semismoothness}, 
\eqref{eq:gamma_order_DM_condition} hold, $m=1$ and $p \in [1/2,1)$, then the convergences
\begin{align*}
&x_k \to x^*, \quad F(x_k)\to F^* %
\end{align*}
are all superlinear, the convergence of $\norm{F'(x_k)}{*}$ is $\frac{1+2\gamma}{1+\gamma}$-order superlinear, and the convergence of $x_k$ is two-step $\frac{1+2\gamma}{1+\gamma}$-order superlinear.
\end{restatable}

\subsection{Faster rates in the finite-dimensional \texorpdfstring{$C^2$}{C\textasciicircum 2} setting}  \label{sec:c2_setting}  
Finally, we consider the finite-dimensional locally \(C^2\) case, where the local strong convexity assumption in \cref{ass:strong_convexity} can be weakened to local convexity. While \(C^2\)-regularity is standard in Newton methods, existing results do not yield improved asymptotic rates. Here, local \(C^2\)-smoothness implies \(\eta_k \to 0\) in \eqref{eq:gradient_variation}, hence \(\lambda_k \to 0\) by \rev{\cref{prop:lambda_k_to_zero}}, which together with \cref{thm:global_convergence_nonconv_PL} yields superlinear convergence. \rev{The formal statement, proved in \cref{sec:superlin_C_2_proof}, is as follows.}
\begin{restatable}[Main theorem, finite dimensions, $C^2$]{theorem}{thmglobalconvergencectwo}
\label{thm:global_convergence_c2} 
Let $\dim(\mathcal{H}) < \infty$, \cref{ass:basic} hold, $F$ satisfy the Polyak--\L{}ojasiewicz condition \eqref{eq:PL}, $p<1$, and the sublevel set $\F_0$ be bounded. Then, $F(x_k)-F^* \to 0$, $x_k \to x^*$, and $\norm{F'(x_k)}{*} \to 0$ linearly as $k\to \infty$.
If additionally $f$ is locally $C^2$ and locally convex around $x^*$, and $H \coloneqq  \nabla^2 f$, then the above convergences
are all superlinear.
\end{restatable}

 \section{Experiments}    
We consider three examples from focus topics in ML and demonstrate the efficacy of our method.  %
\rev{Implementation and computational details as well as relevant semismoothness properties of our examples are deferred to \cref{sec:app_experiments}. In each figure where the runtime is reported, this refers to the wall-clock time.} \rev{The code is available at \url{https://github.com/clemens-sirotenko/skip-the-hessian-keep-the-rates}.} 
\subsection{Neural networks with Lipschitz constraints.}\label{sec:numerics_NN}
Motivated by inverse imaging tasks such as plug-and-play (PnP) reconstruction and the use of adversarial networks, there is growing interest in training neural
networks whose realized input-to-output map \(x \mapsto \Phi(x,\theta)\) has a controlled Lipschitz
constant; see \cite{ducotterd2024improving} and \cite{hertrich2025learning}. We consider a prototypical regression problem under a soft \(1\)-Lipschitz constraint of the gradient $\nabla_x \Phi(\cdot,\theta)$. Given training
pairs \((x_i,y_i)_{i=1}^M \subset \mathbb{R}^n \times \mathbb{R}^d\), we solve for $\lambda>0$
\begin{equation}\label{eq:Lipschitz_objective}
\begin{aligned}
\min_{\theta}\;&
\frac{1}{M}\sum_{i=1}^M \|\Phi(x_i,\theta)-y_i\|^2  +
\frac{\lambda}{2M}\sum_{i=1}^M
\Big((\|\nabla_x^2\Phi(x_i,\theta)\|_{2}^2-1)^+\Big)^2,  
\end{aligned}
\end{equation}
where \(\nabla^2_x\Phi(x_i,\theta)\) denotes the Hessian of \(\Phi(\cdot,\theta)\) with respect to the input, and \(\|\cdot\|_2\) is the spectral norm. \cref{fig:NN} shows the results of solving \eqref{eq:Lipschitz_objective} with \gladssn, with and without laziness, and a comparison to \adam, the standard solver in learning. We see that our algorithm performs excellently in iteration count, and laziness saves considerable time. This clearly demonstrates the advantage lazy Hessian updates can have.  \rev{\cref{fig:NN_lip} highlights the iteration efficiency of \gladssn, which brings the sampled Lipschitz estimate close to the target value of one in substantially fewer iterations than \adam.}

\begin{figure*}
\centering
 \includegraphics[scale=0.5]{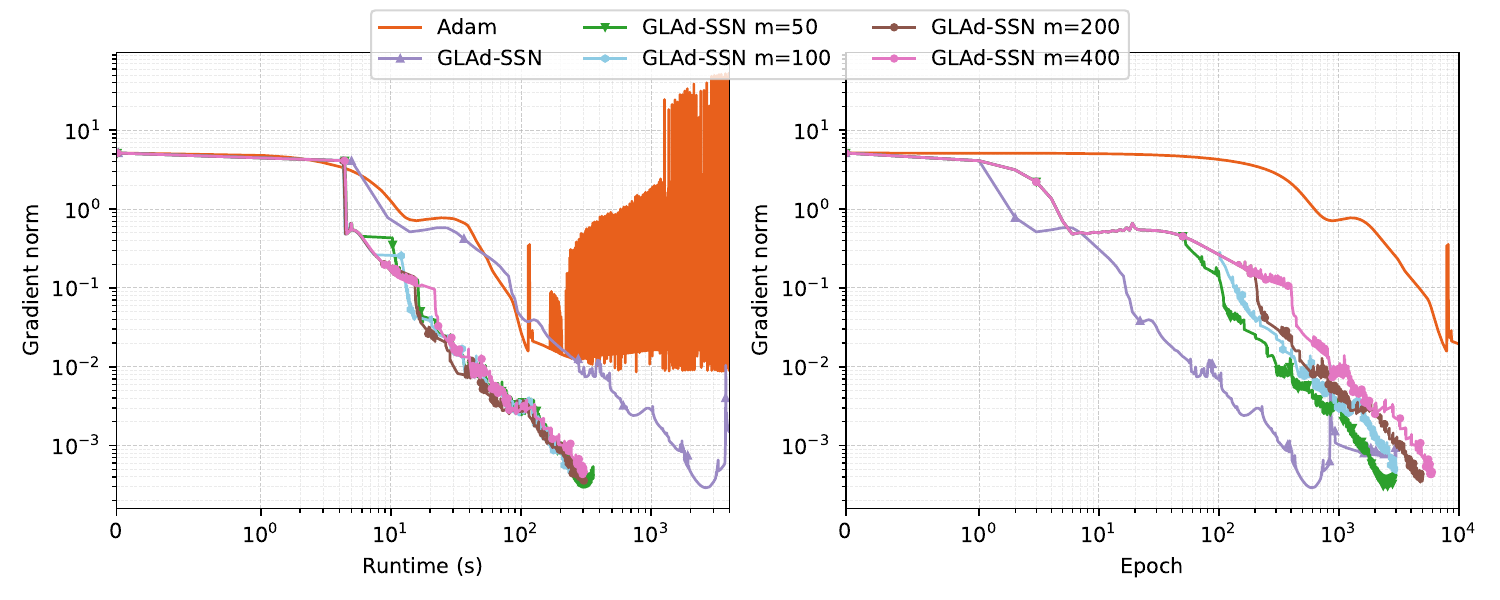} 
\caption{Training neural networks with Lipschitz constraints (\cref{sec:numerics_NN}). \gladssn performs an order of magnitude better than \adam in iteration count. Laziness (\rev{increasing $m$}) helps significantly! %
}
\label{fig:NN}
\end{figure*}

\subsection{Non-negative matrix factorization}\label{sec:numericsNMF}
Non-negative matrix factorization (NMF) is a standard tool in signal processing and machine learning, see  \citep{cichocki2007nonnegative, gillis2020nonnegative}.
Given a data matrix $Y\in\mathbb{R}^{n\times d}$ whose columns are (noisy) measurements, NMF seeks nonnegative factors $U\in\mathbb{R}^{n\times r}_+$ (basis elements) and
$V\in\mathbb{R}^{r\times d}_+$ (coefficients) such that $UV \approx Y$, leading to the problem
    $\min_{U,V} \frac{1}{2}\|UV - Y \|^2 + \rev{\frac{\alpha}{2}}(\|U\|^2 + \|V\|^2)$ s.t. $U,V \geq 0$. %
We use Moreau--Yosida regularization and consider
\begin{equation}\label{eq:mat_fac_2}
\begin{aligned}
    \min_{U,V} %
    \frac{1}{2}\|UV - Y \|^2 + \rev{\frac{\alpha
    }{2}}(\|U\|^2 + \|V\|^2)  + \rev{\frac{\beta}{2}}(\|(-U)^+\|^2  + \|(-V)^+\|^2)
\end{aligned}
\end{equation}
for a sufficiently \rev{large} penalty parameter $\beta>0$.
\rev{ \cref{fig:MF_grad} compares \gladssn for different gradient exponents $p$ with \leapssn \cite{LeapSSN}, gradient descent with Armijo backtracking \cite{wright2006numerical}, and a cubic-regularized Newton baseline \cite{nesterov2006cubic} with adaptive backtracking. %
\gladssn substantially outperforms gradient descent in reducing the gradient norm and achieves runtimes comparable to cubic Newton, while our analysis provides global and local convergence guarantees in the semismooth setting.}

In \cref{fig:MF_penalty} we plot $\rev{\frac{\beta}{2}}(\|(-U)^+\|^2  + \|(-V)^+\|^2)$, which measures the violation of the non-negativity constraint. \rev{In \cref{sec:NMF_appendix}, we also show that \gladssn outperforms PALM \cite{bolte2014proximal} %
and compare the full generalized Hessian with its positive definite Gauss--Newton approximation.}

\subsection{Support vector classification}\label{sec:numerics_SVM}
We consider a support vector machine classification problem (also studied in \cite{LeapSSN}). Given points $(x_i)_{i=1}^\ell \subset \mathbb{R}^n$  where each point is classified into one of two groups denoted by $(y_i)_{i \in 1}^\ell \subset \{-1,1\}$, the aim is to find $\omega \in \mathbb{R}^n$ and $b \in \mathbb{R}$ such that the hyperplane $\omega^\top x + b =0$ separates the two groups. Given a $\gamma>0$, we consider (as in \cite{yin2019}), the L2-loss SVM model: 
\begin{align*}
\min_{(\omega, b) \in \mathbb{R}^n \times \mathbb{R}}  F(\omega,b) = \frac{1}{2} \|\omega\|^2_{\ell^2} + \gamma \sum_{i=1}^\ell \max(1 - y_i (\omega^\top x_i + b), 0)^2 .
\end{align*}
We again run \gladssn with and without laziness and compare it to \leapssn, gradient descent with Armijo linesearch, cubic Newton and SVRG-BB \cite{SVRGBB}, which is used in \cite{yin2019} as a baseline. As  can be seen in \cref{fig:SVM},   lazy updates can save compute time while possessing strong error guarantees.

\begin{figure}[!htbp]
\centering
\begin{minipage}{0.48\textwidth}
  \centering
  \includegraphics[width=\textwidth]{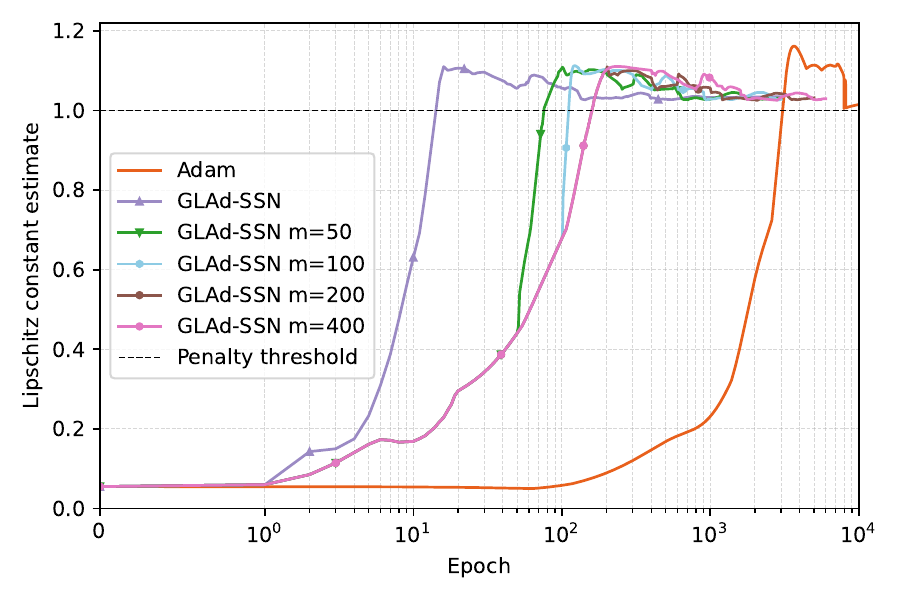}
  \caption{Estimated Lipschitz constant of the neural network \rev{input gradient} (\cref{sec:numerics_NN}).}
  \label{fig:NN_lip}
\end{minipage}\hfill
\begin{minipage}{0.49\textwidth}
  \centering
 \includegraphics[width=\textwidth]{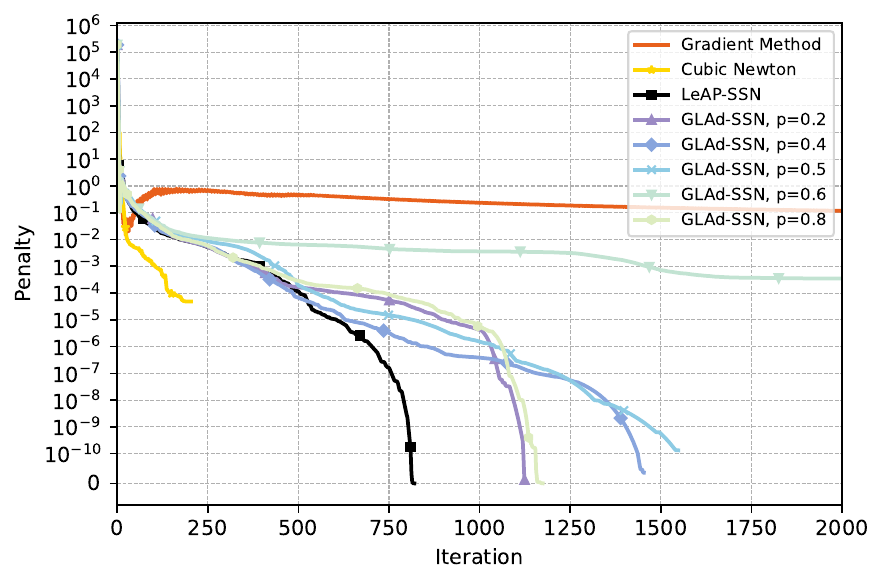}
  \caption{\rev{Non-negativity penalty for the representative run shown in \cref{fig:MF_grad}; smaller is better.}}
  \label{fig:MF_penalty}
\end{minipage}
\end{figure}

\begin{figure}[!htbp]
\centering
\includegraphics[width=1\textwidth]{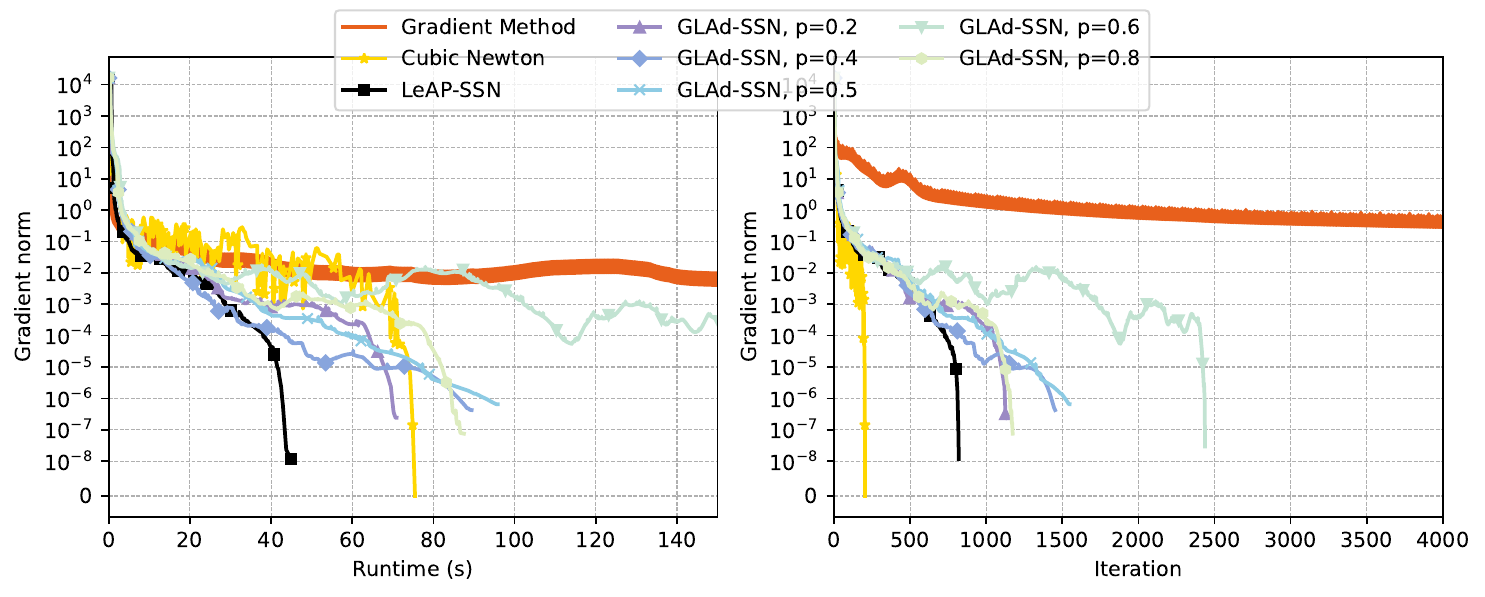}

\caption{Non-negative matrix factorization (\cref{sec:numericsNMF}): \leapssn and \gladssn for $p=0.2$ both beat the Cubic Newton in runtime; the gradient method is fast but inaccurate.}%
\label{fig:MF_grad}
\end{figure}

\begin{figure}[!htbp]
\centering
\includegraphics[width=0.9\textwidth]{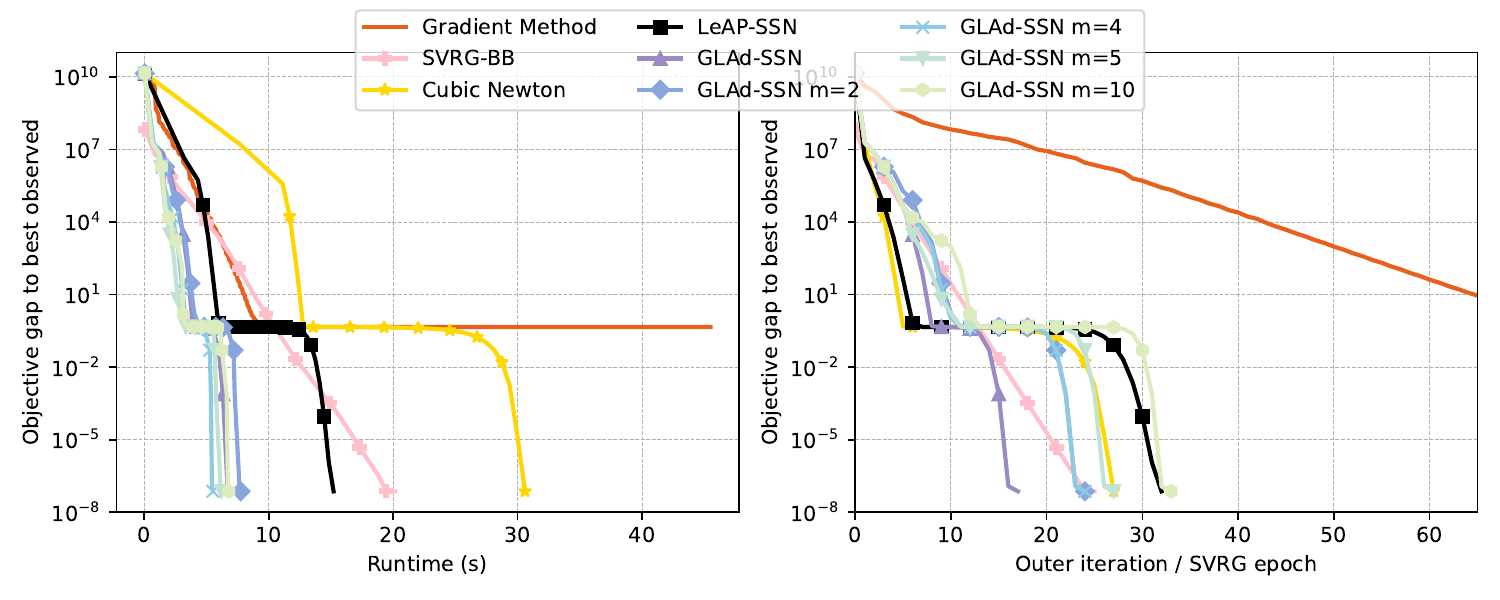}
\caption{SVM classification (\cref{sec:numerics_SVM}): \gladssn and \leapssn behave comparably and both outperform the gradient method, Cubic Newton, SVRG-BB.}
\label{fig:SVM}
\end{figure}

\clearpage %
\begin{ack}
\rev{The authors thank the referees for their careful reading and useful feedback.} PD is thankful for support from the Deutsche 
Forschungsgemeinschaft (DFG, German Research Foundation) under Germany’s Excellence Strategy – The Berlin Mathematics Research Center MATH+ (EXC-2046/1, EXC-2046/2, project ID: 390685689).

There are no competing interests to declare.
\end{ack}
\bibliographystyle{abbrvnat}
\bibliography{main}

@article{Shamanskii1967AMO,
  title={A modification of Newton's method},
  author={V. E. Shamanskii},
  journal={Ukrainian Mathematical Journal},
  year={1967},
  volume={19},
  pages={118-122},
  url={https://api.semanticscholar.org/CorpusID:122185269}
}

@book {Mordukhovich1,
    AUTHOR = {Mordukhovich, Boris S.},
     TITLE = {Variational analysis and generalized differentiation. {I}},
    SERIES = {Grundlehren der mathematischen Wissenschaften [Fundamental
              Principles of Mathematical Sciences]},
    VOLUME = {330},
      NOTE = {Basic theory},
 PUBLISHER = {Springer-Verlag, Berlin},
      YEAR = {2006},
     PAGES = {xxii+579},
      ISBN1 = {978-3-540-25437-9; 3-540-25437-4},
   MRCLASS = {49-02 (49J52 49J53 49K27 49K40 90C31)},
  MRNUMBER = {2191744},
MRREVIEWER = {Lionel Thibault},
}

@inproceedings{doikov2023second,
  title={{Second-order optimization with lazy Hessians}},
  author={Doikov, Nikita and Jaggi, Martin and others},
  booktitle={International Conference on Machine Learning},
  pages={8138--8161},
  year={2023},
  organization={PMLR}
}

@article {MR3023752,
    AUTHOR = {Dontchev, Asen L.},
     TITLE = {Generalizations of the {D}ennis-{M}or\'e{} theorem},
   JOURNAL = {SIAM J. Optim.},
  FJOURNAL = {SIAM Journal on Optimization},
    VOLUME = {22},
      YEAR = {2012},
    NUMBER = {3},
     PAGES = {821--830},
      ISSN = {1052-6234,1095-7189},
   MRCLASS = {49J53 (49M37 65J15 90C30)},
  MRNUMBER = {3023752},
MRREVIEWER = {G\'erard\ Lebourg},
       DOI = {10.1137/110833567},
       URL = {https://doi.org/10.1137/110833567},
}

@article {MR343581,
    AUTHOR = {Dennis, Jr., J. E. and Mor\'e, Jorge J.},
     TITLE = {A characterization of superlinear convergence and its
              application to quasi-{N}ewton methods},
   JOURNAL = {Math. Comp.},
  FJOURNAL = {Mathematics of Computation},
    VOLUME = {28},
      YEAR = {1974},
     PAGES = {549--560},
      ISSN = {0025-5718,1088-6842},
   MRCLASS = {65H05},
  MRNUMBER = {343581},
MRREVIEWER = {S.\ F.\ McCormick},
       DOI = {10.2307/2005926},
       URL = {https://doi.org/10.2307/2005926},
}

@article{chen2025computationally,
  title={{Computationally Faster Newton Methods by Lazy Evaluations}},
  author={Chen, Lesi and Liu, Chengchang and Luo, Luo and Zhang, Jingzhao},
  journal={arXiv preprint arXiv:2501.17488},
  year={2025}
}

@article{RickmannHerzogHeberg, 
title={{Global Convergence of Semismooth Newton Methods for Quadratic Problems}}, volume={7}, url={https://www.bibliothek.tu-chemnitz.de/ojs/index.php/GAMMAS/article/view/810}, DOI={10.14464/gammas.v7i1.810}, abstractNote={&amp;lt;div&amp;gt; &amp;lt;div&amp;gt;This paper investigates the semismooth Newton method and its application to solving constrained quadratic optimization problems.&amp;lt;/div&amp;gt; &amp;lt;div&amp;gt;We begin by reviewing various generalized concepts of differentiability, such as Clarke&amp;#039;s generalized Jacobian, slanting functions, and Newton differentiability, providing a comparative analysis to clarify their differences.&amp;lt;/div&amp;gt; &amp;lt;div&amp;gt;We then establish the equivalence between the semismooth Newton method (SSN) and the primal-dual active set algorithm (PDASA), and key global convergence results are summarized.&amp;lt;/div&amp;gt; &amp;lt;div&amp;gt;Our study focuses on the cycle behavior of small-dimensional quadratic problems.&amp;lt;/div&amp;gt; &amp;lt;div&amp;gt;For the case $n=2$, we prove global convergence for arbitrary quadratic problems, demonstrating the robustness of the method in this setting.&amp;lt;/div&amp;gt; &amp;lt;div&amp;gt;A necessary condition for cycles of certain lengths is derived and used to identify possible cycling patterns, when $n=3$.&amp;lt;/div&amp;gt; &amp;lt;div&amp;gt;While only two of the cycling patterns were observed in randomly generated examples, the other remain a theoretical possibility, suggesting further exploration of the method&amp;#039;s behavior.&amp;lt;/div&amp;gt; &amp;lt;/div&amp;gt;}, number={1}, journal={GAMM Archive for Students}, author={Rickmann, Hannah and Herzog, Roland and Herberg, Evelyn}, year={2025}, month1={May}, pages={16} }

@article{potzl2022second,
  title={{Second order semi-smooth Proximal Newton methods in Hilbert spaces}},
  author={P{\"o}tzl, Bastian and Schiela, Anton and Jaap, Patrick},
  journal={Computational Optimization and Applications},
  volume={82},
  number={2},
  pages={465--498},
  year={2022},
  publisher={Springer}
}

@article{nesterov2006cubic,
  title={{Cubic regularization of Newton method and its global performance}},
  author={Nesterov, Yurii and Polyak, Boris T},
  journal={Mathematical programming},
  volume={108},
  number={1},
  pages={177--205},
  year={2006},
  publisher={Springer}
}

@article{doikov2024super,
  title={{Super-universal regularized Newton method}},
  author={Doikov, Nikita and Mishchenko, Konstantin and Nesterov, Yurii},
  journal={SIAM Journal on Optimization},
  volume={34},
  number={1},
  pages={27--56},
  year={2024},
  publisher={SIAM}
}

@article {MR1972217,
    AUTHOR = {Ulbrich, Michael},
     TITLE = {Semismooth {N}ewton methods for operator equations in function
              spaces},
   JOURNAL = {SIAM J. Optim.},
  FJOURNAL = {SIAM Journal on Optimization},
    VOLUME = {13},
      YEAR = {2002},
    NUMBER = {3},
     PAGES = {805--842},
      ISSN = {1052-6234,1095-7189},
   MRCLASS = {49M15 (47J25 49J52 90C30)},
  MRNUMBER = {1972217},
MRREVIEWER = {Vaithilingam\ Jeyakumar},
       DOI = {10.1137/S1052623400371569},
       URL = {https://doi.org/10.1137/S1052623400371569},
}

@article {MR1972649,
    AUTHOR = {Ito, Kazufumi and Kunisch, Karl},
     TITLE = {Semi-smooth {N}ewton methods for variational inequalities of
              the first kind},
   JOURNAL = {M2AN Math. Model. Numer. Anal.},
  FJOURNAL = {M2AN. Mathematical Modelling and Numerical Analysis},
    VOLUME = {37},
      YEAR = {2003},
    NUMBER = {1},
     PAGES = {41--62},
      ISSN = {0764-583X,1290-3841},
   MRCLASS = {49J40 (65K10 90C53)},
  MRNUMBER = {1972649},
       DOI = {10.1051/m2an:2003021},
       URL = {https://doi.org/10.1051/m2an:2003021},
}

@article {MR1786137,
    AUTHOR = {Chen, Xiaojun and Nashed, Zuhair and Qi, Liqun},
     TITLE = {Smoothing methods and semismooth methods for nondifferentiable
              operator equations},
   JOURNAL = {SIAM J. Numer. Anal.},
  FJOURNAL = {SIAM Journal on Numerical Analysis},
    VOLUME = {38},
      YEAR = {2000},
    NUMBER = {4},
     PAGES = {1200--1216},
      ISSN = {0036-1429,1095-7170},
   MRCLASS = {65H10 (47J25 49J52)},
  MRNUMBER = {1786137},
MRREVIEWER = {Zhen\ Yu\ Huang},
       DOI = {10.1137/S0036142999356719},
       URL = {https://doi.org/10.1137/S0036142999356719},
}

@article {MR1216791,
    AUTHOR = {Qi, Li Qun and Sun, Jie},
     TITLE = {A nonsmooth version of {N}ewton's method},
   JOURNAL = {Math. Programming},
  FJOURNAL = {Mathematical Programming},
    VOLUME = {58},
      YEAR = {1993},
    NUMBER = {3},
     PAGES = {353--367},
      ISSN = {0025-5610,1436-4646},
   MRCLASS = {90C30 (49J52 65H05)},
  MRNUMBER = {1216791},
MRREVIEWER = {Aharon\ Ben-Tal},
       DOI = {10.1007/BF01581275},
       URL = {https://doi.org/10.1007/BF01581275},
}

@article {MR1250115,
    AUTHOR = {Qi, Li Qun},
     TITLE = {Convergence analysis of some algorithms for solving nonsmooth
              equations},
   JOURNAL = {Math. Oper. Res.},
  FJOURNAL = {Mathematics of Operations Research},
    VOLUME = {18},
      YEAR = {1993},
    NUMBER = {1},
     PAGES = {227--244},
      ISSN = {0364-765X,1526-5471},
   MRCLASS = {65H10 (49J52 90C30)},
  MRNUMBER = {1250115},
       DOI = {10.1287/moor.18.1.227},
       URL = {https://doi.org/10.1287/moor.18.1.227},
}

@article{wachsmuth2025globalized,
  title={{A globalized inexact semismooth Newton for strongly convex optimal control problems}},
  author={Wachsmuth, Daniel},
  journal={arXiv preprint arXiv:2503.21612},
  year={2025}
}

@article {hintermuller2002primal,
    AUTHOR = {Hinterm\"uller, M. and Ito, K. and Kunisch, K.},
     TITLE = {The primal-dual active set strategy as a semismooth {N}ewton
              method},
   JOURNAL = {SIAM J. Optim.},
  FJOURNAL = {SIAM Journal on Optimization},
    VOLUME = {13},
      YEAR = {2002},
    NUMBER = {3},
     PAGES = {865--888},
      ISSN = {1052-6234,1095-7189},
   MRCLASS = {90C33 (65K10 90C53)},
  MRNUMBER = {1972219},
MRREVIEWER = {Hou\ Duo\ Qi},
       DOI = {10.1137/S1052623401383558},
       URL = {https://doi.org/10.1137/S1052623401383558},
}

@book {ulbrich2011semismooth,
    AUTHOR = {Ulbrich, Michael},
     TITLE = {Semismooth {N}ewton methods for variational inequalities and
              constrained optimization problems in function spaces},
    SERIES = {MOS-SIAM Series on Optimization},
    VOLUME = {11},
 PUBLISHER = {Society for Industrial and Applied Mathematics (SIAM),
              Philadelphia, PA; Mathematical Optimization Society,
              Philadelphia, PA},
      YEAR = {2011},
     PAGES = {xiv+308},
      ISBN = {978-1-611970-68-5},
   MRCLASS = {49-02 (47J20 47J25 49J40 49M15 65-02 65K15)},
  MRNUMBER = {2839219},
MRREVIEWER = {Michael\ Hinterm\"uller},
       DOI = {10.1137/1.9781611970692},
       URL = {https://doi.org/10.1137/1.9781611970692},
}

@book{facchinei2003finite,
  title={Finite-dimensional variational inequalities and complementarity problems},
  author={Facchinei, Francisco and Pang, Jong-Shi},
  publisher={Springer},
    year={2003}
}

@inproceedings{lin2024specialized,
  title={{A specialized semismooth Newton method for kernel-based optimal transport}},
  author={Lin, Tianyi and Cuturi, Marco and Jordan, Michael},
  booktitle={International Conference on Artificial Intelligence and Statistics},
  pages={145--153},
  year={2024},
  organization={PMLR}
}

@inproceedings{karimi2016linear,
  title={Linear convergence of gradient and proximal-gradient methods under the polyak-{\l}ojasiewicz condition},
  author={Karimi, Hamed and Nutini, Julie and Schmidt, Mark},
  booktitle={Joint European conference on machine learning and knowledge discovery in databases},
  pages={795--811},
  year={2016},
  organization={Springer}
}

@article{potzl2024inexact,
  title={{Inexact proximal Newton methods in Hilbert spaces}},
  author={P{\"o}tzl, Bastian and Schiela, Anton and Jaap, Patrick},
  journal={Computational Optimization and Applications},
  volume={87},
  number={1},
  pages={1--37},
  year={2024},
  publisher={Springer}
}

@inproceedings{SVRGBB,
author = {Tan, Conghui and Ma, Shiqian and Dai, Yu-Hong and Qian, Yuqiu},
title = {Barzilai-Borwein step size for stochastic gradient descent},
year = {2016},
isbn = {9781510838819},
publisher = {Curran Associates Inc.},
address = {Red Hook, NY, USA},
booktitle = {Proceedings of the 30th International Conference on Neural Information Processing Systems},
pages = {685–693},
numpages = {9},
location = {Barcelona, Spain},
series = {NIPS'16}
}

@book {wright2006numerical,
    AUTHOR = {Nocedal, Jorge and Wright, Stephen J.},
     TITLE = {Numerical optimization},
    SERIES = {Springer Series in Operations Research and Financial
              Engineering},
   EDITION = {Second},
 PUBLISHER = {Springer, New York},
      YEAR = {2006},
     PAGES = {xxii+664},
      ISBN = {978-0387-30303-1; 0-387-30303-0},
   MRCLASS = {90-01 (49Mxx 65K05 90-02 90C30)},
  MRNUMBER = {2244940},
}

@article{khanh2024globally,
  title={{Globally convergent coderivative-based generalized Newton methods in nonsmooth optimization}},
  author={Khanh, Pham Duy and Mordukhovich, Boris S and Phat, Vo Thanh and Tran, Dat Ba},
  journal={Mathematical Programming},
  volume={205},
  number={1},
  pages={373--429},
  year={2024},
  publisher={Springer}
}

@book{mordukhovich2024second,
  title={Second-order variational analysis in optimization, variational stability, and control: theory, algorithms, applications},
  author={Mordukhovich, Boris S},
  year={2024},
  publisher={Springer Nature}
}

@article{pang1995globally,
  title={{A globally convergent Newton method for convex SC1 minimization problems}},
  author={Pang, Jong-Shi and Qi, Liqun},
  journal={Journal of Optimization Theory and Applications},
  volume={85},
  number={3},
  pages={633--648},
  year={1995},
  publisher={Springer}
}

@book{cui2021modern,
  title={Modern nonconvex nondifferentiable optimization},
  author={Cui, Ying and Pang, Jong-Shi},
  year={2021},
  publisher={SIAM}
}

@article{cibulka2015inexact,
  title={{Inexact Newton Methods and Dennis--More Theorems for Nonsmooth Generalized Equations}},
  author={Cibulka, Radek and Dontchev, A and Geoffroy, Michel H},
  journal={SIAM Journal on Control and Optimization},
  volume={53},
  number={2},
  pages={1003--1019},
  year={2015},
  publisher={SIAM}
}

@article{bolte2014proximal,
  title={Proximal alternating linearized minimization for nonconvex and nonsmooth problems},
  author={Bolte, J{\'e}r{\^o}me and Sabach, Shoham and Teboulle, Marc},
  journal={Mathematical Programming},
  volume={146},
  number={1},
  pages={459--494},
  year={2014},
  publisher={Springer}
}

@book {RockafellarWets,
    AUTHOR = {Rockafellar, R. Tyrrell and Wets, Roger J.-B.},
     TITLE = {Variational analysis},
    SERIES = {Grundlehren der mathematischen Wissenschaften [Fundamental
              Principles of Mathematical Sciences]},
    VOLUME = {317},
 PUBLISHER = {Springer-Verlag, Berlin},
      YEAR = {1998},
     PAGES = {xiv+733},
      ISBN = {3-540-62772-3},
   MRCLASS = {49-02 (46N10 47N10 49J52 49K40 90C30)},
  MRNUMBER = {1491362},
MRREVIEWER = {Francis\ H.\ Clarke},
       DOI = {10.1007/978-3-642-02431-3},
       URL = {https://doi.org/10.1007/978-3-642-02431-3},
}

@article{yin2019,
  title={A semismooth {N}ewton method for support vector classification and regression},
  author={Yin, Juan and Li, Qingna},
  journal={Computational Optimization and Applications},
  volume={73},
  number={2},
  pages={477--508},
  year={2019},
  publisher={Springer},
  doi={10.1007/s10589-019-00075-z}
}

@misc{super-newton,
author={Doikov, Nikita},
title={{super-newton}},
year={2022},
url={https://github.com/doikov/super-newton},
}

@article{scikit-learn,
  title={Scikit-learn: Machine Learning in {P}ython},
  author={Pedregosa, F. and Varoquaux, G. and Gramfort, A. and Michel, V.
          and Thirion, B. and Grisel, O. and Blondel, M. and Prettenhofer, P.
          and Weiss, R. and Dubourg, V. and Vanderplas, J. and Passos, A. and
          Cournapeau, D. and Brucher, M. and Perrot, M. and Duchesnay, E.},
  journal={Journal of Machine Learning Research},
  volume={12},
  pages={2825--2830},
  year={2011}
}

@article{cartis2011adaptive,
  title={{Adaptive cubic regularisation methods for unconstrained optimization. Part I: motivation, convergence and numerical results}},
  author={Cartis, Coralia and Gould, Nicholas IM and Toint, Philippe L},
  journal={Mathematical Programming},
  volume={127},
  number={2},
  pages={245--295},
  year={2011},
  publisher={Springer}
}

@Article{nesterov2008accelerating,
author="Nesterov, Yu.",
title={{Accelerating the cubic regularization of Newton's method on convex problems}},
journal="Mathematical Programming",
year="2008",
month1="Mar",
day1="01",
volume="112",
number="1",
pages="159--181",
issn="1436-4646",
doi="10.1007/s10107-006-0089-x",
url="https://doi.org/10.1007/s10107-006-0089-x"
}

@article{monteiro2013accelerated,
author = {Monteiro, R. and Svaiter, B.},
title = {An Accelerated Hybrid Proximal Extragradient Method for Convex Optimization and Its Implications to Second-Order Methods},
journal = {SIAM Journal on Optimization},
volume = {23},
number = {2},
pages = {1092-1125},
year = {2013},
doi = {10.1137/110833786},
URL = {         https://doi.org/10.1137/110833786},
eprint = {         https://doi.org/10.1137/110833786}
}

@inproceedings{carmon2022optimal,
author = {Carmon, Yair and Hausler, Danielle and Jambulapati, Arun and Jin, Yujia and Sidford, Aaron},
title = {{Optimal and adaptive Monteiro--Svaiter acceleration}},
year = {2022},
isbn = {9781713871088},
publisher = {Curran Associates Inc.},
address = {Red Hook, NY, USA},
booktitle = {Proceedings of the 36th International Conference on Neural Information Processing Systems},
articleno = {1479},
numpages = {13},
location = {New Orleans, LA, USA},
series = {NIPS '22}
}

@inproceedings{kovalev2022first,
author = {Kovalev, Dmitry and Gasnikov, Alexander},
title = {The first optimal acceleration of high-order methods in smooth convex optimization},
year = {2022},
isbn = {9781713871088},
publisher = {Curran Associates Inc.},
address = {Red Hook, NY, USA},
booktitle = {Proceedings of the 36th International Conference on Neural Information Processing Systems},
articleno = {2561},
numpages = {13},
location = {New Orleans, LA, USA},
series = {NIPS '22}
}

@InProceedings{gasnikov2019near,
  title = 	 {Near Optimal Methods for Minimizing Convex Functions with Lipschitz $p$-th Derivatives},
  author = 	 {Gasnikov, Alexander and Dvurechensky, Pavel and Gorbunov, Eduard and Vorontsova, Evgeniya and Selikhanovych, Daniil and Uribe, C\'esar A. and Jiang, Bo and Wang, Haoyue and Zhang, Shuzhong and Bubeck, S\'ebastien and Jiang, Qijia and Lee, Yin Tat and Li, Yuanzhi and Sidford, Aaron},
  booktitle = 	 {Proceedings of the Thirty-Second Conference on Learning Theory},
  pages = 	 {1392--1393},
  year = 	 {2019},
  editor = 	 {Beygelzimer, Alina and Hsu, Daniel},
  volume = 	 {99},
  series = 	 {Proceedings of Machine Learning Research},
  address = 	 {Phoenix, USA},
  month1 = 	 {25--28 Jun},
  publisher = 	 {PMLR},
  url = 	 {http://proceedings.mlr.press/v99/gasnikov19b.html},
  note = {arXiv:1809.00382}
}

@article{grapiglia2019accelerated,
author = {Grapiglia, Geovani N. and Nesterov, Yurii},
title = {{Accelerated Regularized Newton Methods for Minimizing Composite Convex Functions}},
journal = {SIAM Journal on Optimization},
volume = {29},
number = {1},
pages = {77-99},
year = {2019},
doi = {10.1137/17M1142077},

URL = { 
    
        https://doi.org/10.1137/17M1142077
    
    

},
eprint = { 
    
        https://doi.org/10.1137/17M1142077
    
    

}
}

@article{grapiglia2017regularized,
author = {Grapiglia, G. N. and Nesterov, Yu.},
title = {{Regularized Newton Methods for Minimizing Functions with Hölder Continuous Hessians}},
journal = {SIAM Journal on Optimization},
volume = {27},
number = {1},
pages = {478-506},
year = {2017},
doi = {10.1137/16M1087801},

URL = { 
    
        https://doi.org/10.1137/16M1087801
    
    

},
eprint = { 
    
        https://doi.org/10.1137/16M1087801
    
    

}
}

@article{mishchenko2023regularized,
author = {Mishchenko, Konstantin},
title = {{Regularized Newton Method with Global \({\boldsymbol{\mathcal{O}(1/{k}^2)}}\) Convergence}},
journal = {SIAM Journal on Optimization},
volume = {33},
number = {3},
pages = {1440-1462},
year = {2023},
doi = {10.1137/22M1488752},
URL = {https://doi.org/10.1137/22M1488752},
eprint = {https://doi.org/10.1137/22M1488752}
}

@Article{doikov2024gradient,
author={Doikov, Nikita
and Nesterov, Yurii},
title={Gradient regularization of Newton method with Bregman distances},
journal={Mathematical Programming},
year={2024},
month1={Mar},
day1={01},
volume={204},
number={1},
pages={1-25},
issn={1436-4646},
doi={10.1007/s10107-023-01943-7},
url={https://doi.org/10.1007/s10107-023-01943-7}
}

@article{ouyang2025trust,
  title={{A trust region-type normal map-based semismooth Newton method for nonsmooth nonconvex composite optimization}},
  author={Ouyang, Wenqing and Milzarek, Andre},
  journal={Mathematical Programming},
  volume={212},
  number={1},
  pages={389--435},
  year={2025},
  publisher={Springer}
}

@article{qi1995trust,
  title={Trust region algorithms for solving nonsmooth equations},
  author={Qi, Liqun},
  journal={SIAM Journal on Optimization},
  volume={5},
  number={1},
  pages={219--230},
  year={1995},
  publisher={SIAM}
}

@misc{github:leapssn,
title={{LeAP-SSN}},
year={2025},
version = {2025.0.1},
url={https://github.com/amal-alphonse/leapssn},
author={Alphonse, Amal and Dvurechensky, Pavel and Papadopoulos, Ioannis P.~A. and Sirotenko, Clemens}
}

@article {MR2421297,
    AUTHOR = {Bolte, J\'er\^ome and Daniilidis, Aris and Lewis, Adrian},
     TITLE = {Tame functions are semismooth},
   JOURNAL = {Math. Program.},
  FJOURNAL = {Mathematical Programming. A Publication of the Mathematical
              Programming Society},
    VOLUME = {117},
      YEAR = {2009},
    NUMBER = {1-2},
     PAGES = {5--19},
      ISSN = {0025-5610,1436-4646},
   MRCLASS = {49J52 (03C64 03C98 14P10 65H10 90C31)},
  MRNUMBER = {2421297},
       DOI = {10.1007/s10107-007-0166-9},
       URL = {https://doi.org/10.1007/s10107-007-0166-9},
}

@InProceedings{pmlr-v80-yuan18a,
  title = 	 {An Efficient Semismooth {N}ewton based Algorithm for Convex Clustering},
  author =       {Yuan, Yancheng and Sun, Defeng and Toh, Kim-Chuan},
  booktitle = 	 {Proceedings of the 35th International Conference on Machine Learning},
  pages = 	 {5718--5726},
  year = 	 {2018},
  editor = 	 {Dy, Jennifer and Krause, Andreas},
  volume = 	 {80},
  series = 	 {Proceedings of Machine Learning Research},
  month = 	 {10--15 Jul},
  publisher =    {PMLR},
  url = 	 {https://proceedings.mlr.press/v80/yuan18a.html}
}

@book {nesterov2018lectures,
    AUTHOR = {Nesterov, Yurii},
     TITLE = {Lectures on convex optimization},
    SERIES = {Springer Optimization and Its Applications},
    VOLUME = {137},
   EDITION = {Second},
 PUBLISHER = {Springer, Cham},
      YEAR = {2018},
     PAGES = {xxiii+589},
      ISBN = {978-3-319-91577-7; 978-3-319-91578-4},
   MRCLASS = {90-01 (90C25)},
  MRNUMBER = {3839649},
MRREVIEWER = {Giorgio\ Giorgi},
       DOI = {10.1007/978-3-319-91578-4},
       URL = {https://doi.org/10.1007/978-3-319-91578-4},
}

@article{LeapSSN,
  title={{LeAP-SSN: A Semismooth Newton Method with Global Convergence Rates}},
  author={Amal Alphonse and Pavel Dvurechensky and Ioannis P.A. Papadopoulos and Clemens Sirotenko},
  journal={arXiv preprint arXiv:2508.16468},
  year={2025}
}

@inproceedings{hurault2022proximal,
  title={Proximal denoiser for convergent plug-and-play optimization with nonconvex regularization},
  author={Hurault, Samuel and Leclaire, Arthur and Papadakis, Nicolas},
  booktitle={International Conference on Machine Learning},
  pages={9483--9505},
  year={2022},
  organization={PMLR}
}

@article{hertrich2025learning,
  title={Learning regularization functionals for inverse problems: A comparative study},
  author={Hertrich, Johannes and Wong, Hok Shing and Denker, Alexander and Ducotterd, Stanislas and Fang, Zhenghan and Haltmeier, Markus and Kereta, Zeljko and Kobler, Erich and Leong, Oscar and Salehi, Mohammad Sadegh and others},
  journal={arXiv preprint arXiv:2510.01755},
  year={2025}
}

@article{ducotterd2024improving,
  title={Improving Lipschitz-constrained neural networks by learning activation functions},
  author={Ducotterd, Stanislas and Goujon, Alexis and Bohra, Pakshal and Perdios, Dimitris and Neumayer, Sebastian and Unser, Michael},
  journal={Journal of Machine Learning Research},
  volume={25},
  number={65},
  pages={1--30},
  year={2024}
}

@article{pesquet2021learning,
  title={Learning maximally monotone operators for image recovery},
  author={Pesquet, Jean-Christophe and Repetti, Audrey and Terris, Matthieu and Wiaux, Yves},
  journal={SIAM Journal on Imaging Sciences},
  volume={14},
  number={3},
  pages={1206--1237},
  year={2021},
  publisher={SIAM}
}

@article{CarGouToi19,
	author = {Cartis, Coralia and Gould, Nick I. and Toint, Philippe L.},
	booktitle = {SIAM Journal on Optimization},
	da = {2019/01/01},
	date = {2019/01/01},
	doi = {10.1137/16M1106316},
	isbn = {1052-6234},
	journal = {SIAM Journal on Optimization},
	journal1 = {SIAM J. Optim.},
	m3 = {doi: 10.1137/16M1106316},
	number = {1},
	pages = {595--615},
	publisher = {Society for Industrial and Applied Mathematics},
	title = {Universal Regularization Methods: Varying the Power, the Smoothness and the Accuracy},
	ty = {JOUR},
	url = {https://doi.org/10.1137/16M1106316},
	volume = {29},
	year = {2019},
	year1 = {2019}}

@inproceedings{kim2025partial,
  title={{Partial Correlation Network Estimation by Semismooth Newton Methods}},
  author={Kim, DongWon and Lee, Sungdong and Won, Joong-Ho},
  booktitle={The Thirty-ninth Annual Conference on Neural Information Processing Systems},
  year={2025}
}

@ARTICLE{shi2020semismooth,

  author={Shi, Yueyong and Huang, Jian and Jiao, Yuling and Yang, Qinglong},

  journal={IEEE Transactions on Neural Networks and Learning Systems}, 

  title={{A Semismooth Newton Algorithm for High-Dimensional Nonconvex Sparse Learning}}, 

  year={2020},

  volume={31},

  number={8},

  pages={2993-3006},

  doi={10.1109/TNNLS.2019.2935001}}

@article{Luo2019Solving,
  author       = {Ziyan Luo and
                  Defeng Sun and
                  Kim{-}Chuan Toh and
                  Naihua Xiu},
  title        = {Solving the {OSCAR} and {SLOPE} Models Using a Semismooth Newton-Based
                  Augmented Lagrangian Method},
  journal      = {J. Mach. Learn. Res.},
  volume       = {20},
  pages        = {106:1--106:25},
  year         = {2019},
  url          = {https://jmlr.org/papers/v20/18-172.html},
  bibsource    = {dblp computer science bibliography, https://dblp.org}
}

@article{Hu2024projected,
  author  = {Jiang Hu and Kangkang Deng and Jiayuan Wu and Quanzheng Li},
  title   = {{A projected semismooth Newton method for a class of nonconvex composite programs with strong prox-regularity}},
  journal = {Journal of Machine Learning Research},
  year    = {2024},
  volume  = {25},
  number  = {56},
  pages   = {1--32},
  url     = {http://jmlr.org/papers/v25/23-0371.html}
}

@article{Tang2020Sparse,
  author  = {Peipei Tang and Chengjing Wang and Defeng Sun and Kim-Chuan Toh},
  title   = {{A Sparse Semismooth Newton Based Proximal Majorization-Minimization Algorithm for Nonconvex Square-Root-Loss Regression Problems}},
  journal = {Journal of Machine Learning Research},
  year    = {2020},
  volume  = {21},
  number  = {226},
  pages   = {1--38},
  url     = {http://jmlr.org/papers/v21/19-247.html}
}

@article{Zhang2024Efficient,
  author  = {Yangjing Zhang and Ying Cui and Bodhisattva Sen and Kim-Chuan Toh},
  title   = {On Efficient and Scalable Computation of the Nonparametric Maximum Likelihood Estimator in Mixture Models},
  journal = {Journal of Machine Learning Research},
  year    = {2024},
  volume  = {25},
  number  = {8},
  pages   = {1--46},
  url     = {http://jmlr.org/papers/v25/22-1120.html}
}

@article{gratton2025yet,
  title={{Yet another fast variant of Newton’s method for nonconvex optimization}},
  author={Gratton, Serge and Jerad, Sadok and Toint, Philippe L},
  journal={IMA Journal of Numerical Analysis},
  volume={45},
  number={2},
  pages={971--1008},
  year={2025},
  publisher={Oxford University Press}
}

@article{jiang2020semismooth,
  title={{A Semismooth-Newton's-Method-Based Linearization and Approximation Approach for Kernel Support Vector Machines}},
  author={Jiang, Chen and Li, Qingna},
  journal={arXiv preprint arXiv:2007.11954},
  year={2020}
}

@article{Chow2016Analysis,
author = {Chow, Yat Tin and Ito, Kazufumi and Zou, Jun},
title = {Analysis on a Nonnegative Matrix Factorization and Its Applications},
journal = {SIAM Journal on Scientific Computing},
volume = {38},
number = {5},
pages = {B645-B684},
year = {2016},
doi = {10.1137/15M1020824},

URL = { 
    
        https://doi.org/10.1137/15M1020824
    
    

},
eprint = { 
    
        https://doi.org/10.1137/15M1020824
    
    

}
}

@inproceedings{lin_cuturi_jordan_2024,
  title={{A Specialized Semismooth Newton Method for Kernel-Based Optimal Transport}},
  author={Lin, Tianyi and Cuturi, Marco and Jordan, Michael I.},
  booktitle={Proceedings of the 27th International Conference on Artificial Intelligence and Statistics (AISTATS)},
  series={Proceedings of Machine Learning Research},
  volume={238},
  pages={145--153},
  year={2024}
}

@article{sun_sun_2002,
  title={Semismooth Matrix-Valued Functions},
  author={Sun, Defeng and Sun, Jie},
  journal={Mathematics of Operations Research},
  volume={27},
  number={1},
  pages={150--169},
  year={2002},
  doi={10.1287/moor.27.1.150.342}
}

@InProceedings{sparsemax,
  title = 	 {From Softmax to Sparsemax: A Sparse Model of Attention and Multi-Label Classification},
  author = 	 {Martins, Andre and Astudillo, Ramon},
  booktitle = 	 {Proceedings of The 33rd International Conference on Machine Learning},
  pages = 	 {1614--1623},
  year = 	 {2016},
  editor = 	 {Balcan, Maria Florina and Weinberger, Kilian Q.},
  volume = 	 {48},
  series = 	 {Proceedings of Machine Learning Research},
  address = 	 {New York, New York, USA},
  month = 	 {20--22 Jun},
  publisher =    {PMLR},
  url = 	 {https://proceedings.mlr.press/v48/martins16.html}
}

@ARTICLE{cichocki2007nonnegative,
  author={Cichocki, Andrzej and Zdunek, Rafal and Amari, Shun-ichi},
  journal={IEEE Signal Processing Magazine}, 
  title={Nonnegative Matrix and Tensor Factorization [Lecture Notes]}, 
  year={2008},
  volume={25},
  number={1},
  pages={142-145},
  doi={10.1109/MSP.2008.4408452}}

@book{gillis2020nonnegative,
    AUTHOR = {Gillis, Nicolas},
     TITLE = {Nonnegative matrix factorization},
    SERIES = {Data Science},
    VOLUME = {2},
 PUBLISHER = {Society for Industrial and Applied Mathematics (SIAM),
              Philadelphia, PA},
      YEAR = {2021},
     PAGES = {xxv+350},
      ISBN = {978-1-611976-40-3},
   MRCLASS = {15A23 (15-02 15B48 65-02 65F99 68W40)},
  MRNUMBER = {4191210},
}

@article{overton1992large,
  title={Large-scale optimization of eigenvalues},
  author={Overton, Michael L},
  journal={SIAM Journal on Optimization},
  volume={2},
  number={1},
  pages={88--120},
  year={1992},
  publisher={SIAM}
}

@book{golub2013matrix,
  title     = {Matrix Computations},
  author    = {Golub, Gene H. and Van Loan, Charles F.},
  publisher = {JHU Press},
  year      = {2013},
  edition   = {4},
  isbn      = {9781421407944},
  doi       = {10.1137/1.9781421407944}
}

@article{Gfrerer2025,
  author    = {Gfrerer, Helmut},
  title     = {On a globally convergent semismooth* {N}ewton method in nonsmooth nonconvex optimization},
  journal   = {Computational Optimization and Applications},
  year      = {2025},
  volume    = {91},
  number    = {1},
  pages     = {67--124},
  doi       = {10.1007/s10589-025-00658-z},
}

@InProceedings{Wang2024Monotone,
  title     = {Monotone, Bi-Lipschitz, and Polyak--{\L}ojasiewicz Networks},
  author    = {Wang, R. and Dvijotham, K.~D. and Manchester, I.~R.},
  booktitle = {Proceedings of the 41st International Conference on Machine Learning},
  pages     = {50379--50399},
  year      = {2024},
  editor    = {Salakhutdinov, R. and Kolter, Z. and Heller, K. and Weller, A. and Oliver, N. and Scarlett, J. and Berkenkamp, F.},
  volume    = {235},
  series    = {Proceedings of Machine Learning Research},
  month     = {Jul},
  publisher = {PMLR},
  url       = {https://proceedings.mlr.press/v235/wang24p.html}
}

@article{xiao2016semismooth,
  author    = {Xiao, X. and Li, Y. and Wen, Z. and Zhang, L.},
  title     = {A Regularized Semi-Smooth Newton Method With Projection Steps for Composite Convex Programs},
  journal   = {arXiv preprint arXiv:1603.07870},
  year      = {2016},
  eprint    = {1603.07870},
  archivePrefix = {arXiv},
  primaryClass  = {math.OC}
}

@article{liu2022loss,
  author  = {Liu, C. and Zhu, L. and Belkin, M.},
  title   = {Loss Landscapes and Optimization in Over-Parameterized Nonlinear Systems and Neural Networks},
  journal = {Applied and Computational Harmonic Analysis},
  volume  = {59},
  pages   = {85--116},
  year    = {2022},
  publisher = {Elsevier}
}

@inproceedings{chen2023pl,
  author    = {Chen, X. and Shi, Y. and Dong, M. and Yang, X. and Li, D. and Wang, Y. and Dick, R.~P. and Lv, Q. and Zhao, Y. and Yang, F. and Gu, N. and Shang, L.},
  title     = {Over-Parameterized Model Optimization with Polyak--{\L}ojasiewicz Condition},
  booktitle = {Proceedings of the 11th International Conference on Learning Representations (ICLR)},
  year      = {2023}
}

@inproceedings{liu2024sophia,
  title={Sophia: A Scalable Stochastic Second-order Optimizer for Language Model Pre-training},
  author={Hong Liu and Zhiyuan Li and David Hall and Percy Liang and Tengyu Ma},
  booktitle={The International Conference on Learning Representations (ICLR)},
  year={2024}
}

@InProceedings{gupta2018shampoo,
  title = 	 {Shampoo: Preconditioned Stochastic Tensor Optimization},
  author =       {Gupta, Vineet and Koren, Tomer and Singer, Yoram},
  booktitle = 	 {Proceedings of the 35th International Conference on Machine Learning},
  pages = 	 {1842--1850},
  year = 	 {2018},
  editor = 	 {Dy, Jennifer and Krause, Andreas},
  volume = 	 {80},
  series = 	 {Proceedings of Machine Learning Research},
  month = 	 {10--15 Jul},
  publisher =    {PMLR},
  url = 	 {https://proceedings.mlr.press/v80/gupta18a.html}
}

@article{doikov2021minimizing,
  author  = {Doikov, Nikita and Nesterov, Yurii},
  title   = {Minimizing Uniformly Convex Functions by Cubic Regularization of {Newton} Method},
  journal = {Journal of Optimization Theory and Applications},
  volume  = {189},
  number  = {1},
  pages   = {317--339},
  year    = {2021},
  doi     = {10.1007/s10957-021-01838-7}
}
 \appendix
\onecolumn

\section{Discussions}
\subsection{Newton differentiability and semismoothness}
\label{S:Newton_derivative}
For a function $f \in C^{1}(\dom f, \overline{\mathbb{R}})$,  we say that $f'$ is \textit{Newton differentiable} at a point $x\in \dom f$ if there exists an open subset $U\subset \dom f $ with $x \in U$ and a family of maps $H\colon U \to \mathcal{L}(\HH, \HH^*)$ such that 
\begin{equation}
\label{eq:semismoothness_def}
    \lim_{h \in \HH, h \to 0} \frac{\norm{f'(x+h)-f'(x)-H(x+h)h}{*}}{\norm{h}{}} = 0.
\end{equation}
(For Fréchet differentiability, one should replace  $H(x+h)$ above with $H(x)$.) We also say that $f'$ is Newton differentiable at $x$ with respect to $H$. If $f'$ is Newton differentiable for every $x\in U$, we say that it is Newton differentiable on $U$. %
\begin{itemize}
\item Newton differentiability is defined in some works in terms of set-valued derivatives: a map $T\colon X \to Y$ is said to be Newton differentiable with Newton derivative $G\colon X \rightrightarrows \mathcal{L}(X,Y)$ if
\[\lim_{h \to 0} \sup_{M \in G(x+h)}\frac{\norm{T(x+h)-T(x)-Mh}{}}{\norm{h}{}} =0.\]
The definition of Newton derivative we use  results from taking a selection from the set-valued map $G$ at each point. Our Newton derivative can be viewed as an abuse of notation where a particular realization of the derivative has been chosen; in a sense then, our algorithm also can handle set-valued derivatives. 
\item A commonly-used definition of semismoothness in literature is the following. %
The map $T$ is \emph{semismooth} at $x$ if it is locally Lipschitz, directionally differentiable at $x$, and for any $V \in \partial T(x+h)$ with  $h \to 0$, we have
\[T(x+h)-T(x) - Vh = \so(h),\]
where $\partial T$ refers to the generalized or Clarke Jacobian of $T$. See for example \cite{pmlr-v80-yuan18a} or \cite{lin_cuturi_jordan_2024}.

We did not specify which particular maps $H(x)$ are used in our definition of Newton differentiability. We could choose, for example, a single-valued selection of Clarke’s Jacobian as generalized derivative $H(x)$ and in this way we recover semismoothness (provided the Lipschitz condition holds). Moreover, if $f'$ is semismooth, every single-valued selection of Clarke’s Jacobian works as $H$ in the above definition.     
    \item In light of these facts, and following established nomenclature within the community, we use the terms semismoothness and Newton differentiability interchangeably. Particularly, Newton differentiability implies \cref{ass:semismoothness} and one can choose Newton derivative as $H$ in our algorithm. At the same time, let us underline that our paper covers a much more general setting that extends the notion of differentiability even further than semismooth functions. Indeed, despite being \emph{less} restrictive than other definitions, Newton differentiability is sufficient for a local superlinear convergence analysis, as we see in \cref{sec:fast}.
\end{itemize}
For more details and comparisons, we refer to \cite{RickmannHerzogHeberg}.  

\subsection{Further discussion of our algorithm}
A key feature of our method is that we allow for lazy Hessian updates, with fresh Hessians being taken only every $m$ iterations. To be precise, we utilize $H(x_0)$ in the trial step \eqref{eq:prox_step_alg}  for iterates $k=0, \hdots m-1$, then $H(x_m)$ for iterates $k=m, \hdots, 2m-1$, and so on. This is, as mentioned, the same setup considered in \cite{doikov2023second} (but under Lipschitz Hessian assumptions there). 

Let us note that when $\psi \equiv 0$, the trial step \eqref{eq:prox_step_alg} simplifies to 
\[x_+ = x_k - (H(x_{\pi(k)}) + \lambda\mathcal{R})^{-1}f'(x_k)\]
with $\lambda = 4^{j_k}\Lambda_k \norm{f'(x_k)}{}^p.$ In this case, each trial step requires a standard linear system solve. In the general case, following the existing literature \cite{potzl2022second,doikov2023second,doikov2024super,nesterov2006cubic,nesterov2008accelerating}, we assume that the subproblem \eqref{eq:prox_step_alg} is solved exactly. We defer the study of inexact subproblem solves to future work. We emphasize that gradient regularization is essential (i.e., $p$ is not allowed to be $0$) for our $\gamma$-order semismoothness result \cref{thm:fast_local_gamma}. Namely, it is needed in the proof of \cref{prop:lambda_k_to_zero} (given below) in the inductive step. The use of gradient regularization in our algorithm is inspired by the paper \cite{doikov2024super}.  Yet, one key difference between our algorithm and theirs is the second stopping condition in line \eqref{eq:stop_cond_alg} of \cref{alg:proximal_newton}, which is crucial to obtain global convergence of the iterates under the PL condition, see the proof of \cref{thm:global_convergence_nonconv_PL} below. We note also that gradient regularization is crucial for achieving $\frac{1+2\gamma}{1+\gamma}$-order superlinear convergence.

\begin{table}[t]
\centering
\caption{Positioning against closely related methods.}
\label{tab:method-comparison-text}
\begin{tabular}{>{\centering\arraybackslash}m{2.8cm}ccccc}
\toprule
Method
& Global rates%
& \parbox[c]{1.6cm}{\centering $\gamma$-order local\\ rates}
& \parbox[c]{1.6cm}{\centering Lazy Hessian \\updates}
& \parbox[c]{1.6cm}{\centering Infinite dimensional \\setting} 
& \parbox[c]{1.6cm}{\centering $f'$ only semismooth \\allowed} \\
\midrule
\leapssn \cite{LeapSSN} 
& \greencheck
& \redcross
& \redcross
& \greencheck 
& \greencheck \\
\midrule

Method from \cite{doikov2023second}
& \;\greencheck $^\star$
& n/a
& \greencheck
& \redcross 
& \redcross \\
\midrule

Super-universal Newton \cite{doikov2024super}
& \;\greencheck 
& \greencheck $^\star$ 
& \redcross 
& \redcross 
& \redcross \\
\midrule

\gladssn (this work)
& \greencheck
& \greencheck%
& \greencheck
& \greencheck
& \greencheck\\
\bottomrule
\end{tabular}

 \vspace{2pt}
 {\footnotesize $^{\star}$ In \cite{doikov2023second} global rates are proven under the additional assumption that $f \in C^2$ and $\nabla^2 f$ is Lipschitz continuous. \cite{doikov2024super} obtain $\gamma$-order local rates under $\gamma$-H\"older Hessian with $\gamma>0$ or under H\"older third derivative.}
\end{table}
We conclude this subsection by summarizing what our paper achieves in comparison to the most closely-related existing works in the literature, see \cref{tab:method-comparison-text}.
\subsection{Further discussion of results}
\label{sec:app_discussion}
\rev{We summarize the results and assumptions
in \cref{tab:table_results}. 
\begin{table}[ht]
\caption{High-level summary of our main results under \cref{ass:basic} and other assumptions indicated in the table.}
\centering
\begin{tabularx}{\textwidth}{
    >{\raggedright\arraybackslash}p{0.47\textwidth}
    l
    X
}
\midrule
\textbf{Assumption} & \textbf{Convergence rate}  &\textbf{Result} \\
\midrule
$H\succeq 0$ & $\min_{i \leq k}\norm{F'(x_{i})}{*}=\mathcal{O}(1\slash \sqrt{k})$  &\cref{thm:global_convergence_sublinear_nonconv}\\
\midrule
$F$ convex, $H\succeq 0$ &$\mathcal{O}(1\slash k)$ for $F(x_k)$ &\cref{thm:global_convergence_sublinear}
\\\midrule
$F$ strictly convex,  locally $f \in C^2$,  $H\succeq 0$, $\dim(\HH) < \infty$  & $\so(1\slash k)$ for $F(x_k)$ &\cref{thm:global_convergence_sublinear}    
\\\midrule
PL  & Linear for $x_k$ and $F(x_k)$ &\cref{thm:global_convergence_nonconv_PL}\\\midrule
PL, locally $f \in C^2$, $F$ locally convex, $\dim(\HH) < \infty$  & Local superlinear &\cref{thm:global_convergence_c2}   \\\midrule
PL, $f$ semismooth, DM-type condition,
$F$ locally strongly convex, $H \succeq 0$
& Local superlinear
& \cref{thm:global_superlinear} \\\midrule

PL, $f$ $\gamma$-order semismooth,
$\gamma$-order DM-type condition,
$F$ locally strongly convex, $H \succeq 0$
& Local superlinear with rate
& \cref{thm:global_superlinear_gamma} \\
\bottomrule
\end{tabularx}
\label{tab:table_results}
\end{table}
}
The overall picture of the convergence rate results depending on the parameter $p$ is as follows.
Global rates hold for the whole range $p\in [0,1]$ except for linear convergence $x_k\to x^*$ under the PL condition, which is valid for $p\in [0,1)$. Superlinear convergence under semismoothness or in $C^2$ case is obtained for the same range of $p$, and improved superlinear convergence under $\gamma$-order semismoothness further restricts $p$ to $p\in[1/2,1)$. Our algorithm is class- and parameter-agnostic, and the universal choice is $p=1/2$ for which all the results hold. If, on the other hand, some properties of the objective class are known, they can be used for the choice of $p$. Yet, neither theoretically nor in experiments do we see a significant influence of the choice of $p$ on the convergence of the algorithm. 

The main technical difficulties that we resolve in this paper lie in the construction of the stopping conditions that allow us to obtain all the results in a completely class- and parameter-agnostic way. In particular, compared to other adaptive second-order methods, we had to introduce the second inequality in line \eqref{eq:stop_cond_alg} of \cref{alg:proximal_newton}, which is crucial for obtaining the convergence of the iterates under PL condition. Another main technical challenge was to prove the superlinear and $\gamma$-order superlinear convergence. The high-level idea is to show that the semismoothness allows us to obtain an analogue of \eqref{eq:L_condition}, but with a constant $L$  which goes to $0$ with a certain rate. The next step is to show that locally the regularization parameter $\lambda_k$ also goes to $0$ with a certain rate, which eventually gives the desired improved rates under semismoothness. In that way, we strictly prove that our stopping conditions allow the algorithm to fully adapt to the local geometry of the objective function. To show the improved local rate in \cref{thm:fast_local_gamma}, we had to introduce gradient regularization with a careful choice of  the exponent $p$ and show that this does not ruin other results like global rates.  Finally, we prove all this despite the fact that the algorithm uses “old” Hessians from the previous iterations, and the challenge was to show that this does not ruin other results, especially the local rates.

The introduction mentions Huber, squared hinge, and \rev{sparsemax losses. All these losses are convex, have strongly-semismooth gradients, and \cref{ass:basic} holds, thus, \cref{thm:global_convergence_sublinear} holds}. Moreover, if one adds quadratic regularization, then these regularized losses become strongly convex, and \rev{many of} our theoretical results apply. %

According to \cref{thm:global_superlinear}, the whole convergence picture is as follows. From any starting point, the method converges linearly to a solution $x^*$ thanks to the PL condition. Since the linear convergence $x_k\to x^*$ is guaranteed, at some point $x_k$ starts to belong to the ball around $x^*$ where the local strong convexity of \cref{ass:strong_convexity} holds. The asymptotic semismoothness \cref{ass:semismoothness} and DM conditions also start to work after some iteration when $x_k$ is sufficiently close to $x^*$. All this implies that $\lambda_k\to0$ and the method starts to converge superlinearly when $\lambda_k$ is sufficiently small. In that sense, the superlinear convergence is local and asymptotic, but the linear rate is global and nonasymptotic. \rev{\cref{thm:global_superlinear_gamma}} complements the convergence picture by quantifying the order of local superlinear convergence under $\gamma$-order semismoothness in \cref{ass:gamma_order_semismoothness}.

Compared to the existing literature on the global rates for second-order methods, which require $f \in C^{2,2}$, our results are obtained under a weaker \cref{ass:basic}, which does not even require that $f \in C^2$. For example, $\max\{a^\top x - b, 0\}^2$ satisfies our \cref{ass:basic,ass:semismoothness,ass:gamma_order_semismoothness} while not being twice continuously differentiable, nor having Lipschitz Hessian. Further, existing works obtain local superlinear convergence under Lipschitz Hessians in the vicinity of the solution. On the contrary, our superlinear convergence results are based on the weaker ($\gamma$-order) semismoothness \cref{ass:semismoothness,ass:gamma_order_semismoothness} and these assumptions are made only at the solution, rather than in the whole vicinity.

We would like to underline that \cref{lm:lemlazyDMHolds} implies that the results on superlinear convergence in \cref{thm:fast_local} hold also in the infinite-dimensional $C^2$ setting. This also means that we have a hierarchy of results: for $f \in C^{1,1}$, we have results on linear and sublinear convergence rates, for $f \in C^2$, these results asymptotically improve, and for $f \in C^{2,2}$, it is possible to obtain improved convergence rates. The intermediate case $f \in C^2$ is not well explored in the literature and our paper gives some understanding of how algorithms should work in this situation.

To simplify the presentation, we assumed local strong convexity in \cref{ass:strong_convexity}. This can be relaxed to local convexity to prove \cref{prop:lambda_k_to_zero} and to $F$ satisfying locally the PL condition, quadratic growth, and $x^*$ being isolated to prove the inequalities in \cref{thm:fast_local,thm:fast_local_gamma}.

Our theoretical results are not limited to the PL setting. As stated in \cref{thm:fast_local,thm:fast_local_gamma}, transition to superlinear convergence happens whenever $x_k\to x^*$, which can be established in settings beyond PL, e.g., under strict convexity. Importantly, all our results are proved for a single class-agnostic method and there is no need to check in advance whether the PL condition holds or not. And even in the cases when PL is difficult to verify, we have global rate given by \cref{thm:global_convergence_sublinear_nonconv,thm:global_convergence_sublinear}.  Our \cref{thm:global_convergence_sublinear_nonconv} states that the algorithm stays on the level set and that each iteration decreases the value of the objective. Thus, it is sufficient to assume PL on a level set.

Regarding our results in the finite-dimensional $C^2$ setting, the reader might wonder why we have both  \cref{thm:global_convergence_sublinear} (in particular, the statement in parentheses) and \cref{thm:global_convergence_c2}. The reason is that in the former, the PL inequality is not needed, unlike in \cref{thm:global_convergence_c2}.

\rev{Our current theory does not prescribe a way to choose the laziness parameter $m$. This seems to be due to the lack of sufficient smoothness in our setting, unlike \cite{doikov2023second}. We expect though, that in some problems, in some neighbourhood of the generated iterates, the function is actually smoother, wherein similar optimal choices to \cite{doikov2023second} may be available.}

Finally, we conclude by mentioning that this work is a first step in deriving an algorithm that possesses accelerated local rates and takes into account that real-world problems may have Hessians that are costly to evaluate (this is indeed the case in our neural network example of \cref{sec:numerics_NN}). There are, of course, many additional aspects that one could do to increase the efficiency of the proposed method, e.g., by employing such tools as sketching or other randomization techniques, quasi-Newton updates, etc. We leave these interesting aspects for future work.

\section{Technical lemmas}\label{sec:appendix_helper}
In this section, we collect several technical results that are used later to prove our main results.

For further convenience, using the definition of the point $x_{k+1}$ in line 9 of \cref{alg:proximal_newton} and the shortcut notation \eqref{eq:F_r_g_notaion}, the acceptance conditions in line \eqref{eq:stop_cond_alg} of \cref{alg:proximal_newton} can be written as 
\begin{align}   
\langle F'(x_{k+1}),x_k -x_{k+1}\rangle &\geq \frac{1}{2\lambda_k}\norm{F'(x_{k+1})}{*}^2= \frac{1}{2\lambda_k}g_{k+1}^2, \label{eq:stop_cond_alg_new_1} \\
    F_k-F_{k+1}=F(x_k)-F(x_{k+1})&\geq \frac{\lambda_k}{4}\norm{x_{k+1}-x_k}{}^2= \frac{\lambda_k}{4}r_k^2.\label{eq:stop_cond_alg_new_2}
\end{align}

 We will use the fact that our special choice of $F'(x)$ allows us to control the step length. Namely, if $H(z) \succeq \mu I$  for some $\mu \in \mathbb{R}$ and $\lambda+\mu>0$, %
\citep[Lemma~1]{doikov2024super}
 gives 
\begin{equation} 
    \norm{x_+-x}{} \leq  \frac{\norm{F'(x)}{*}}{\lambda + \mu}.\label{eq:rk_leq_gk_over_lambda_k_new}
\end{equation}
\rev{Here recall that $x_+ = x_+(\lambda, x, z)$. }

The following lemma collects several consequences of the stopping conditions. For convenience, we repeat each statement in full and also in shortcut notation \eqref{eq:F_r_g_notaion}.
\begin{lem}\label{lem:functionValuesGradientRelation}
Let, at iteration $k\geq 0 $ of \cref{alg:proximal_newton}, acceptance criteria  in line \eqref{eq:stop_cond_alg} of \cref{alg:proximal_newton} hold. Then, we have %
    \begin{align}
    F^*\leq F(x_{k+1})\leq F(x_k) &\qquad\text{or equivalently}\quad 0 \leq F_{k+1}\leq F_k, \label{eq:F_k_decr}\\
    \norm{F'(x_{k+1})}{*} \leq 2\lambda_k\norm{x_k-x_{k+1}}{} &\qquad\text{or equivalently}\quad g_{k+1} \leq 2\lambda_kr_k,\label{eq:g_k_leq_r_k}\\
    F(x_k)-F(x_{k+1}) \geq \frac{1}{16\lambda_k}\norm{F'(x_{k+1})}{*}^2 &\qquad\text{or equivalently}\quad F_k-F_{k+1}%
    \geq \frac{1}{16\lambda_k}g_{k+1}^2,\label{eq:functionValuesGradientRelation}
    \intertext{\rev{and if $H(x_{\pi(k-1)}) \succeq 0$,}}
        \norm{F'(x_{k})}{*} \leq 2 \norm{F'(x_{k-1})}{*} &\qquad\text{or equivalently} \quad g_k \leq 2g_{k-1}.\label{eq:g_k_leq_g_k-1}
\end{align}
\end{lem}
\begin{proof}
The inequalities in \eqref{eq:F_k_decr} follow from \eqref{eq:stop_cond_alg_new_2}, \eqref{eq:F_r_g_notaion}, and \cref{ass:basic}. The inequalities in \eqref{eq:g_k_leq_r_k} follow from \eqref{eq:stop_cond_alg_new_1} and \eqref{eq:F_r_g_notaion}. 
Combining \eqref{eq:stop_cond_alg_new_2}, \eqref{eq:F_r_g_notaion}, and \eqref{eq:g_k_leq_r_k}, we obtain \eqref{eq:functionValuesGradientRelation}. Finally,  \eqref{eq:g_k_leq_g_k-1} follows from \eqref{eq:g_k_leq_r_k}, \eqref{eq:F_r_g_notaion}, and \eqref{eq:rk_leq_gk_over_lambda_k_new}.
\end{proof}
The next result establishes a sensitivity estimate of how the point $x_+(\lambda,x,z)$ in \eqref{eq:PLMSN_step} varies for fixed $x$ when $\lambda$ changes. We omit the proof because it is almost identical to that of \citep[Lemma 2.5]{LeapSSN}.
\begin{lem}\label{lem:lambda_dependence}
Let $x \in \dom f$ be such that $H(z) \succeq 0$. Then, the following inequality holds for all $0 < \lambda \leq \lambda'$:
  \begin{equation}
       \|x_+(\lambda,x,z) - x_+(\lambda',x,z) \| \leq \frac{\lambda' - \lambda}{\lambda' } \| x -x_+(\lambda,x,z)\|.\label{eq:lambda_dependence}
  \end{equation}
\end{lem}

The next result can be found in the proof of Lemma 6.3 of \cite{LeapSSN}. We state it here for convenience, and omit the proof.
\begin{lem}\label{lem:helper_lambdak_zero}
         If there exists a sequence $(\varepsilon_k)_{k \in \mathbb{N}} \subset \R_+$ such that $\varepsilon_k \to 0$ and, for sufficiently large $k$,
            \begin{align*}
        \lambda_k \leq \begin{cases}
        \varepsilon_k, & j_k > 0, \\
            \frac{\lambda_{k-1}}{2}, & j_k = 0,
        \end{cases}
    \end{align*}
        then, we have $\lambda_k \to 0$ as $k \to \infty$. 
        \end{lem}

\section{Proofs for \cref{sec:global} on global convergence rates}
\label{sec:global_proofs}

We first prove the result with the upper bound on the sequence of the regularization parameters $(\lambda_k)_{k \in \mathbb{N}}$.

\lemacceptanceinnernonconv*
\begin{proof}%
Let us denote for brevity $x_+(\lambda)\coloneqq x_+(\lambda,x_k,x_{ \pi(k)})$.
\rev{First note that by \cref{ass:basic} a trial parameter $\lambda \geq L > L \slash 2$ implies that the subproblem \eqref{eq:prox_step_alg} is strongly convex and admits a unique stationary point. We show that also the two acceptance  inequalities in \eqref{eq:stop_cond_alg}  hold in this case. }
Applying items (i) and (iv) of \cref{ass:basic}, we obtain, for any $\lambda \geq 0 $,
\begin{align}
    \norm{f'(x_+(\lambda)) -  f'(x_k) - H(x_{\pi(k)})(x_+(\lambda)-x_k)}{*}&\leq \norm{f'(x_+(\lambda)) -  f'(x_k)}{*} \notag\\
    &\quad+ \norm{H(x_{\pi(k)})(x_+(\lambda)-x_k)}{*} \notag\\
    &\leq L\norm{x_k-x_+(\lambda)}{}.\label{eq:L_condition}
\end{align}
Whence, for any $\lambda \geq 0 $,
\begin{align}
L^2\norm{x_k-x_+(\lambda)}{}^2 \overset{\eqref{eq:L_condition}}&{\geq} \norm{f'(x_+(\lambda)) -  f'(x_k) - H(x_{\pi(k)})(x_+(\lambda)-x_k)}{*}^2 \notag\\
\overset{\eqref{eq:F_div_def}}&{=} \norm{F'(x_+(\lambda)) + \lambda \mathcal{R}(x_+(\lambda) - x_k)}{*}^2 \notag \\
&= \norm{F'(x_+(\lambda))}{*}^2 + 2\lambda\langle F'(x_+(\lambda)),x_+(\lambda) - x_k \rangle + \lambda^2\norm{x_k - x_+(\lambda)}{}^2. \label{eq:correctness_proof_1}
\end{align} 
Dividing by $2 \lambda$ and rearranging, we obtain, for any $\lambda \geq  L$,
\begin{align}
\langle F'(x_+(\lambda)),x_k - x_+(\lambda) \rangle
          &\geq \frac{1}{2 \lambda}\norm{F'(x_+(\lambda))}{*}^2 +  \frac{\lambda}{2} \norm{x_k - x_+(\lambda)}{}^2 - \frac{L^2\norm{x_k - x_+(\lambda)}{}^2}{2\lambda} \label{eq:correctness_proof_1_1}\\
         \overset{\lambda \geq L }&{\geq} \frac{1}{2 \lambda}\norm{F'(x_+(\lambda))}{*}^2. \notag
\end{align} 
Thus, since $\lambda=4^{j_k}\Lambda_kg_k^p$ in line 4 of \cref{alg:proximal_newton} is increasing when $j_k$ is increasing, we obtain that the first inequality in line \eqref{eq:stop_cond_alg} of \cref{alg:proximal_newton} holds after a finite number of trials when $\lambda\geq L$. 

We now move to the second inequality in \eqref{eq:stop_cond_alg} in \cref{alg:proximal_newton}. 
\rev{Using that $L/2$-smoothness of $f$ implies the descent lemma for $f$ and using the $L/2$ boundedness of $H(x_{\pi(k)})$,} 
we obtain, for any $\lambda \geq 0 $,
\begin{align}
f(x_+(\lambda))\leq &f(x_k)+\langle f'(x_k), x_+(\lambda)-x_k \rangle + \frac 12 \langle H(x_{\pi(k)}) (x_+(\lambda)-x_k), x_+(\lambda)-x_k \rangle \nonumber\\
&+ \frac{L}{2}\norm{x_+(\lambda)-x_k}{}^2. \label{eq:lem:acceptance_inner_nonconv_1_proof_1}
\end{align}
Using convexity of $\psi$ and \eqref{eq:PLMSN_step_optimality}, we obtain, for any $\lambda \geq 0$
\begin{align}
\psi(x_+(\lambda))
&\leq \psi(x_k)+\langle -\psi'(x_+(\lambda)), x_k-x_+(\lambda) \rangle \notag\\
\overset{\eqref{eq:PLMSN_step_optimality}}&{=} \psi(x_k)+\langle f'(x_k) + H(x_{\pi(k)}) (x_+(\lambda)-x_k)+\lambda \mathcal{R} (x_+(\lambda)-x_k), x_k-x_+(\lambda) \rangle.
\label{eq:lem:acceptance_inner_nonconv_1_proof_2}
\end{align}
Summing the last inequality with \eqref{eq:lem:acceptance_inner_nonconv_1_proof_1} and using that the Hessian is bounded by $L/2$, we further get, for any $\lambda \geq L$,
\begin{align}
F(x_+(\lambda))&\leq F(x_k)-\frac 12 \langle H(x_{\pi(k)}) (x_+(\lambda)-x_k),x_+(\lambda)-x_k\rangle -\lambda \norm{x_+(\lambda)-x_k}{}^2 \notag\\
& \quad + \frac{L}{2}\norm{x_+(\lambda)-x_k}{}^2 \notag \\
\overset{\text{bounded Hess.}}&{\leq}F(x_k)+ \frac{L}{4}\norm{x_+(\lambda)-x_k}{}^2-\lambda \norm{x_+(\lambda)-x_k}{}^2+ \frac{L}{2}\norm{x_+(\lambda)-x_k}{}^2 \notag\\
\overset{L \leq \lambda}&{\leq} %
F(x_k)-\frac{\lambda}{4}\norm{x_+(\lambda)-x_k}{}^2. \label{eq:lem:acceptance_inner_nonconv_1_proof_22}
\end{align}
Thus, since $\lambda=4^{j_k}\Lambda_kg_k^p$ in line 4 of \cref{alg:proximal_newton} is increasing when $j_k$ is increasing, we obtain, rearranging the last inequality, that the second inequality in line \eqref{eq:stop_cond_alg} of \cref{alg:proximal_newton} holds after a finite number of trials when $\lambda\geq L$. 
Thus, both inequalities in line \eqref{eq:stop_cond_alg} of \cref{alg:proximal_newton} hold after a finite number of trials, and the step of \cref{alg:proximal_newton} is well defined. 

We now move to the proof of the bound \eqref{eq:lambda_bound} by induction. Consider the base case,
\rev{
i.e., iteration $k=0$. If $j_{0}=0$, i.e., the acceptance condition holds for $\lambda_{0}=\Lambda_{0} g_{0}^p$ and we have that $\lambda_{0}\leq\Lambda_{0}g_{0}^p\leq\overline{\lambda}_{}$. 
If $j_{0}>0$, the preceding trial with parameter \(\lambda_{0}/4\) was rejected. By the first part of the lemma, this is impossible if \(\lambda_{0}/4\geq L\), since that trial would admit a stationary point and this point also satisfies both acceptance inequalities.}
Thus, $\lambda_{0}\leq 4L\leq \overline{\lambda}_{}$. 

Assume now that the bound \eqref{eq:lambda_bound} holds for iteration $k-1$ and consider iteration $k$. If $j_k=0$, we have 
\begin{equation}
\label{eq:lambda_bound_proof_1}
\lambda_k=\Lambda_kg_k^p = \frac{4^{j_{k-1}}\Lambda_{k-1}}{4}g_k^p = \frac{\lambda_{k-1}}{4g_{k-1}^p}g_k^p    \overset{\eqref{eq:g_k_leq_g_k-1}}{\leq}   \frac{\lambda_{k-1}2^p}{4} \overset{p\leq 1}{\leq} \frac{\lambda_{k-1}}{2} \leq \overline{\lambda}_{} ,    
\end{equation}
where we used that $2^p/4\leq 1/2$ for any $p\in [0,1]$ and the induction hypothesis $\lambda_{k-1} \leq\overline{\lambda}_{} $.
If $j_k>0$, we have again that $\lambda_k\slash 4 \leq L$ since it was not accepted by the same argument as before. Hence, $\lambda_k \leq 4L \leq \overline{\lambda}_{}$. Thus, we have proved the first bound in \eqref{eq:lambda_bound}. 
The second inequality in \eqref{eq:lambda_bound} just follows by definition, since for any \rev{ $k \geq 0$},
$\Lambda_{k+1}=\lambda_k/(4g_k^p)\leq\overline{\lambda}_{}/(4g_k^p)$. 
\end{proof}

\subsection{Non-convex problems}
\label{sec:global_proofs_nonconv}

Equipped with the upper bound on $\lambda_k$, we can now move on to the proof of the convergence rate theorem for the general non-convex setting. 

\thmglobalconvergencesublinearnonconv*
\begin{proof}
We have by \cref{lem:acceptance_inner_nonconv} that both inequalities in line \eqref{eq:stop_cond_alg} of \cref{alg:proximal_newton} hold in each iteration $k\geq 0$, i.e., the acceptance conditions of the inner loop of \cref{alg:proximal_newton} hold and the inner loop ends after a finite number of trials. Thus, the sequence $(x_k)_{k \in \mathbb{N}}$ is well defined. 
From \eqref{eq:F_k_decr}, it is clear by induction that for all $k\geq 0$, $F(x_k)\leq F(x_0)$ and thus $(x_k)_{k \in \mathbb{N}} \subset \F_0$. 

Summing \eqref{eq:functionValuesGradientRelation} from $0$ to $k-1$ and telescoping, using that $F$ is bounded from below by $F^*$ by \cref{ass:basic}, and
\cref{lem:acceptance_inner_nonconv}, we obtain
\[
F_0=F({x}_0) - F^*  \geq F({x}_0) - F(x_{k})  \overset{\eqref{eq:functionValuesGradientRelation}}{\geq}  \sum_{i=0}^{k-1}\frac{1}{16\lambda_i}g_{i+1}^2 \overset{\eqref{eq:lambda_bound}}{\geq} 
\sum_{i=0}^{k-1} \frac{1}{16\overline{\lambda}}g_{i+1}^2 \geq\frac{k}{16\overline{\lambda}} \min_{0\leq i \leq k-1}\norm{F'(x_{i+1})}{*}^2.
\]
First, observe that this chain of inequalities implies that for all $k\geq 1$, $\sum_{i=0}^{k-1}g_{i+1}^2$ is bounded from above. Thus, $\norm{F'(x_{i})}{*}=g_{i} \to 0$ as $i\to \infty$.
Second, recalling from \eqref{eq:lambda_bound} that $\overline{\lambda}=\max\{4L,\Lambda_0g_0^p\}$, we immediately get \eqref{eq:nonconv_rate}. Now, suppose that $\lambda_i \to 0$ as $i \to \infty$. We also obtain from the above inequalities, by keeping $\lambda_i$, that
\[
F(x_0)  - F^* \geq \frac{1}{16}\left(\min_{0 \leq i \leq k-1}g_{i+1}^2\right) \sum_{i=0}^{k-1}\frac{1}{\lambda_i}
\]
and hence
\[
\min_{0 \leq i \leq k-1}g_{i+1}^2 \leq 16(F(x_0)  - F^*)\frac{1}{\sum_{i=0}^{k-1}\frac{1}{\lambda_i} }.
\]
The inequality between arithmetic and geometric means (AM-GM) yields $\sum_{i=0}^{k-1}\frac{1}{\lambda_i} \geq k\left(\prod_{i=0}^{k-1}\frac{1}{\lambda_i}\right)^{1\slash k}$, whence
\[
\min_{0 \leq i \leq k-1}g_{i+1}^2 \leq 16(F(x_0)  - F^*)\frac{1}{k\left(\prod_{i=0}^{k-1}\frac{1}{\lambda_i}\right)^{1\slash k}} = 16(F(x_0)  - F^*)\frac{\left(\prod_{i=0}^{k-1}\lambda_i\right)^{1\slash k}}{k},
\]
and the right-hand side is $\so(1\slash k)$ because the geometric mean converges to zero. By taking the square root, we deduce the $\so(1\slash \sqrt{k})$ rate for $\min_{0\leq i \leq k-1}\norm{F'(x_{i+1})}{*}$.

It remains to estimate the number of Newton steps \eqref{eq:prox_step_alg} and Hessian evaluations. By line 9 of \cref{alg:proximal_newton}, we have for all $k\geq 0$ that $\Lambda_{k+1} = 4^{j_k}\Lambda_k/4$, or equivalently $j_k=1+\log_4\frac{\Lambda_{k+1}}{\Lambda_k}$. Summing these equalities from $0$ to $k$, we obtain that the total number of Newton steps up to the end of iteration $k$ is
\[
N_k=\sum_{i=0}^k \rev{(1+}j_i\rev{)} = \sum_{i=0}^k \left(\rev{2}+\log_4\frac{\Lambda_{i+1}}{\Lambda_i}\right)=\rev{2}(k+1)+\log_4\frac{\Lambda_{k+1}} {\Lambda_0}\overset{\eqref{eq:lambda_bound}}{\leq} \rev{2}(k+1)+\log_4\frac{\overline{\lambda}} {4g_k^p\Lambda_0}.%
\]
Clearly, since we evaluate the Hessian once in $m$ iterations, the total number of Hessian evaluations up to iteration $k$ does not exceed $\lceil \frac{k \rev{+1}}{m} \rceil$.
\end{proof}

\rev{
We will now prove the weaker convergence statement when $H(\cdot)$ is not necessarily positive semidefinite, as stated in \Cref{rem:generalHnonconvex}.
\begin{restatable}{theorem}{thmliminfnonconv}
\label{thm:liminf_nonconv}
Let \cref{ass:basic} hold. Then the sequence $(x_k)_{k \in \mathbb{N}}$ generated by \cref{alg:proximal_newton} is well defined, satisfies $(x_k)_{k\in\mathbb N}\subset\F_0$, and
\begin{equation}
    \liminf_{k\to\infty}\norm{F'(x_k)}{*} = 0.
\end{equation}
\end{restatable}

\begin{proof}
As in the proof of \cref{thm:global_convergence_sublinear_nonconv} above, by using the first part of \cref{lem:acceptance_inner_nonconv} (which does not require positive semidefiniteness of $H$), we find that $(x_k)_{k \in \mathbb{N}} \subset \F_0$ is well defined and
that $(F_k)_{k\in\mathbb N}$ is nonincreasing and bounded below by $F^*$. Telescoping \eqref{eq:functionValuesGradientRelation} gives
\begin{equation}
    \label{eq:liminf_v2_summable}
    \sum_{k=0}^\infty \delta_k \;\le\; F_0 <\infty, \qquad \delta_k\coloneqq F_k-F_{k+1}\ge0,
\end{equation}
together with the pointwise bound
\begin{equation}
    \label{eq:liminf_v2_pointwise}
    g_{k+1}^2 \le 16\,\lambda_k\,\delta_k, \qquad k\ge0.
\end{equation}
We distinguish two exhaustive cases.

\paragraph{Case (a): $j_k\ge1$ for infinitely many $k$.}
Let $(k_n)_{n\in\mathbb N}$ enumerate these indices. By the argument establishing \eqref{eq:correctness_proof_1_1} in the proof of \cref{lem:acceptance_inner_nonconv}, $j_{k_n}\ge1$ implies $\lambda_{k_n}<4L$. Hence, by \eqref{eq:liminf_v2_pointwise},
\[
g_{k_n+1}^2 \le 16\,\lambda_{k_n}\,\delta_{k_n} < 64L\,\delta_{k_n}.
\]
Since $\big(\delta_{k_n}\big)_{n}$ is a subsequence of the summable, hence null, sequence $(\delta_k)_k$ in \eqref{eq:liminf_v2_summable}, $\delta_{k_n}\to0$ as $n\to\infty$, and therefore $g_{k_n+1}\to0$. As $(g_{k_n+1})_n$ is a subsequence of $(g_k)_k$ converging to $0$,
\[
\liminf_{k\to\infty}\norm{F'(x_k)}{*} = \liminf_{k\to\infty} g_k \le \lim_{n\to\infty} g_{k_n+1} =0,
\]
and %
we conclude $\liminf_{k\to\infty}\norm{F'(x_k)}{*}=0$.

\paragraph{Case (b): $j_k=0$ for all $k\ge K_1$, for some $K_1\in\mathbb N$.}
Then, by line~9 of \cref{alg:proximal_newton}, $\Lambda_{k+1}=\Lambda_k/4$ and $\lambda_k=\Lambda_kg_k^p$ for all $k\ge K_1$, so by induction
\begin{equation}
    \label{eq:liminf_v2_lambda_exact}
    \lambda_k = \Lambda_{K_1}\,4^{K_1-k}\,g_k^{\,p}, \qquad k\ge K_1.
\end{equation}
Substituting \eqref{eq:liminf_v2_lambda_exact} into \eqref{eq:liminf_v2_pointwise}, summing and using \eqref{eq:liminf_v2_summable},
\begin{equation}
    \label{eq:liminf_v2_sum2}
    \sum_{k\ge K_1} 4^{\,k-K_1}\,\frac{g_{k+1}^2}{g_k^{\,p}} \;\le\; 16\,\Lambda_{K_1}\sum_{k\ge K_1}\delta_k \;<\;\infty,
\end{equation}
so in particular the summands tend to $0$. Thus,
\begin{equation}
    \label{eq:liminf_v2_terms}
    (k-K_1)\ln4 + 2\ln g_{k+1} - p\ln g_k \longrightarrow -\infty \qquad (k\to\infty).
\end{equation}
Set $y_k\coloneqq\ln g_k$. By \eqref{eq:liminf_v2_terms}, there is $K_2\ge K_1 $ such that
\begin{equation}
    \label{eq:liminf_v2_affine}
    y_{k+1} \le \frac{p}{2}y_k -\frac{(k-K_1)\ln4}{2}, \qquad k\ge K_2.
\end{equation}
Unrolling \eqref{eq:liminf_v2_affine} and using $0\le p/2\le1/2<1$,
\[
\begin{aligned}
y_n
&\le \Big(\frac p2\Big)^{n-K_2}y_{K_2}
-\frac{\ln4}{2}\sum_{j=K_2}^{n-1}
\Big(\frac p2\Big)^{n-1-j}(j-K_1)\\
&\le \Big(\frac p2\Big)^{n-K_2}y_{K_2}
-\frac{\ln4}{2}(n-1-K_1)
\longrightarrow-\infty \quad (n\to\infty),
\end{aligned}
\]
so $g_n=e^{y_n}\to0$, and thus $\liminf_{k\to\infty}\norm{F'(x_k)}{*}=0$.

Since cases (a) and (b) are exhaustive, the claim follows in either case. \qedhere
\end{proof}
}

\subsection{Non-convex problems under PL condition}
\label{sec:global_proofs_PL}

Next, we consider the (possibly non-convex) setting under the PL condition. We start

\thmglobalconvergencenonconvPL*
\begin{proof}
\begin{enumerate}[label=(\roman*), wide, labelwidth=!, labelindent=0pt]\itemsep=0em 
\item \rev{
We first prove the bound \eqref{eq:boundlambdakunderPL}. For that purpose, we divide the set of iterates into two disjoint sets
\begin{align*}
    \mathcal{I}_1:= \{ k \in \mathbb{N}: \text{$k=0$ or $j_k>0$} \}, 
    \quad
    \mathcal{I}_2:= \{ k \in \mathbb{N}: \text{$k>0$ and $j_k=0$} \}.
\end{align*}
Consider the case $k \in \mathcal{I}_1$. If $k = 0$, we trivially have $\lambda_k \leq \overline{\lambda}$ and the bound follows. If $j_k > 0$, we argue as follows. First note that in order to have the inequalities \eqref{eq:correctness_proof_1_1} and \eqref{eq:lem:acceptance_inner_nonconv_1_proof_22} in the proof of  \Cref{lem:acceptance_inner_nonconv}, we need only the well posedness of the subproblems \eqref{eq:prox_step_alg}, and not the positive semidefiniteness $H(x_k)\succeq0$. As $j_k > 0$, $\lambda_k\slash 4$ was not accepted, hence by \eqref{eq:correctness_proof_1_1} and \eqref{eq:lem:acceptance_inner_nonconv_1_proof_22}, we must have $\lambda_k\slash 4 \leq L$. Thus we have shown
$$\lambda_k \leq \overline{\lambda}= \max\{ 4L, \Lambda_0 g_0^p\} \qquad \text{for all $k \in \mathcal{I}_1$}.
$$
Now consider the case $k \in \mathcal{I}_2$. Whenever iteration $k$ in this regime is given, we can find an index $n \in \mathbb{N}$ with $n < k$ and $n \in \mathcal{I}_1$ which happened last before $k$ (i.e., the last index from $\mathcal{I}_1$ before $k$). From that choice we also have
\begin{equation}
j_{n+1} = \ldots = j_{n+(k-n)} = j_k = 0 \label{eq:boundedness_lambda_k_PL_eq1}.    
\end{equation}
Choose now
$$
l^\star\in
\arg\max_{0 \leq l\leq k-n}\lambda_{n+l},
$$
that is, $l^\star$ is such that $\lambda_{n+l^\star}$ is the biggest element of $\{\lambda_n, \lambda_{n+1}, \hdots, \lambda_k\}.$
If $l^\star=0$, we recall that $n \in \mathcal{I}_1$ and therefore as we already showed $\lambda_{n+l^\star} \leq \overline{\lambda}$. 

If instead $l^\star>0$, we conclude
\begin{equation*}
\lambda_{n+l^\star}
= \Lambda_{n+l^\star}g_{n+l^\star}^p
= \frac{\lambda_n}{4^{l^\star}}
  \left(\frac{g_{n+l^\star}}{g_n}\right)^p,
\end{equation*}
where we used the update
$\Lambda_{n+1}=\lambda_n/(4g_n^p)$
and $j_{n+1}=\cdots=j_{n+l^\star}=0$.
We further estimate using \eqref{eq:functionValuesGradientRelation}
\begin{equation*}
    g_{n+l^{\star}} \overset{\eqref{eq:functionValuesGradientRelation}}{\leq} \sqrt{16\lambda_{n+l^{\star}-1} (F_{n+l^{\star}-1} - F_{n+l^{\star}})}  \overset{(a)}{\leq}
    4\sqrt{\lambda_{n+l^{\star}}F_{n}} \overset{(PL)}{\leq} \frac{4g_n}{\sqrt{2\mu}}\sqrt{\lambda_{n+l^{\star}}},
\end{equation*}
where in (a) the monotonicity of $(F_k)_{k \in \mathbb{N}}$ and also the maximum property of $l^\star$ were used. Combining the two latter inequalities gives 
\begin{equation*}
     \lambda_{n+l^{\star}} \leq  \frac{\lambda_n}{4^{l^{\star}}} \left(\frac{8}{\mu}\right)^{\frac{p}{2}} \lambda_{n+l^{\star}}^{\frac{p}{2}}.
\end{equation*}
Rearranging yields, using that $1- p/2>0$ and $l^{\star}\geq 1$, 
\begin{equation*}
    \lambda_{n+l^{\star}}
    \overset{(n \in \mathcal{I}_1)}{\leq} \left(\frac{\lambda_n}{4} \right)^{\frac{2}{2-p}} \left(\frac{8}{\mu}\right)^{\frac{p}{2-p}} 
    \leq \left(\frac{\overline{\lambda}}{4} \right)^{\frac{2}{2-p}} \Bigg(\frac{8}{\mu}\Bigg)^{\frac{p}{2-p}}
    =
    \left(\frac{\overline{\lambda}}{4} \right) \Bigg(\frac{2\overline{\lambda}}{\mu}\Bigg)^{\frac{p}{2-p}}.
\end{equation*}
Combining all cases gives the bound in \eqref{eq:boundlambdakunderPL}.
}

\item  
\rev{The sequence $(x_k)_{k \in \mathbb{N}} \subset \F_0$ is well defined as explained in the proof of \cref{thm:liminf_nonconv} above.}
Applying the PL condition \eqref{eq:PL}, we get
\begin{equation*}
     F_k - F_{k+1}  \overset{\eqref{eq:functionValuesGradientRelation}}{\geq}  \frac{1}{16\lambda_k}g_{k+1}^2  \geq \frac{1}{16\lambda_k}  {\rm dist}(0, \partial F(x_{k+1}))^2\overset{\eqref{eq:PL}}{\geq}  \frac{1}{16\lambda_k} \cdot 2\mu F_{k+1}. 
\end{equation*}
Rearranging, we obtain
\begin{equation}
\label{eq:thm:global_convergence_nonconv_PL_proof_-1.5}
    F_{k+1}  \leq \frac{1}{1+\frac{\mu}{8\lambda_k}}F_k=\frac{8\lambda_k}{\mu+8\lambda_k}F_k=\left(1-\frac{\mu}{\mu+8\lambda_k}\right)F_k.
\end{equation}
We immediately see that if $\lambda_k \to 0$ as $k\to \infty$, then $1-\frac{\mu}{\mu+8\lambda_k}\to 0$ and we obtain superlinear convergence. Using the upper bound $\lambda_k\leq \overline{\lambda}_{\rm PL}$  in \eqref{eq:boundlambdakunderPL}, we obtain 
\[
    \left(1-\frac{\mu}{\mu+8\lambda_k}\right) \leq \left(1-\frac{\mu}{\mu+8\overline{\lambda}_{\rm PL}}\right) \leq \exp\left(-\frac{\mu}{\mu+8\overline{\lambda}_{\rm PL}}\right).
\]
This and \eqref{eq:thm:global_convergence_nonconv_PL_proof_-1.5}, by induction, give \eqref{eq:nonconv_PL_rate}.

\item 
Let $\sigma > 1$. Then $t \mapsto t^{1/\sigma}$ is concave and we have that
\begin{align}
    F_{k}^{1/\sigma} -F_{k_{}+1}^{1/\sigma} \geq \frac{1}{\sigma F_{k_{}}^{1-1/\sigma}}(F_{k_{}} - F_{k_{}+1}) \overset{\eqref{eq:stop_cond_alg_new_2}}{\geq} \frac{ \lambda_{k_{}} r_{k_{}}^2}{4\sigma F_{k_{}}^{1-1/\sigma}}. 
\label{eq:thm:global_convergence_nonconv_PL_proof_1_tau_ver}
\end{align}
By the PL condition \eqref{eq:PL} and since $1-1/\sigma>0$, we infer $F_k^{1 - 1/{\sigma}} \leq  (g_k^2/(2\mu))^{1 - 1/\sigma}=(2\mu)^{-(1-1/\sigma)} g_k^{1 + \frac{\sigma-2}{\sigma}}$ and hence
\begin{align*}
    F_{k_{}}^{1/\sigma} -F_{k_{}+1}^{1/\sigma} \overset{\eqref{eq:thm:global_convergence_nonconv_PL_proof_1_tau_ver},\eqref{eq:PL}}{\geq} \frac{  (2\mu)^{(1-1/\sigma)} \lambda_{k_{}} r_{k_{}}^2}{4 \sigma g_k^{1 + \frac{\sigma-2}{\sigma}}}  \overset{\eqref{eq:g_k_leq_r_k}}{\geq}  \frac{  (2\mu)^{(1-1/\sigma)} \lambda_{k_{}} r_{k_{}}^2}{8 \sigma g_k^{\frac{\sigma-2}{\sigma}} \lambda_{k-1}r_{k-1}}.
\end{align*}
By \rev{(i) and (iv) of \cref{ass:basic}, \eqref{eq:F_div_def} and \eqref{eq:boundlambdakunderPL},}  $(x_k)_{k \in \mathbb{N}} \subset \F_0$, and boundedness of $\mathcal{F}_0$, there exists a constant $C >0$ such that $g_k \leq C$ for all $k\geq 0$.
Combining this with line 9 of \cref{alg:proximal_newton}, we obtain $\lambda_k\geq \Lambda_k g_k^p=\lambda_{k-1}g_k^p/(4g_{k-1}^p) \geq \lambda_{k-1}g_k^p/(4C^p)$ and hence, choosing $\sigma = 2\slash(1-p)$ (so that $p=(\sigma-2)\slash \sigma$), we obtain
\begin{align}
    F_{k_{}}^{1/\sigma} -F_{k_{}+1}^{1/\sigma} 
    {\geq}
\frac{  (2\mu)^{(1-1/\sigma)} g_k^p r_{k_{}}^2}{32 \sigma C^p g_k^{\frac{\sigma-2}{\sigma}} r_{k-1}}=\frac{  (2\mu)^{(1-1/\sigma)} r_{k_{}}^2}{32 \sigma  C^p  r_{k-1}}.\label{eq:tau_ew_new}
\end{align}
Rearranging and using that $r_k$ is bounded by $D_0$ according to \eqref{eq:bounded_sublevel_set} and $(x_k)_{k \in \mathbb{N}} \subset \F_0$,
we obtain
\begin{align}
r_{k_{}}^2 \leq \frac{32 \sigma  C^p  D_0}{(2\mu)^{(1-1/\sigma)}}    (F_{k_{}}^{1/\sigma} -F_{k_{}+1}^{1/\sigma} ) \overset{\eqref{eq:nonconv_PL_rate}}{\to} 0 \quad \text{as $k \to \infty$}.
    \label{eq:tau_ew_new_1}
\end{align}
Denote for simplicity $C_1^{-1} := \frac{  (2\mu)^{(1-1/\sigma)}}{32 \sigma  C^p}$. Rearranging \eqref{eq:tau_ew_new} and using the inequality $2\sqrt{ab} \leq a+b$ that holds for $a,b\geq0$, we get
\begin{align*}
    r_{k} &\leq  \sqrt{C_1r_{k-1} (F_{k_{}}^{1/\sigma} -F_{k_{}+1}^{1/\sigma})}
    \leq 
    \frac{r_{k-1}}{2} + \frac{C_1(F_{k_{}}^{1/\sigma} -F_{k_{}+1}^{1/\sigma})}{2}.
\end{align*}
Summing from $l+1 \geq 1$ to $n \geq l+1$ yields 
\begin{align*}
        \sum_{k = l+1}^n r_{k} \leq \frac{1}{2}\sum_{k=l}^{n-1}r_k +  \frac{C_1  (F_{l_{}+1}^{1/\sigma} -F_{n_{}+1}^{1/\sigma})}{2}.
\end{align*}
Putting the sums on one side and using that $r_k\geq 0$ and triangle inequality, we obtain
\begin{align}
        \frac{1}{2}\norm{x_{l+1}-x_{n}}{}\leq \frac{1}{2}\sum_{k = l+1}^{n-1} r_{k} \leq r_l +  \frac{C_1(F_{l_{}+1}^{1/\sigma} -F_{n_{}+1}^{1/\sigma})}{2} \to 0 \quad \text{as $l,n \to \infty$,} \label{eq:main_thm_PL_5}
\end{align}
where we used that $r_k\to 0$ by \eqref{eq:tau_ew_new_1} and $F_k \to 0$ by \eqref{eq:nonconv_PL_rate}. Thus, we obtain that $(x_k)_{k \in \mathbb{N}} $ is a Cauchy sequence and hence strongly converges to a point $x^* \in \HH$.

 \rev{From \eqref{eq:functionValuesGradientRelation} and part (i), we have
 \begin{equation*}
     \frac{1}{16\overline{\lambda}_{\rm PL}}g_{k+1}^2 \leq F(x_{k}) - F(x_{k+1}) \to 0   \quad (k \to  \infty). 
 \end{equation*} 
 Hence  $f'(x_{k})+\psi'(x_{k}) = F'(x_k) \to 0$ as $k\to \infty$. Since $f$ is continuously Fréchet differentiable, we have that $f'(x_{k}) \to f'(x^*)$.
 }

 Consequently $\psi'(x_k) = (\psi'(x_{k}) + f'(x_{k})) - f'(x_{k}) \to - f'(x^*)$. Since $\psi$ is convex and lower semicontinuous, $x\mapsto \partial \psi(x)$ has a closed graph and we deduce that $-f'(x^*) =  \lim_{k\to \infty} \psi'(x_{k}) \in \partial \psi(x^*)$. Thus, equivalently, we have that $0 \in \partial F(x^*) =  f'(x^*) + \partial \psi(x^*)$ and $x^*$ is a stationary point.
Further, the functional $F$ is a sum of continuously differentiable $f$ and convex proper lower semicontinuous $\psi$. Thus, $F$ is lower semicontinuous, which implies that $x^*$ is a global minimizer:
\begin{equation*}
    F^* \overset{\eqref{eq:nonconv_PL_rate}}{=} \lim_{k \to \infty}F(x_{k}) = \liminf_{k \to \infty} F(x_{k}) \geq F(x^*).
\end{equation*}

Next, we show the linear convergence of $(x_k)_{k \in \mathbb{N}} $.  Fixing some $l \geq 0$ and arbitrary $n \geq l+2$, we obtain using the triangle inequality that 
    \begin{align}
        \norm{x_{l+1} - x_n}{}
        \leq \sum_{k = l+1}^{n-1} r_{k} 
        \overset{\eqref{eq:main_thm_PL_5}}&{\leq} 2r_l +  C_1(F_{l_{}+1}^{1/\sigma} -F_{n_{}+1}^{1/\sigma}) \label{eq:need_for_useful_bounds} %
         \\
        \overset{\eqref{eq:tau_ew_new_1}}&{\leq}  2\sqrt{C_1D_0(F_{l}^{1/\sigma} - F_{l+1}^{1/\sigma})} 
        +   C_1(F_{l_{}+1}^{1/\sigma} -F_{n_{}+1}^{1/\sigma}).\label{eq:main_thm_PL_6}
    \end{align}
    Sending $n\to \infty$ in \eqref{eq:main_thm_PL_6}, we obtain 
    \begin{equation}
    \norm{x_{l+1} - x^*}{}   \leq 2\sqrt{C_1D_0F_{l-1}^{1/\sigma}} +   C_1F_{l_{}-1}^{1/\sigma}, \label{eq:main_thm_PL_7} 
    \end{equation}
    where we used that $0\leq F_{l+1} \leq F_l \leq F_{l-1}$ by \eqref{eq:F_k_decr}. Thus, the linear convergence of $x_l$  follows from the linear convergence of  $F_{l}$ \eqref{eq:nonconv_PL_rate}. In the same way, superlinear convergence when $\lambda_l \to 0$ as $l \to \infty$ follows since then $F_{l}$ converges to 0 superlinearly, see \eqref{eq:thm:global_convergence_nonconv_PL_proof_-1.5}.

It remains to show the convergence of $\norm{F'(x_k)}{*}$ to $0$. We have
\begin{equation}
    \norm{F'(x_k)}{*}\overset{\eqref{eq:g_k_leq_r_k}}{\leq} 2\lambda_{k-1}\norm{x_{k-1}-x_{k}}{} \leq 2\lambda_{k-1}(\norm{x_{k-1}-x^*}{}+\norm{x^*-x_{k}}{}). \label{eq:PL_proof_lastinequality} %
\end{equation}
Hence, linear convergence follows from the bound $\lambda_{k-1}\leq\overline{\lambda}_{\rm PL}$ in \eqref{eq:boundlambdakunderPL} and linear convergence of $x_k$ to $x^*$; superlinear convergence follows from $\lambda_k\to0$ and superlinear convergence of $x_k$ to $x^*$ in that case. 

\rev{To be precise if $\lambda_k \to 0$ the R-superlinear convergence of $x_k \to x^*$ and $\norm{F'(x_k)}{*} \to 0$ follows directly from \eqref{eq:main_thm_PL_7}, \eqref{eq:PL_proof_lastinequality} and the Q-superlinear convergence of $F_k \to 0$. In addition, the Q-superlinear convergence of $g_k \to 0$ follows from 
\begin{equation*}
    g_{k+1}^2 \overset{\eqref{eq:functionValuesGradientRelation}}{\leq} 16 \lambda_k F_k \overset{ \eqref{eq:PL}}{\leq} \frac{8 \lambda_k}{\mu} g_k^2.
\end{equation*}}

\rev{Finally, the claim regarding the number of Newton steps and Hessian evaluations can be shown identically as in the proof of \cref{thm:global_convergence_sublinear_nonconv}.}

\qedhere
\end{enumerate}    
\end{proof}

\subsection{Convex problems}
\label{sec:global_proofs_convex}

We now proceed to the proof of the convergence rate in the convex setting. By \cref{thm:global_convergence_sublinear_nonconv}, we have that 
$(x_k)_{k \in \mathbb{N}} \subset \mathcal{F}_0$ is well defined. 
This, together with convexity of $F$ and \eqref{eq:bounded_sublevel_set} (which we assume in this setting)  imply that, for any $\bar x \in \mathcal{F}_0$ and $k\geq 0$, we have
    \begin{equation}
    \label{eq:convexity_id} 
        F(x_k) - F(\bar x) \leq \langle F'(x_{k}), x_{k}-\bar x \rangle \leq %
        g_{k}D_0.   
    \end{equation}

\thmglobalconvergencesublinear*
\begin{proof}  

We start with the following chain of inequalities for any iteration $k\geq0$:
\begin{align*} 
    \frac{1}{F_{k+1}}-\frac{1}{F_k} &= \frac{F_k - F_{k+1}}{F_kF_{k+1}}
    \overset{\eqref{eq:functionValuesGradientRelation}}{\geq} \frac{g_{k+1}^2}{16\lambda_kF_kF_{k+1}} 
    = \frac{g_k^2}{16\lambda_k F_kF_{k+1}}\left(\frac{g_{k+1}}{g_k}\right)^2 \\
        \overset{\eqref{eq:convexity_id}}&{\geq} \frac{F_k^2}{16D_0^2\lambda_k F_kF_{k+1}}\left(\frac{g_{k+1}}{g_k}\right)^2 \geq \frac{1}{16D_0^2\lambda_k}\left(\frac{g_{k+1}}{g_k}\right)^2,  
\end{align*}
where in the last step we used that, by \eqref{eq:F_k_decr}, $F_k \geq F_{k+1}$.
Summing these inequalities for $j$ from $0$ to $k-1$ and using the inequality between arithmetic and geometric means (AM-GM), we obtain 
\begin{align*}
    \frac{1}{F_{k}}-\frac{1}{F_0}  
        &\geq \frac{1}{16D_0^2}\sum_{j=0}^{k-1}\frac{1}{\lambda_j}\left(\frac{g_{j+1}}{g_j}\right)^2 
        \geq \frac{1}{16D_0^2}k\left(\prod_{j=0}^{k-1} \frac{1}{\lambda_j}\left(\frac{g_{j+1}}{g_j}\right)^2\right)^{1/k}\\
        &= \frac{k}{16D_0^2} \left(\frac{g_{k}}{g_0}\right)^{2/k}\left(\prod_{j=0}^{k-1} \frac{1}{\lambda_j}\right)^{1/k}\\
         \overset{\eqref{eq:convexity_id}}&{\geq} \frac{ k}{16D_0^2}\left(\frac{F_{k}}{D_0g_0}\right)^{2/k}\left(\prod_{j=0}^{k-1}\lambda_j\right)^{-1/k} 
         = \frac{  k}{16D_0^2}\left(\frac{D_0g_0}{F_{k}}\right)^{-2/k}\left(\prod_{j=0}^{k-1}\lambda_j\right)^{-1/k}\\
         &= \frac{  k}{16D_0^2}\exp\left(-\frac{2}{k}\ln\left(\frac{D_0g_0}{F_{k}}\right)\right)\left(\prod_{j=0}^{k-1}\lambda_j\right)^{-1/k}\\ 
         &\geq \frac{  k}{16D_0^2}\left(1-\frac{2}{k}\ln\left(\frac{D_0g_0}{F_{k}}\right)\right)\left(\prod_{j=0}^{k-1}\lambda_j\right)^{-1/k},
\end{align*}
where in the last step we used that $\exp(x) \geq 1+x$.  
Consider two cases. If $\frac{2}{k}\ln\left(\frac{D_0g_0}{F_{k}}\right) \geq \frac 12$, we obtain $F_k \leq g_0D_0\exp\left(-\frac{k}{4}\right)$. Otherwise, if $\frac{2}{k}\ln\left(\frac{D_0g_0}{F_{k}}\right) < \frac 12$, we obtain
\begin{align}
\frac{1}{F_k} \geq  \frac{  k}{32D_0^2}\left(\prod_{j=0}^{k-1}\lambda_j\right)^{-1/k} \quad\Longleftrightarrow \quad F_k \leq  \frac{32D_0^2}{  k}\left(\prod_{j=0}^{k-1}\lambda_j\right)^{1/k}.
\label{eq:conv_rate_product}
\end{align}
Combining these two cases, we obtain
\[
 F(x_k) - F^* = F_k \leq  g_0D_0\exp\left(-\frac{k}{4}\right) + \frac{32D_0^2}{  k}\left(\prod_{j=0}^{k-1}\lambda_j\right)^{1/k}.
\]
Applying the bound $\lambda_k \leq \overline{\lambda} $ in \eqref{eq:lambda_bound}, we obtain the convergence rate result \eqref{eq:conv_rate}. Further, if $\lambda_k \to 0$ as $k\to \infty$, we obtain that $\left(\prod_{j=0}^{k-1}\lambda_j\right)^{1/k} \to 0$ since $\lambda_k\geq 0$ for $k\geq0$, and hence $F(x_k)-F^*=\so(1\slash k)$. \\

Finally, let us now consider the case when $f$ is locally $C^2$, $F$ is strictly convex, $H:=\nabla^2 f(x)$, and $\dim(\HH) <  \infty$ and show that $\lambda_k\to 0$ (as claimed in the parentheses in the theorem). As already mentioned $(x_k)_{k \in \mathbb{N}} \subset \mathcal{F}_0$ and $\mathcal{F}_0$ is bounded. Thus, since $\dim(\HH) <  \infty$, we have that there exists a convergent subsequence
$(x_{k_j})_{j\in \mathbb{N}}$ such that $x_{k_j}\to x^*$. By \cref{thm:global_convergence_sublinear_nonconv} we have that $\norm{F'(x_k)}{*} \to 0$. Hence, by the convexity of $F$ and closedness of the convex subdifferential, it follows that $x^*$ is a stationary point of $F$. Moreover, strict convexity implies that $x^*$ is the unique minimizer of $F$, and a standard contradiction argument shows that the whole
sequence $(x_k)_{k \in \mathbb{N}}$ converges to $x^*$. Below, as a part of the proof of \cref{thm:global_convergence_c2}, we show that if $x_k\to x^*$, $f$ is locally $C^2$ around $x^*$, $H:=\nabla^2 f(x)$, and $\dim(\HH) <  \infty$, then $\lambda_k\to0$.
\end{proof}

\section{Proofs for \cref{sec:fast} on superlinear convergence}

In this section we shall frequently use the local strong convexity introduced in \cref{ass:strong_convexity}. Let us recall that, as stated at the start of \cref{sec:fast}, we already have $x_k \to x^*$ and for $k \geq K$, $x_k$ belongs to the ball of strong convexity $B_R(x^*)$ and $H(x_k) \succeq 0$. Note that $H(x_k) \succeq 0$ for sufficiently large $k$ implies that 
$H(x_{\pi(k)}) \succeq 0$ for sufficiently large $k$. W.l.o.g., we further assume that $H(x_{\pi(k)}) \succeq 0$ for $k \geq K$ for the same $K$. Note that, for the sake of potential generalizations, in some of the next lemmas we list only necessary subassumptions of assumptions made in \cref{sec:fast} and in particular in \cref{ass:strong_convexity}. 

\subsection{Technical lemmas}

Our first steps are to establish several relations between points $x_k$ (the point from which we make a step at iteration $k$), $x_{k+1}=x_+(\lambda_k, x_k, x_{\pi(k)})$ (the point accepted at iteration $k$), $x_{k,+}=x_+(\lambda_k/4,x_k)=x_+(\lambda_k/4, x_k, x_{\pi(k)})$ (the last non-accepted point in the case of nontrivial linesearch at iteration $k$, i.e., if $j_k>0$), and $x^*$ (the point to which $x_k$ converges by the assumption made at the beginning of \cref{sec:fast}).

    \begin{lem}\label{lem:useful_bounds_new_strong}%
       Let $(x_k)_{k \in \mathbb{N}}$ be the iterates of \cref{alg:proximal_newton}. Let \cref{ass:strong_convexity} hold and for all $k\geq K$, $x_k \in B_R(x^*)$. 
       Then, there exist constants $c_1,c_2>0$ such that, for all $k\geq K$,
    \begin{equation}
        \norm{x_{k+1}-x^*}{} \leq c_1 \| x_{k+1}- x_k\| \quad \text{and} \quad  \norm{x_k-x^*}{} \leq c_2 \|x_{k+1} - x_k\|. \label{eq:useful_bounds_new}
    \end{equation}
     We also have, for $0<\lambda < \lambda_k$ and $k\geq K$, %
         \begin{align}
             \| x_{+}(\lambda,x_k,x_{\pi(k)}) - x^*\|  &\leq
             \frac{(2c_1+1)\lambda_k -(1+c_1)\lambda}{\lambda_k}\| x_{+}(\lambda,x_k,x_{\pi(k)})  - x_k\|,\label{eq:good_bound_1}\\
             \norm{x_k-x^*}{} &\leq  \frac{(2c_1+2)\lambda_k -(1+c_1)\lambda}{\lambda_k}\| x_{+}(\lambda,x_k,x_{\pi(k)})  - x_k\|.\label{eq:good_bound_2}
         \end{align}
    \end{lem}

Note that we will mainly use the last two inequalities with $\lambda=\lambda_k/4$, i.e., when $x_{+}(\lambda,x_k,x_{\pi(k)})=x_{+}(\lambda_k/4,x_k,x_{\pi(k)})=x_{k,+}$.
    \begin{proof}
\rev{First, let us note that due to the fact that $H(x_k)$ is positive semidefinite only for $k \geq K$ (see \cref{ass:strong_convexity} and the discussion afterwards)  we have the following analogue of \eqref{eq:lambda_bound}:
\begin{equation}
    \label{eq:lambda_bound_eventual}
    \lambda_k\leq \overline{\lambda}_{\rm loc} \coloneqq \max\{4L,\Lambda_0g_0^p, \lambda_K\}, \qquad k \geq K.
\end{equation}
This can be proved with a similar induction argument as for \eqref{eq:lambda_bound}. Indeed, assume that 
\eqref{eq:lambda_bound_eventual} holds for iteration $k-1 \geq K$ and consider iteration $k > K$. If $j_k=0$, we have 
\begin{equation*}
\lambda_k=\Lambda_kg_k^p = \frac{4^{j_{k-1}}\Lambda_{k-1}}{4}g_k^p = \frac{\lambda_{k-1}}{4g_{k-1}^p}g_k^p    \overset{\eqref{eq:g_k_leq_g_k-1}}{\leq}   \frac{\lambda_{k-1}2^p}{4} \overset{p\leq 1}{\leq} \frac{\lambda_{k-1}}{2} \leq \overline{\lambda}_{\rm loc} ,    
\end{equation*}
where we used that $2^p/4\leq 1/2$ for any $p\in [0,1]$ and the induction hypothesis $\lambda_{k-1} \leq\overline{\lambda}_{\rm loc} $.
If $j_k>0$, we have that $\lambda_k\slash 4 \leq L$ since it was not accepted. Hence, $\lambda_k \leq 4L \leq \overline{\lambda}_{\rm loc}$.}

Trivial consequences of strong convexity (which we can use since for $k\geq K$, $x_k$ belongs to the ball of strong convexity) imply 
\begin{equation}
    \mu \norm{x_{k+1}-x^*}{}^2 \leq 2(F(x_{k+1})-F(x^*)) = 2F_{k+1} \leq g_{k+1}^2\slash \mu \overset{\eqref{eq:g_k_leq_r_k}}{\leq} 4\lambda_k^2r_k^2/\mu, \label{eq:GQ+isolated}
    \end{equation}
whence, by \eqref{eq:lambda_bound_eventual}, for $k\geq K$,
\begin{equation}
\label{eq:discussion_1}
    \norm{x_{k+1}-x^*}{} %
            \leq \frac{2 \lambda_k}{\mu} r_k = \frac{2 \lambda_k}{\mu}\norm{x_k-x_{k+1}} \leq %
            \frac{2\overline{\lambda}_{\rm loc}}{\mu} \norm{x_k-x_{k+1}}{},
\end{equation}
which is the first inequality in \eqref{eq:useful_bounds_new} with $c_1=\frac{2\overline{\lambda}_{\rm loc}}{\mu} $. The second inequality in \eqref{eq:useful_bounds_new} follows with $c_2=1+\frac{2\overline{\lambda}_{\rm loc}}{\mu}$ via the triangle inequality.

To show \eqref{eq:good_bound_1}, applying two times the triangle inequality, one time \eqref{eq:useful_bounds_new} and one time \cref{lem:lambda_dependence} (which we can apply since $H(x_{\pi(k)}) \succeq 0$ for $k \geq K$ by \cref{ass:strong_convexity}), recalling from \cref{alg:proximal_newton} and \eqref{eq:PLMSN_step} that $x_{k+1}=x_+(\lambda_k, x_k, x_{\pi(k)})$, we obtain
\begin{align*}
    &\| x_{+}(\lambda,x_k,x_{\pi(k)}) - x^*\| \leq \| x_{+}(\lambda,x_k,x_{\pi(k)}) - x_{k+1}\| + \|x_{k+1} -x^*\|\\
    \overset{\eqref{eq:useful_bounds_new}}&{\leq} \| x_{+}(\lambda,x_k,x_{\pi(k)}) - x_{k+1}\| + c_1\|x_{k+1} -x_k\|\\
    &\leq \| x_{+}(\lambda,x_k,x_{\pi(k)}) - x_{k+1}\| + c_1(\| x_{+}(\lambda,x_k,x_{\pi(k)}) - x_{k+1}\|+\|x_{+}(\lambda,x_k,x_{\pi(k)}) -x_k\|)\\
    &=(1+c_1)\| x_{+}(\lambda,x_k,x_{\pi(k)}) - x_{k+1}\|+c_1\|x_{+}(\lambda,x_k,x_{\pi(k)}) -x_k\|)\\
   \overset{\eqref{eq:lambda_dependence}}&{\leq}  (1+c_1)\frac{\lambda_k-\lambda}{\lambda_k}\|x_{+}(\lambda,x_k,x_{\pi(k)})  - x_k\| + c_1\|x_{+}(\lambda,x_k,x_{\pi(k)}) -x_k\|, 
\end{align*}
which is \eqref{eq:good_bound_1}.
Applying once again the triangle inequality, we obtain \eqref{eq:good_bound_2}:
\begin{align*}
    \norm{x_k-x^*}{} &\leq \norm{x_k-x_+(\lambda, x_k,x_{\pi(k)})}{} + \norm{x_+(\lambda, x_k,x_{\pi(k)})-x^*}{}\\
    \overset{\eqref{eq:good_bound_1}}&{\leq}   \norm{x_k-x_+(\lambda, x_k,x_{\pi(k)})}{}  + \frac{(2c_1+1)\lambda_k -(1+c_1)\lambda}{\lambda_k}\| x_{+}(\lambda,x_k,x_{\pi(k)})  - x_k\|.
\end{align*}
\end{proof} 

\begin{lem}\label{lem:iterates_g_k_and_lambda_k}   
    Let $(x_k)_{k \in \mathbb{N}}$ be the iterates of \cref{alg:proximal_newton}. Let \cref{ass:strong_convexity} hold and for all $k\geq K$ $x_k \in B_R(x^*)$. If $\lambda_k \to 0$, then for $k$  sufficiently large,
    \begin{align}
        \norm{x_{k+1}-x^*}{} &\leq \frac{4\lambda_k}{\mu} \norm{x_k-x^*}{},\label{eq:lem_iterates_1}\\
        g_{k+1} &\leq \frac{4\lambda_k}{\mu}g_k. \label{eq:lem_iterates_2}
    \end{align}
\end{lem}
\begin{proof} 
Applying \eqref{eq:discussion_1} and the triangle inequality, we get, for $k\geq K$
\[
\norm{x_{k+1}-x^*}{}  \leq  \frac{2 \lambda_k}{\mu}\norm{x_k-x_{k+1}}{}  \leq \frac{2 \lambda_k}{\mu} \left(  \norm{x_{k+1}-x^*}{}  +  \norm{x_k-x^*}{} \right).
\]
Since $\lambda_k \to 0$, for $k$ large enough we have $2\lambda_k/\mu\leq 1/2$ and, hence, we get \eqref{eq:lem_iterates_1}:
\[
\norm{x_{k+1}-x^*}{}  \leq  \frac{4 \lambda_k}{\mu}   \norm{x_k-x^*}{} .
\]
By \eqref{eq:GQ+isolated} we have
\begin{equation}
\norm{x_{k+1}-x^*}{} \leq g_{k+1}/\mu, \qquad \norm{x_{k}-x^*}{} \leq g_{k}/\mu.\label{eq:to_refer_to}    
\end{equation}
Combining this with \eqref{eq:g_k_leq_r_k} and triangle inequality, we obtain, 
\begin{align*}
    g_{k+1} & \overset{\eqref{eq:g_k_leq_r_k}}{\leq} 2\lambda_k r_k  \leq 2\lambda_k\norm{x_{k+1}-x^*}{} + 2\lambda_k\norm{x_{k}-x^*}{} \overset{\eqref{eq:to_refer_to}}{\leq} \frac{2}{\mu}\lambda_kg_{k+1} + \frac{2}{\mu} \lambda_kg_k.
\end{align*}
We obtain \eqref{eq:lem_iterates_2} for the same $k$ large enough s.t. $2\lambda_k/\mu\leq 1/2$.
\end{proof}

The next result establishes the convergence of non-accepted trial points $x_{k,+}=x_+(\lambda_k/4,x_k,x_{\pi(k)})$ to $x^*$ when accepted points $x_k$ converge to $x^*$. Note that the condition $H(x_{\pi(k)}) \succeq 0 $ 
for $k \geq K$ is guaranteed by Assumption 2.
\begin{lem}\label{lem:strange_convergence}
Let $(x_k)_{k \in \mathbb{N}}$ be the iterates of \cref{alg:proximal_newton} and $x_k \to x^*$ as $k \to \infty$. Let also $H(x_{\pi(k)}) \succeq 0 $ 
for $k \geq K$. Then, $x_{k,+}=x_+(\lambda_k/4,x_k,x_{\pi(k)}) \to x^*$ as $k \to \infty$.
\end{lem}
\begin{proof}
For any $k \geq K$ and $0<\lambda \leq \lambda_k$, by the triangle inequality, the definition $x_{k+1}=x_+(\lambda_k,x_k,x_{\pi(k)})$ in line 9 of \cref{alg:proximal_newton} and \eqref{eq:PLMSN_step}, and \cref{lem:lambda_dependence}, we have
 \begin{align*}
 \| x_k -x_+(\lambda,x_k,x_{\pi(k)})\| &\leq \|x_{k+1} - x_+(\lambda,x_k,x_{\pi(k)})\| + \|x_{k+1} - x_k\| \\&\overset{\eqref{eq:lambda_dependence}}{\leq} \frac{\lambda_k-\lambda}{\lambda_k}\| x_k -x_+(\lambda,x_k,x_{\pi(k)})\| + \|x_{k+1} - x_k\|.
 \end{align*}
Rearranging, we obtain
\begin{equation}
\label{eq:r_k+vsr_k}
    \| x_k -x_+(\lambda,x_k,x_{\pi(k)})\| \leq \frac{\lambda_k}{\lambda} \| x_{k+1} -x_k\|.
\end{equation}
Now, utilizing again the triangle inequality, we calculate
\begin{align*}
    \|x_+(\lambda,x_k,x_{\pi(k)}) - x^*\| &\leq \|x_+(\lambda,x_k,x_{\pi(k)})-x_{k}\| + \|x_{k} -x^*\|\leq \frac{\lambda_k}{\lambda} \| x_{k+1} -x_k\| + \|x_{k} -x^*\|.
\end{align*}
The claim follows by setting  $\lambda = \lambda_k/4$
and passing to the limit.%
\end{proof}

\subsection{Refined upper bound on $\lambda_k$}
\label{sec:lambda_refined_proof}

The next technical lemma is a step towards proving \cref{prop:lambda_k_to_zero} with a refined upper bound on $\lambda_k$, and this lemma establishes an upper bound on $\lambda_k$ in the particular case when the linesearch at iteration $k$ is nontrivial, i.e., $j_k>0$. 

\begin{lem}
\label{lem:helper_lemma_lambda_k_jk_pos}
Let $(x_k)_{k \in \mathbb{N}}$ be the iterates of \cref{alg:proximal_newton}. Let $x_k\to x^*$, $F$ be locally convex around $x^*$, and $H(x_{\pi(k)}) \succeq 0$ 
for $k\geq K$.  Let $\theta \geq 0$ and suppose that there exists a sequence $(\eta_k)_{k \geq K} \subset \R_+$ such that \eqref{eq:eps_condition_1} holds for $k\geq K$.

Let $k\geq K$ be arbitrary and suppose that at iteration $k$ we have $j_k > 0$. Then 
\begin{equation}
    \lambda_k \leq 4 \left(\sqrt{2}\eta_k\right)^{\frac{1}{1+\theta}} g_k^{\frac{\theta}{1+\theta}},\label{eq:helper_lemma_bound_lambda_k_eps_k_g_k}
\end{equation}    
\begin{align}
    \Lambda_{k+1} \leq \left(\sqrt{2}\eta_k\right)^{\frac{1}{1+\theta}} \left(\frac{1}{g_k}\right)^{p-\frac{\theta}{1+\theta}}.\label{eq:helper_lemma_bound_big_Lambda_k_eps_k_g_k}
\end{align}
\end{lem}

\begin{proof}
As $j_k >0$, we know that $\lambda_k/4$ and $x_{k,+} = x_+(\lambda_k/4,x_k)=x_+(\lambda_k/4,x_k,x_{\pi(k)})$ were not accepted \rev{(since $H(x_{\pi(k)})\succeq0$ for all sufficiently large $k$, every trial subproblem admits a unique stationary point)}. %
Recall that we denote $r_{k,+}  = \norm{x_k-x_{k,+}}{}$.
Repeating the same steps as in the derivation of \eqref{eq:correctness_proof_1_1}, but using \eqref{eq:eps_condition_1} instead of \eqref{eq:L_condition} and $\lambda=\lambda_k/4$, we obtain
\begin{align}
\langle F'(x_{k,+}),x_k - x_{k,+} \rangle
\overset{\eqref{eq:eps_condition_1}}&{\geq} \frac{2}{\lambda_k}\norm{F'(x_{k,+})}{*}^2 + \left(\frac{\lambda^2_k- 16\eta_k^2r_{k,+}^{2\theta}}{8\lambda_k}\right) r_{k,+}^2.%
\label{eq:gamma_superlinear_1} 
\end{align}
Since $x_k\to x^*$, w.l.o.g., we have that for $k\geq K$ $x_k$ belong to the ball of local convexity. Using that convexity, that the first term in the r.h.s. of the previous inequality is non-negative, adding and subtracting $\frac{\lambda_kr_{k,+}^2}{16}$ in the r.h.s., we get 
\begin{align}
F(x_k) - F(x_{k,+}) \geq \langle F'(x_{k,+}),x_k - x_{k,+} \rangle
&\geq  \frac{\lambda_kr_{k,+}^2}{16} + \left(\frac{\lambda^2_k- 16\eta_k^2r_{k,+}^{2\theta} -(1\slash 2)\lambda_k^2}{8\lambda_k}\right) r_{k,+}^2\nonumber\\
&=  \frac{\lambda_kr_{k,+}^2}{16} + \left(\frac{(1\slash 2)\lambda^2_k- 16\eta_k^2r_{k,+}^{2\theta}}{8\lambda_k}\right) r_{k,+}^2.
\label{eq:gamma_superlinear_2}  
\end{align}
Note that by \eqref{eq:r_k+vsr_k} with $\lambda=\lambda_k/4$ we have $r_{k,+} \leq 4r_{k}$ and by \eqref{eq:rk_leq_gk_over_lambda_k_new} we have $r_k\leq g_k\slash\lambda_k$. Thus, we have $r_{k,+} \leq 4g_k/\lambda_k$. Assume for a contradiction that
 $ \lambda_k> 4 \left(\sqrt{2}\eta_k\right)^{\frac{1}{1+\theta}} g_k^{\frac{\theta}{1+\theta}}.$ 
Then, 
\begin{align}
    \lambda_k^{2+2\theta} > 4^{2+2\theta} 2\eta_k^2 g_k^{2\theta} \qquad \Longleftrightarrow \qquad \lambda_k^{2} > 32 \eta_k^2\left(\frac{4g_k}{\lambda_k}\right)^{2\theta} \geq   32\eta_k^2 r_{k,+}^{2\theta}.
\end{align}
Whence, we obtain that both the r.h.s. of \eqref{eq:gamma_superlinear_1} and the r.h.s. of \eqref{eq:gamma_superlinear_2} are non-negative and $\lambda_k/4$ is accepted. This is a contradiction. Thus, \eqref{eq:helper_lemma_bound_lambda_k_eps_k_g_k} holds, and \eqref{eq:helper_lemma_bound_big_Lambda_k_eps_k_g_k} follows from this and $\Lambda_{k+1}=\lambda_k/(4g_k^p)$ according to line 9 of \cref{alg:proximal_newton}.
\end{proof}

We are now in a position to prove the refined upper bound on $\lambda_k$ stated in the main text.

    \proplambdaktozero*
\begin{proof}
We first show \eqref{eq:lambda_k_iterative_estimate}. For arbitrary $k\geq K$, if $j_k >0$, by \cref{lem:helper_lemma_lambda_k_jk_pos} we directly get
\begin{align}
    \lambda_k \leq 4 \left(\sqrt{2}\eta_k\right)^{\frac{1}{1+\theta}} g_k^{\frac{\theta}{1+\theta}}\label{eq:eq_with_epsilon_k_case_1}.
\end{align}
If, on the other hand, $j_k=0$, we have by line 9 of \cref{alg:proximal_newton}
\begin{align}
    \lambda_k %
    = \frac{\Lambda_{k}}{4}\cdot 4g_k^p = \frac{\lambda_{k-1}}{4 \cdot 4g_{k-1}^p}\cdot 4g_k^p = \frac{\lambda_{k-1}}{4} \cdot \left(\frac{g_k}{g_{k-1}}\right)^p\overset{\eqref{eq:g_k_leq_g_k-1}}{\leq} \frac{\lambda_{k-1}}{4} 2^p \leq \frac{1}{2}\lambda_{k-1},
\end{align}
where we used that for $p\in[0,1]$, $2^p\leq2$.
Thus, we have \eqref{eq:lambda_k_iterative_estimate}. 

The remaining inequalities %
\eqref{eq:lambda_k_rate_bounds}  are shown by induction. %
\begin{enumerate}[label=(\roman*), wide, labelwidth=!, labelindent=0pt]\itemsep=0em
\item[\textbf{Base case $k=K$}.] 
If $j_K > 0$ then we have %
\eqref{eq:lambda_k_rate_bounds}  by \cref{lem:helper_lemma_lambda_k_jk_pos} and the bound  $0\leq \eta_K \leq C$.  
Now suppose that $j_K=0$; then $\Lambda_{K+1} = \Lambda_K\slash 4 \leq \Lambda_K \leq \left(\sqrt{2}C\right)^{\frac{1}{1+\theta}} \left(\frac{1}{g_K}\right)^{p-\frac{\theta}{1+\theta}}$ by assumption \eqref{eq:Lambda_0_assumpt}, and we also have
\[
\lambda_K = \Lambda_K g_K^p \leq 4 \left(\sqrt{2}C\right)^{\frac{1}{1+\theta}} g_K^{\frac{\theta}{1+\theta}}.
\]
Hence, the base case holds. 

\item[\textbf{Inductive step}.] Now consider iteration $k$ and consider again the two cases.

{If $j_k>0$,}  
again by \cref{lem:helper_lemma_lambda_k_jk_pos} we directly get
\eqref{eq:eq_with_epsilon_k_case_1}, and by line 9 of \cref{alg:proximal_newton}
\begin{align}
    \Lambda_{k+1}=\lambda_k/(4g_k^p) \overset{\eqref{eq:eq_with_epsilon_k_case_1}}{\leq} \left(\sqrt{2}\eta_k\right)^{\frac{1}{1+\theta}} \left(\frac{1}{g_k}\right)^{p-\frac{\theta}{1+\theta}}.
\end{align}
Thus the inductive step holds in this case since $0\leq \eta_k\leq C$.

{If $j_k=0$,} we have, using that $g_k \leq  2g_{k-1}$ by \eqref{eq:g_k_leq_g_k-1}, and that $p-\theta/(1+\theta)\geq 0 $ for all $\theta \in [0,1]$ and $p \in [1/2,1]$,  making use of the induction hypothesis, we derive 
\begin{align}
    \Lambda_{k+1}=\frac{\Lambda_{k}}{4}\overset{\text{(inductive hypothesis)}}&{\leq} \frac{1}{4}\left(\sqrt{2}C\right)^{\frac{1}{1+\theta}} \left(\frac{1}{g_{k-1}}\right)^{p-\frac{\theta}{1+\theta}} \\
    \overset{\eqref{eq:g_k_leq_g_k-1}}&{\leq} \frac{2^{p-\frac{\theta}{1+\theta}}}{4}\left(\sqrt{2}C\right)^{\frac{1}{1+\theta}} \left(\frac{1}{g_{k}}\right)^{p-\frac{\theta}{1+\theta}} \\
    &\leq \left(\sqrt{2}C\right)^{\frac{1}{1+\theta}} \left(\frac{1}{g_{k}}\right)^{p-\frac{\theta}{1+\theta}},
\end{align}
where in the last step we used that $p-\frac{\theta}{1+\theta}\leq 1$ for 
all $\theta \in [0,1]$ and $p \in [1/2,1]$.
Hence, we have also by line 9 of \cref{alg:proximal_newton} and $j_k=0$
\begin{align}
    \lambda_k=\Lambda_{k} g_k^p=\Lambda_{k+1} \cdot 4g_k^p \leq 4 \left(\sqrt{2}C\right)^{\frac{1}{1+\theta}} g_k^{\frac{\theta}{1+\theta}}.
\end{align}
\end{enumerate}
Combining the two cases, we finish the inductive step. Thus,  
\eqref{eq:lambda_k_rate_bounds} hold for all $k\geq K$.

\end{proof}

\subsection{Proofs of superlinear convergence under semismoothness}
\label{sec:superlin_proof}
We are finally ready to prove the main results for the semismooth and $\gamma$-order semismooth cases. We start with the first result that says that if convergence is already established, under semismoothness at $x^*$ (see \cref{ass:semismoothness}) and the Dennis--Moré condition, it accelerates to superlinear asymptotically.

\thmfastlocal*
    \begin{proof}
 Recall the notation \eqref{eq:F_r_g_notaion} and in particular that $x_{k,+} \coloneqq x_+(\lambda_k/4,x_k, x_{\pi(k)})$ is the last nonaccepted trial point if the linesearch was not trivial, i.e., $j_k>0$. Note that by \cref{lem:strange_convergence} we have that $x_{k,+} \to x^*$ as $k \to \infty$.
    Since we are in the asymptotic regime, we take $k$ large enough so that any of the preceding results and properties (that hold for large enough $k$ or in a ball around the minimum $x^*$) that we may require below hold. 
 
    We can use \cref{lem:useful_bounds_new_strong} with $\lambda=\lambda_k/4$ to obtain from \eqref{eq:good_bound_1} and \eqref{eq:good_bound_2}%
    \begin{align}
        \| x_{k,+} - x^*\|  \leq a\| x_{k,+}  - x_k\|,  \quad 
        \norm{x_k-x^*}{}  \leq b\| x_{k,+}  - x_k\|,\label{eq:good_bound_abstract_2}
    \end{align}
    where $a,b>0$ are constants.
    Using these inequalities, by the triangle inequality, semismoothness at $x^*$ \eqref{eq:semismooth_at_x*}, and the Dennis--Moré-type condition \eqref{eq:DM_condition}, we have %
    \begin{align}
    \norm{ f'(x_{k,+}) - f'(x_{k}) - H(x_k)(x_{k,+} - x_k) }{*} &\leq
        \norm{ f'(x_{k,+}) - f'(x^*) - H(x_{k,+})(x_{k,+} - x^*) }{*} \nonumber\\
        &\quad +  \norm{ f'(x_k) - f'(x^{*}) - H(x_k)(x_{k} - x^*) }{*}\nonumber\\
        &\quad +  \norm{ (H(x_{k,+}) - H(x_{k}))(x_{k,+} - x^*) }{*} \nonumber\\
         \overset{\eqref{eq:semismooth_at_x*},\eqref{eq:DM_condition}}&{=} \so(\| x_{k,+} -x^*\|) +\so(\| x_{k} -x^*\|) + \so(\| x_{k,+} -x_k\|) \nonumber\\
        \overset{\eqref{eq:good_bound_abstract_2}}&{=} \so(\| x_{k,+} -x_k\|).\label{eq:proof_part_i}
    \end{align}
It further follows by the Dennis--Moré-type condition \eqref{eq:lazy_DM_condition} that
\begin{align*}
        \norm{ f'(x_{k,+}) - f'(x_{k}) - H(x_{\pi(k)})(x_{k,+} - x_k) }{*} &\leq \norm{ f'(x_{k,+}) - f'(x_{k}) - H(x_k)(x_{k,+} - x_k) }{*}\\
        &\quad + \norm{(H(x_k)-H(x_{\pi(k)})(x_{k,+} - x_k)}{*}\\
        \overset{\eqref{eq:proof_part_i}, \eqref{eq:lazy_DM_condition}}&{=} \so(\| x_{k,+} -x_k\|).
    \end{align*}
    Thus there exists a sequence $(\varepsilon_k)_{k \in \mathbb{N}} \subset \R_+$ such that $\varepsilon_k \to 0$ and, as $k\to \infty$,
    \begin{align*}
            \norm{ f'(x_{k,+}) - f'(x_{k}) - H(x_{\pi(k)})(x_{k,+} - x_k) }{*} \leq \varepsilon_k\| x_{k,+} -x_k\|.
    \end{align*}
    This is \eqref{eq:eps_condition_1} with $\eta_k=\varepsilon_k$ and $\theta=0$. Applying \cref{prop:lambda_k_to_zero}, we obtain \eqref{eq:lambda_k_iterative_estimate} with $\theta=0$ and $\eta_k=\varepsilon_k\to0$ as $k\to \infty$. Applying further \cref{lem:helper_lambdak_zero}, we obtain that $\lambda_k\to0$ as $k\to \infty$. Finally, applying \cref{lem:iterates_g_k_and_lambda_k}, we get the claimed superlinear convergence as $k\to \infty$:
    \begin{align*}
        \norm{x_{k+1}-x^*}{} &\overset{\eqref{eq:lem_iterates_1}}{\leq} \frac{4\lambda_k}{\mu} \norm{x_k-x^*}{}= \so(\|x_k -x^*\|), \\
        \norm{ F'(x_{k+1})}{*} &\overset{\eqref{eq:lem_iterates_2}}{\leq} \frac{4\lambda_k}{\mu}\norm{F'(x_k)}{*} = \so(\norm{F'(x_k)}{*}) .  
    \end{align*}    
\end{proof}

The next lemma gives a sufficient condition for the Dennis--Moré-type conditions \eqref{eq:DM_condition} and \eqref{eq:lazy_DM_condition}.

\lemlazyDMHolds*
    \begin{proof}
    By \cref{lem:strange_convergence}, we obtain that $x_{k,+} \to x^*$ as $k\to \infty$. By \eqref{eq:good_bound_1} with $\lambda=\lambda_k/4$, recalling that $x_{+}(\lambda_k/4,x_k,x_{\pi(k)})=x_{k,+}$, we get that there is a constant $c>0$ s.t. $\norm{x_{k,+} - x^*}{}\leq c \norm{x_{k,+} - x_k}{}$. Thus,  
    \begin{align*}
        \norm{ (H(x_{k,+}) - H(x_{k}))(x_{k,+} - x^*) }{*} \leq \norm{ H(x_{k,+}) - H(x_{k})}{\mathrm{op}} \cdot c \norm{x_{k,+} - x_k}{}
    \end{align*}
    and \eqref{eq:DM_condition} follows by the continuity of $H$ and the fact that $x_{k,+} \to x^*$ and $x_{k} \to x^*$.
     Further, we have
        \begin{align*}
            \norm{(H(x_k)-H(x_{\pi(k))})(x_{k,+} - x_k)}{*} &\leq \norm{H(x_k)-H(x_{\pi(k)})}{\mathrm{op}}\norm{x_{k,+} - x_k}{}
        \end{align*}
        and since $x_k$ and $x_{\pi(k)}$ both converge to $x^*$, 
        we get \eqref{eq:lazy_DM_condition} by continuity of $H$.
    \end{proof}
Note that by applying the triangle inequality to $H$, continuity may be relaxed to hold only at $x^*$.

We finally prove the main theorem under semismoothness.

\thmglobalsuperlinear*
\begin{proof}
Since the assumptions of \cref{thm:global_convergence_nonconv_PL} hold, we obtain the claimed linear convergences. Moreover, $x_k \to x^*$ as $k \to \infty$ where $x^*$ is a global minimum. %
The latter and additional local assumptions mean that all the assumptions of \cref{thm:fast_local} hold and we have that $\lambda_k\to 0$, which by \cref{thm:global_convergence_nonconv_PL} and  \cref{thm:fast_local} imply the claimed superlinear convergence.
\end{proof}   

According to the above theorem, the whole convergence picture is as follows. From any starting point, the method converges linearly to a solution $x^*$ thanks to the PL condition. Since the linear convergence $x_k\to x^*$ is guaranteed, at some point $x_k$ starts to belong to the ball around $x^*$ where the local strong convexity of \cref{ass:strong_convexity} holds. Asymptotic semismoothness and DM conditions also start to work after some iteration when $x_k$ is sufficiently close to $x^*$. All this implies that $\lambda_k\to0$ and the method starts to converge superlinearly when it is sufficiently small. In that sense, the superlinear convergence is local and asymptotic, but the linear rate is global and nonasymptotic.

\subsection{Proofs of superlinear convergence under $\gamma$-order semismoothness}
\label{sec:superlin_gamma_proof}

We now move on to the $\gamma$-order semismooth setting, where we manage to quantify the rate of superlinear convergence.
    
\thmfastlocalgamma*
    \begin{proof}

    As in the proof of \cref{thm:fast_local}, we take below $k$ sufficiently large so that all the needed local or asymptotic properties hold.   
    Note again that by \cref{lem:strange_convergence} we have that $x_{k,+} \to x^*$ as $k \to \infty$. By the triangle inequality, $\gamma$-order semismoothness at $x^*$ \eqref{eq:gamma_order_semismooth_at_x*}, and the $\gamma$-order Dennis--Moré-type condition \eqref{eq:gamma_order_DM_condition}, we have
    \begin{align*}
    \norm{ f'(x_{k,+}) - f'(x_{k}) - H(x_k)(x_{k,+} - x_k) }{*} &\leq
        \norm{ f'(x_{k,+}) - f'(x^*) - H(x_{k,+})(x_{k,+} - x^*) }{*} \nonumber\\
        &\quad +  \norm{ f'(x_k) - f'(x^{*}) - H(x_k)(x_{k} - x^*) }{*}\nonumber\\
        &\quad +  \norm{ (H(x_{k,+}) - H(x_{k}))(x_{k,+} - x^*) }{*} \nonumber\\
         \overset{\eqref{eq:gamma_order_semismooth_at_x*},\eqref{eq:gamma_order_DM_condition}}&{=} \mathcal{O}(\| x_{k,+} -x^*\|^{1+\gamma}) + \mathcal{O}(\| x_{k} -x^*\|^{1+\gamma}) \\
         &\quad+ \mathcal{O}(\| x_{k,+} -x_k\|^{1+\gamma}) \nonumber\\
        \overset{\eqref{eq:good_bound_abstract_2}}&{=} \mathcal{O}(\| x_{k,+} -x_k\|^{1+\gamma}).
    \end{align*}
    Thus, there exists a bounded sequence $(\eta_k)_{k \in \mathbb{N}} \subset \R_+$ such that 
    \begin{align*}
             \norm{ f'(x_{k,+}) - f'(x_{k}) - H(x_k)(x_{k,+} - x_k) }{*}  
            &\leq \eta_k \| x_{k,+} -x_k\|^{1+\gamma}.
        \end{align*}
    This is \eqref{eq:eps_condition_1} with $\theta=\gamma$ and $\pi(k)=k$. Applying \cref{prop:lambda_k_to_zero}, we obtain \eqref{eq:lambda_k_rate_bounds} with $\theta=\gamma$, which tells us that the sequence of regularization parameters goes to zero at the rate $\lambda_k = \mathcal{O}(g_k^{\frac{\gamma}{1+\gamma}})$. Applying \cref{lem:iterates_g_k_and_lambda_k}, we get the claimed $\frac{1+2\gamma}{1+\gamma}$-order superlinear convergence of subgradients as $k\to \infty$
    \begin{align*}
        \norm{ F'(x_{k+1})}{*} &\overset{\eqref{eq:lem_iterates_2}}{\leq} \frac{4\lambda_k}{\mu}\norm{F'(x_k)}{*} = \mathcal{O}(\norm{F'(x_k)}{*}^{\frac{1+2\gamma}{1+\gamma}}) .  
    \end{align*}

    The convergence rate for the iterates is derived as follows. First, we have the following chain of inequalities
    \begin{align*}
        g_{k} &\overset{\eqref{eq:g_k_leq_r_k}}{\leq} 2\lambda_{k-1} r_{k-1} \overset{\eqref{eq:F_r_g_notaion}}{\leq}  2\lambda_{k-1} \left(\norm{x_{k}-x^*}{} + \norm{x_{k-1}-x^*}{} \right) \\
        &\overset{\eqref{eq:lem_iterates_1}}{\leq} 2\lambda_{k-1}  \left(\frac{4}{\mu}\lambda_{k-1} \norm{x_{k-1}-x^*}{} + \norm{x_{k-1}-x^*}{} \right) \overset{\eqref{eq:lambda_bound_eventual}}{\leq}  \mathcal{O}(\norm{x_{k-1}-x^*}{}).
    \end{align*}
    Note that in the meantime we obtained also that $\norm{x_{k}-x^*}{}=\mathcal{O}(\norm{x_{k-1}-x^*}{})$.
    We finish the proof as follows:
    \begin{align*}
    \|x_{k+1} - x^*\| \overset{\eqref{eq:lem_iterates_1}}{\leq} \frac{4\lambda_{k}}{ \mu}   \|x_k -x^*\|=\mathcal{O}(g_k^{\frac{\gamma}{1+\gamma}}\|x_k -x^*\|)=\mathcal{O}(\norm{x_{k-1}-x^*}{}^{\frac{\gamma}{1+\gamma}} \cdot \norm{x_{k-1}-x^*}{}).
    \end{align*}    
\end{proof}
Let us note that the rate in \cref{rem:improvement} follows directly by using $g_k \leq (L\slash 2)\norm{x_k-x^*}{}$ (see (2.1.11) in \cite{nesterov2018lectures}) in the equality in the final displayed equation in the proof above. 

The next lemma gives a sufficient condition for the Dennis--Moré-type condition \eqref{eq:gamma_order_DM_condition}.

\lemctwoimpliesdm*
\begin{proof}
By \cref{lem:strange_convergence}, we have that $x_{k,+} \to x^*$ as $k\to \infty$. By \eqref{eq:good_bound_1} and \eqref{eq:good_bound_2} with $\lambda=\lambda_k/4$, recalling that $x_{+}(\lambda_k/4,x_k,x_{\pi(k)})=x_{k,+}$, we get that there are constants $c_1,c_2>0$ s.t. $\norm{x_{k,+} - x^*}{}\leq c_1 \norm{x_{k,+} - x_k}{}$ and $\norm{x_{k} - x^*}{}\leq c_2 \norm{x_{k,+} - x_k}{}$.
This, together with the triangle inequality 
gives us
\begin{align*}
&\frac{\norm{(H(x_{k,+}) - H(x_k))(x_{k,+} - x^*)}{*}}{\| x_{k,+} - x_k\|} = \frac{\norm{(H(x_{k,+}) -H(x^*) + H (x^*) - H(x_k))(x_{k,+} - x^*)}{*}}{\| x_{k,+}- x_k\|} \\
&\leq \frac{(\norm{H(x_{k,+}) -H(x^*)}{\mathrm{op}} + \norm{H(x^*) - H(x_k)}{\mathrm{op}})\norm{x_{k,+} - x^*}{}}{\| x_{k,+} - x_k\|} \\
\overset{\eqref{eq:good_bound_1}}&{\leq}c_1(\norm{H(x_{k,+}) -H(x^*)}{\mathrm{op}} + \norm{H(x^*) - H(x_k)}{\mathrm{op}})\\
&\leq c_1(\norm{x_{k,+} - x^*}{}^\gamma + \norm{x^* - x_k}{}^\gamma)\\
\overset{\eqref{eq:good_bound_1}, \eqref{eq:good_bound_2}}&{\leq} c_1(c_1^\gamma\norm{x_{k,+} - x_k}{}^\gamma + c_2^\gamma\norm{x_{k,+} - x_k}{}^\gamma).
\end{align*}
This clearly implies the condition \eqref{eq:gamma_order_DM_condition}.
\end{proof}

\thmglobalsuperlineargamma*
\begin{proof}
Since the assumptions of \cref{thm:global_convergence_nonconv_PL} hold, we obtain the claimed linear convergences. Moreover, $x_k \to x^*$ as $k \to \infty$ where $x^*$ is a global minimum.
The latter and additional local assumptions, since \cref{ass:gamma_order_semismoothness}  implies \cref{ass:semismoothness}, mean that all the assumptions of
\cref{thm:fast_local} hold  and we have that $\lambda_k\to 0$, which by \cref{thm:global_convergence_nonconv_PL} and  \cref{thm:fast_local} imply the claimed superlinear convergence.
Finally, $\frac{1+2\gamma}{1+\gamma}$-order superlinear convergence of subgradients and two-step $\frac{1+2\gamma}{1+\gamma}$-order superlinear convergence of $x_k$ follow from \cref{thm:fast_local_gamma}.
\end{proof}   

The above theorem complements the convergence picture described after the proof of \cref{thm:global_superlinear} by quantifying the order of local superlinear convergence.

\subsection{Proofs of superlinear convergence under $C^2$ assumption}
\label{sec:superlin_C_2_proof}
    \thmglobalconvergencectwo*
\begin{proof}
Since the assumptions of \cref{thm:global_convergence_nonconv_PL} hold, we obtain the claimed linear convergences.

Recall again the notation by $x_{k,+} \coloneqq x_+(\lambda_k/\rev{4},x_k,x_{\pi(k)})$ (see \eqref{eq:PLMSN_step} and \eqref{eq:F_r_g_notaion}).
Since we are in the finite dimensional setting, we write $f'\equiv \nabla f$.

The idea is to show that \eqref{eq:eps_condition_1} holds with $\eta_k\to0$ and $\theta=0$ despite the fact that we don't have local strong convexity in this case. This will then imply \eqref{eq:lambda_k_iterative_estimate} by \cref{prop:lambda_k_to_zero}, and applying \cref{lem:helper_lambdak_zero} we will obtain $\lambda_k \to 0$. Thus it suffices to show
\begin{equation}
    \omega_k\coloneqq   \frac{ \norm{\nabla f(x_{k,+}) -  \nabla f(x_k) - \nabla^2 f(x_{\pi(k)})(x_{k,+}-x_k)}{*}}{\|x_{k,+} - x_k\|} \to 0 \quad \text{as $k \to \infty$.} \label{eq:higher_regularity_11}
\end{equation}
For this purpose, fix an arbitrary $\varepsilon>0$. By assumption, there exists a ball $B_R(x^*)$ on which $\nabla^2 f$ exists and is continuous. Moreover, by local convexity and twice differentiability of $f$ we obtain $H(x)\coloneqq  \nabla^2 f(x)\succeq 0$ for all $x \in B_R(x^*)$. In addition, using compactness of $B_R(x^*)$ and the Heine--Cantor theorem, we get that $\grad^2 f$ is even uniformly continuous on $B_R(x^*)$. The latter means that there exists $\delta>0$ such that for every $x,y \in B_R(x^*)$ with $\| x -y \| \leq \delta$, we have 
\begin{equation*}
    \|\nabla^2 f(x) - \nabla^2 f(y)\|_{\mathrm{op}} \leq \varepsilon. %
\end{equation*}
As $x_k \to x^*$, also $x_{\pi(k)} \to x^*$ and there exists $K_1$ such that $k \geq K_1$ implies $x_k, x_{\pi(k)} \in B_R(x^*)$. %
This in turn implies that, for $k \geq K_1$, $H(x_k) \succeq 0$. %
We conclude by \cref{lem:strange_convergence} that $ x_{k,+} \to x^*$ as well.
We further infer that there is an index $K_2 \geq K_1$ such that $x_k, x_{\pi(k)}, x_{k,+} \in B_R(x^*)$, 
$\norm{x_{k,+}-x_k}{} \leq \delta\slash 2$ and  $\norm{x_k-x_{\pi(k)}}{} \leq \delta\slash 2$  for all $k \geq K_2$. 
Consequently, we deduce from the fundamental theorem of calculus that, for all $k \geq K_2$,
\begin{align}
    &\norm{\nabla f(x_{k,+}) -  \nabla f(x_k) -  \nabla^2 f(x_{\pi(k)})(x_{k,+}-x_k)}{*}  \notag\\
    &= \norm{ \int_0^1 (\nabla^2f(x_k + t(x_{k,+} - x_k)) - \nabla^2 f(x_{\pi(k)}))(x_{k,+} - x_k) \, \mathrm{d} t }{*} \notag \leq 
    \varepsilon \|x_{k,+} - x_{k}\|
\end{align}
as, clearly, for $x_t \coloneqq x_{k} + t(x_{k,+}-x_{k})$, we have $x_t \in B_R(x^*)$ and 
\[\norm{x_t - x_{{\pi(k)}}}{} \leq \norm{x_k - x_{\pi(k)}}{} + \norm{x_{k,+}-x_k}{} \leq \delta\]
for any  $t\in (0,1)$. Since the choice of $\varepsilon>0$ was arbitrary, we infer \eqref{eq:higher_regularity_11}. As  described at the start of the proof, this implies $\lambda_k \to 0$. This, by \cref{thm:global_convergence_nonconv_PL} implies that all the convergences $x_k \to x^*, \quad F(x_k)\to F^* , \quad \norm{F'(x_k)}{*} \to 0 $ are superlinear.
\end{proof}
\section{Experimental details}\label{sec:app_experiments}
Our code modifies the framework provided in \cite{super-newton}, released under the MIT License. The SVM example is also modified from \cite{github:leapssn}, likewise released under the MIT License.
\rev{All computed results reported in the paper were obtained using an Intel Core i7-11800H CPU on a Dell XPS 17 laptop with 64 GB of RAM.}
\subsection{Neural networks with Lipschitz constraints}
\paragraph{Semismooth differentiability.}

Let us briefly discuss semismoothness properties of the objective in \eqref{eq:Lipschitz_objective}. For simplicity, in our experiment we use $d=1$, so the neural network $\Phi(\cdot,\theta)\colon \mathbb{R}^n \to \mathbb{R}$. Moreover, we assume that $\Phi(\cdot;\theta)$ uses $C^\infty$ activation functions (in our experiment we use $\tanh$), so that $x\mapsto \nabla_x^2\Phi(x;\theta)$ exists and depends smoothly on both $x$ and $\theta$. For a matrix $A$, the squared spectral norm satisfies $\|A\|_2^2=\lambda_{\max}(A^\top A)$ and can be nondifferentiable when the leading singular value has multiplicity larger than one, see \citep[Theorem~2]{overton1992large}. Following \cite{pesquet2021learning,hurault2022proximal}, we therefore replace the exact squared spectral norm by a surrogate obtained from unrolling a fixed number $K \in \mathbb{N}$ of power iterations. With the regularized normalization described below, the squared norm of the final Hessian-vector product defines a $C^\infty$ surrogate. The remaining nonlinearity in the penalty, $t\mapsto (\max\{t,0\})^2$, is $C^1$ with a semismooth gradient; consequently, after replacing $\|\nabla_x^2\Phi(x,\theta)\|_2^2$ by this power-iteration surrogate, the resulting objective is $C^1$ with a semismooth gradient as well, by the chain rule for semismooth functions (see, e.g., \cite{ulbrich2011semismooth}). Using a finite number $K$ of power iterations amounts to optimizing a fixed surrogate objective, whose derivatives are obtained by differentiating through the full recursion. Although this does not provide a certified approximation error relative to the exact penalty, in this experiment the method remained stable and effective under this approximation.

\paragraph{Power-iteration surrogate via Hessian-vector products.}

\rev{For each sample $x_i$, we approximate $\|\nabla_x^2\Phi(x_i;\theta)\|_2^2$ by $K$ steps of power iteration, see \citep[Chapter 8]{golub2013matrix}. We describe this briefly. Let $H_i(\theta):=\nabla_x^2\Phi(x_i;\theta)\in\mathbb{R}^{n\times n}$. Note that $H_i(\theta)$ is symmetric by our chosen regularity of the activation function. Hence
\[
\|H_i(\theta)\|_2^2
=\lambda_{\max}\!\left(H_i(\theta)^\top H_i(\theta)\right)
=\max_j|\lambda_j(H_i(\theta))|^2,
\]
and power iteration targets eigendirections associated with eigenvalues of largest absolute value. For this purpose, we initialize fixed vectors $v_i^{(0)}\in\mathbb{R}^n$ from the all-ones vector using the regularized normalization below. For $k=0,\dots,K-1$, we then iterate
\[
w_i^{(k)}(\theta)=H_i(\theta)\,v_i^{(k)}(\theta),\qquad
v_i^{(k+1)}(\theta)
=\frac{w_i^{(k)}(\theta)}
{\sqrt{\|w_i^{(k)}(\theta)\|_2^2+\varepsilon^2}},
\]
where $\varepsilon=10^{-7}$ safeguards against vanishing normalization denominators. We set
\[
\widehat{s}_K(x_i;\theta)
=\|w_i^{(K-1)}(\theta)\|_2
\approx \|H_i(\theta)\|_2
\]
and use its square as the surrogate for $\|H_i(\theta)\|_2^2$ in the penalty. With a fixed number of iterations and deterministic initialization, this estimator may underestimate the exact spectral norm. In practice, we do not form $H_i$ explicitly; instead, we compute $H_i(\theta)v$ using automatic differentiation. In our implementation, we run the above power iteration over the full batch.}

\rev{Constructing the parameter Hessian requires second-order differentiation through the unrolled power iterations. This makes Hessian evaluation particularly expensive and motivates reusing the resulting second-order information over several iterations.}

\begin{remark}[Differentiability of truncated power iteration]
Since matrix-vector multiplication is smooth and the normalization map
$g(w)=w/\|w\|_2$ is $C^\infty$ on $\mathbb{R}^d\setminus\{0\}$, it follows by
repeated application of the chain rule (for smooth functions) that the mappings $\theta\mapsto v_i^{(k)}(\theta)$ and $\theta\mapsto \widehat{s}_K(x_i;\theta)$ are as smooth as $\theta\mapsto H_i(\theta)$ on the set
\[
\Omega_{i,K}:=\Big\{\theta:\ \|w_i^{(k)}(\theta)\|_2>0 \text{ for all } k=0,\dots,K-1\Big\}.
\]
In particular, if
$\theta\mapsto H_i(\theta)$ is $C^1$, then
$\theta\mapsto \widehat{s}_K(x_i;\theta)$ is $C^1$ on
$\Omega_{i,K}$. Outside $\Omega_{i,K}$, some iterate satisfies $w_i^{(k)}(\theta)=0$, in which case the normalization step is undefined; thus truncated power iteration (and hence $\widehat{s}_K$) cannot be expected to be differentiable there. 
In our implementation, we therefore additionally safeguard against numerical
degeneracy by using the regularized normalization
$v=w/\sqrt{\|w\|_2^2+\varepsilon^2}$, with $\varepsilon=10^{\rev{-7}}$, which makes the recursion globally $C^\infty$ \rev{if $\theta \mapsto H_i(\theta)$ is smooth as well, which holds for our network with $\tanh$ activations. }
\end{remark}

\paragraph{Experimental data.}
We generate $M=100$ training inputs $x_i\in[0,2\pi]^2$ uniformly at random and define targets
\[
y_i = \sin(x_{i,1})\cos(x_{i,2})
+0.2\,\sin(2x_{i,1}+x_{i,2})
+0.1\,x_{i,1}x_{i,2}
+ \sigma z_i,
\]
with i.i.d.\ Gaussian noise $z_i\sim\mathcal{N}(0,1)$ and $\sigma = 0.05$.

\paragraph{Network architecture.}
We use a fully-connected network $\Phi(\cdot;\theta)\colon\mathbb{R}^2\to\mathbb{R}$ with widths
$2\!-\!16\!-\!16\!-\!1$ and $\tanh$ activations.

\paragraph{Optimization and parameters.}
We apply Algorithm~\ref{alg:proximal_newton} to the objective in \eqref{eq:Lipschitz_objective} with penalty parameter $\lambda=1.0$ and $\psi \equiv 0$, so that the proximal subproblem reduces to solving a linear system. The Hessian is computed via automatic differentiation in PyTorch. At each inner iteration, we solve a damped Newton system of the form $(H + \lambda_k I)s = g$. Because the objective is nonconvex, $H$ may be indefinite, and $H + \lambda_k I$ can be ill-conditioned or singular; in such cases, we fall back to a least-squares solve. In this experiment, we symmetrize the Hessian as $\tfrac12(H + H^\top)$ before solving the linear system. 
\rev{The damped Newton systems are solved by LU factorization, with a least-squares fallback when necessary.}

As a baseline, we minimize the same objective \eqref{eq:Lipschitz_objective} using PyTorch's \footnote{\url{https://pytorch.org/}, License: BSD-3-Clause.} \adam optimizer with default momentum parameters and no weight decay. We run \rev{$T=2 \times 10^5$} %
full-batch optimization steps. The learning rate is initialized to $\eta =  10^{-4}$ and adapted using \texttt{ReduceLROnPlateau} (factor $0.5$, patience $5$, threshold $10^{-5}$, and minimum learning rate $10^{-6}$). In both algorithms, the Hessian penalty is computed as described above using 
\rev{$K=10$}
power iterations with a deterministic initialization. Additional hyperparameters are reported in \cref{tab:hparams}. We test Algorithm~\ref{alg:proximal_newton} under five Hessian-update schedules, recomputing the Hessian every $m$ iterations with 
\rev{$m \in \{1,50,100, 200, 400\}$}, where $m$ is defined in \eqref{eq:def_m}.

\begin{table}[t]
\centering
\caption{Optimization hyperparameters for training the Lipschitz-constrained neural network.}
\label{tab:hparams}
\begin{tabular}{lll}
\toprule
Method & Hyperparameter & Value \\
\midrule
\multirow{6}{*}{\gladssn} 
& coefficient $\Lambda_{0}$  & $1.0$ \\
& gradient exponent $p$ & $0.5$ \\
& power iterations $K$ & \rev{$10$} \\ %
& Hessian update frequencies $m$ & \rev{\{1,50, 100, 200, 400\}} \\
\midrule
\multirow{6}{*}{\adam}
& steps $T$ &  \rev{$2 \times 10^5$} \\ %
& batch size $B$ & $M$ (full batch) \\
& learning rate $\eta$ & $10^{-4}$ \\
& scheduler & \texttt{ReduceLROnPlateau}\\
& schedule params & factor $0.5$, patience $5$, threshold $10^{-5}$, min lr $10^{-6}$ \\
& power iterations $K$ & \rev{$10$} \\
\bottomrule
\end{tabular}
\end{table}

\begin{remark}[On the choice of the spectral norm]
Following \cite{hurault2022proximal,pesquet2021learning}, we use the spectral norm
$\lVert\cdot\rVert_{2}$ to ensure that
$\nabla_x\Phi(\cdot,\theta)$ is non-expansive. Other matrix norms could be used as well,
however, in applications the spectral norm typically performs better; see
\citep[Remark~3.4]{pesquet2021learning}.
\end{remark}
\rev{
\paragraph{Results.}
\Cref{fig:NN} shows that \gladssn reduces the gradient norm in substantially fewer iterations than \adam. Lazy updates yield an almost ninefold speedup for $m=100$ versus $m=1$ in first reaching a gradient norm below $10^{-3}$. Furthermore, \cref{fig:NN_loss} shows lower training objectives than those attained by \adam, an important consideration in neural network training. While our analysis does not prescribe an optimal $m$, $m=100$ offers a good empirical tradeoff in this experiment. In \cref{fig:NN_lip}, we report a sampled estimate of the Lipschitz constant of the input gradient $\nabla_x\Phi(\cdot;\theta)$, computed as
\[
\widehat L_{\mathrm{train}}(\theta)
:= \max_{1\le i\le M}\widehat s_K(x_i;\theta)
= \max_{1\le i\le M}\bigl\|w_i^{(K-1)}(\theta)\bigr\|_2.
\]
The maximum is taken over the training samples using the power-iteration estimates defined above.
}
\begin{figure}[h!]
\centering
 \includegraphics[width = 1\textwidth]{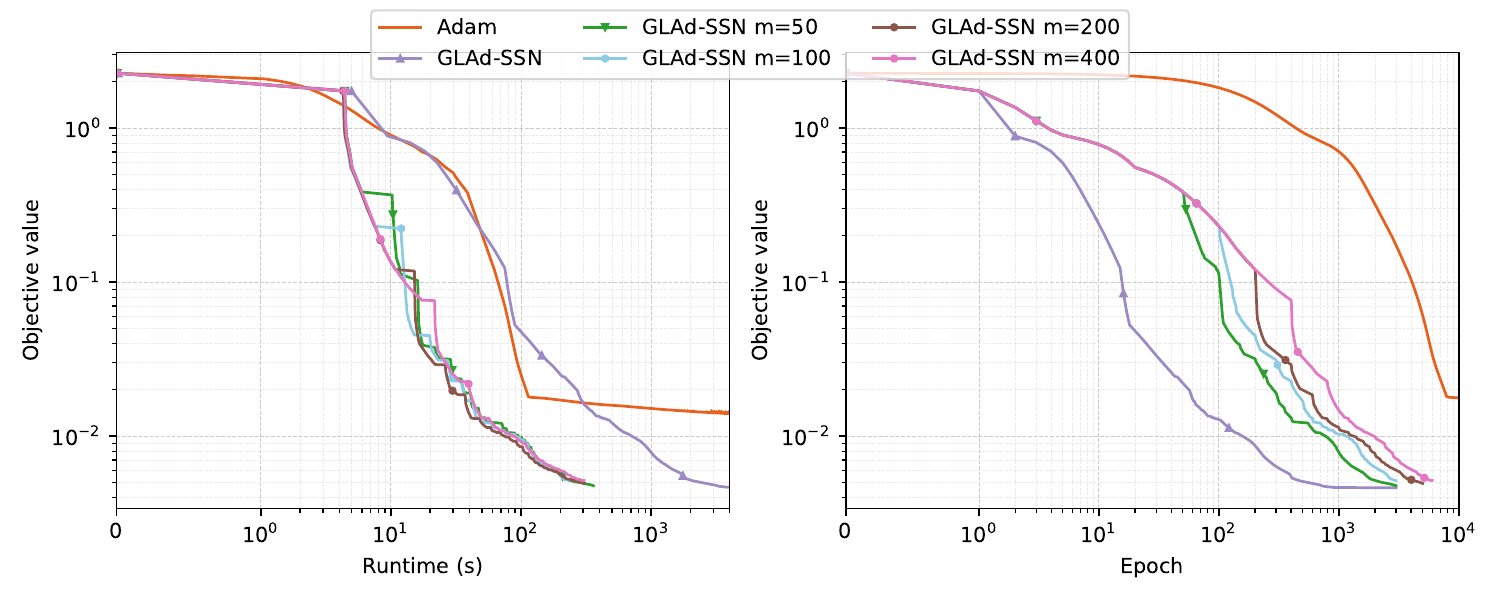}
\caption{\rev{ Training objective for the neural network experiment versus runtime (left) and epochs (right). \gladssn attains lower objective values than \adam, while lazy Hessian updates substantially reduce runtime. }}
\label{fig:NN_loss}
\end{figure}

\subsection{Non-negative matrix factorization}\label{sec:NMF_appendix}
\paragraph{Semismooth differentiability.}
The smooth \rev{quartic} data-fit part of \eqref{eq:mat_fac_2} is a polynomial in the entries of $(U,V)$ and hence $C^\infty$. To show semismooth differentiability of the whole objective it remains to show that $U \mapsto $
\(
\|(-U)^+\|_F^2=\|U_-\|_F^2
\)
is $C^1$ with semismooth gradient. The continuous differentiability directly follows from the differentiability of $t\mapsto \max\{-t,0\}^2$. Its derivative is the piecewise affine map
\(
t\mapsto \rev{2}\min\{t,0\}.
\)
Piecewise affine mappings are strongly semismooth (i.e., $1$-order semismooth), cf.  \cite{facchinei2003finite}. %
\rev{Hence the objective gradient is $\gamma$-order semismooth with $\gamma=1$.}

\paragraph{Experimental data.} 
We consider a synthetic non-negative matrix factorization problem with dimensions
$d=200$, $n=100$ and target rank $r=12$. We draw ground-truth non-negative factors
$U_{\rm true}\in\mathbb{R}_+^{d\times r}$ and $V_{\rm true}\in\mathbb{R}_+^{r\times n}$ with i.i.d.\ entries $[U_{\rm true}]_{ij},[V_{\rm true}]_{ij}\sim\mathrm{Unif}(0,1)$, and form the clean matrix
\[
\rev{Y_{\rm clean}=U_{\rm true}V_{\rm true} \in\mathbb{R}_+^{d\times n}.}
\]
The observed matrix is
\[
Y \equiv M_{\rm obs}=Y_{\rm clean}+\sigma Z,\qquad Z_{ij}\stackrel{\text{i.i.d.}}{\sim}\mathcal{N}(0,1),
\]
and in the reported setting we take 
\rev{$\sigma = 0.2$}. We set the ridge parameter to \rev{$\alpha=2$} %
and the non-negativity penalty weight to 
\rev{$\beta=200$} in \eqref{eq:mat_fac_2}.
The algorithm is initialized from a Gaussian vector $x_0\in\mathbb{R}^{dr+nr}$ with i.i.d.\ entries
\rev{$x_{0,k}\sim\mathcal{N}(0,1)$},  which is reshaped into initial factors
$U_0\in\mathbb{R}^{d\times r}$ and $V_0\in\mathbb{R}^{r\times n}$ by splitting $x_0$ into blocks of sizes
$dr$ and $nr$.

We apply Algorithm~\ref{alg:proximal_newton} with $\Lambda_0 = \rev{10}$ and compare different gradient exponents $p \in \{0.2,0.4,0.5,0.6,0.8\}$. 

\rev{
\paragraph{Gauss--Newton approximation.} Motivated by the positive semidefiniteness assumption in \Cref{thm:global_convergence_sublinear_nonconv}, we test the algorithm with both the full generalized Hessian of the objective in \eqref{eq:mat_fac_2} and its positive definite Gauss--Newton approximation. Writing $R(U,V):=UV-M_{\rm obs}$ and denoting its Jacobian with respect to the vectorized factors by $J_R(U,V)$, we retain the generalized curvature of the regularization terms:
\[
H_{\rm GN}(U,V)
:= J_R(U,V)^\top J_R(U,V)
   + \alpha I
   + \beta D(U,V).
\]
Here, $D(U,V)$ is diagonal, with entries equal to one for negative factor entries and zero otherwise. Since $\alpha>0$ and $\beta\geq0$, we have $H_{\rm GN}(U,V)\succeq\alpha I\succ0$.
}

Again we have $\psi \equiv 0$, and the subproblems become linear systems, which we solve using the MINRES algorithm due to the present non-convexity. We compare our algorithm against \leapssn  presented in \cite{LeapSSN} in the setting suggested in this paper, a standard gradient descent with Armijo linesearch \cite{wright2006numerical}, and the Cubic Newton method\rev{, using adaptive backtracking as in \cite[Algorithm~1]{doikov2021minimizing} and approximately solving the subproblems with the Krylov subspace solver GLRT \cite{cartis2011adaptive}. While this baseline performs well empirically beyond the assumptions of the cited analysis, our method combines comparable practical performance with convergence guarantees under the semismooth assumptions stated in this article}.

In \cref{fig:MF_func} we show also plots of the function residual vs. time and iteration count. In \cref{fig:NMF_seeds} we compare algorithm runs averaged over 20 different random seeds, \rev{including tests on first-order methods, \gladssn for different exponents $p \in [0,1]$, \leapssn and the respective Gauss--Newton approximation (GN)}. In \cref{fig:NMF_PALM}, we additionally compare against the baseline method PALM (see \cite{bolte2014proximal}) -- an alternating gradient method, which we equip with a standard blockwise Armijo line search.

\rev{\paragraph{Results.}
The results averaged over 20 random seeds in \cref{fig:NMF_seeds} show that \gladssn attains substantially smaller gradient norms than gradient descent and PALM. Although cubic Newton requires fewer iterations, \gladssn achieves competitive runtimes and performs comparably to \leapssn. The full generalized-Hessian variants also achieve smaller gradient norms than their Gauss--Newton counterparts in this experiment. The objective residuals exhibit similar qualitative behavior, corroborating these findings.}

\begin{figure}[h!]
\centering
 \includegraphics[width = 1\textwidth]{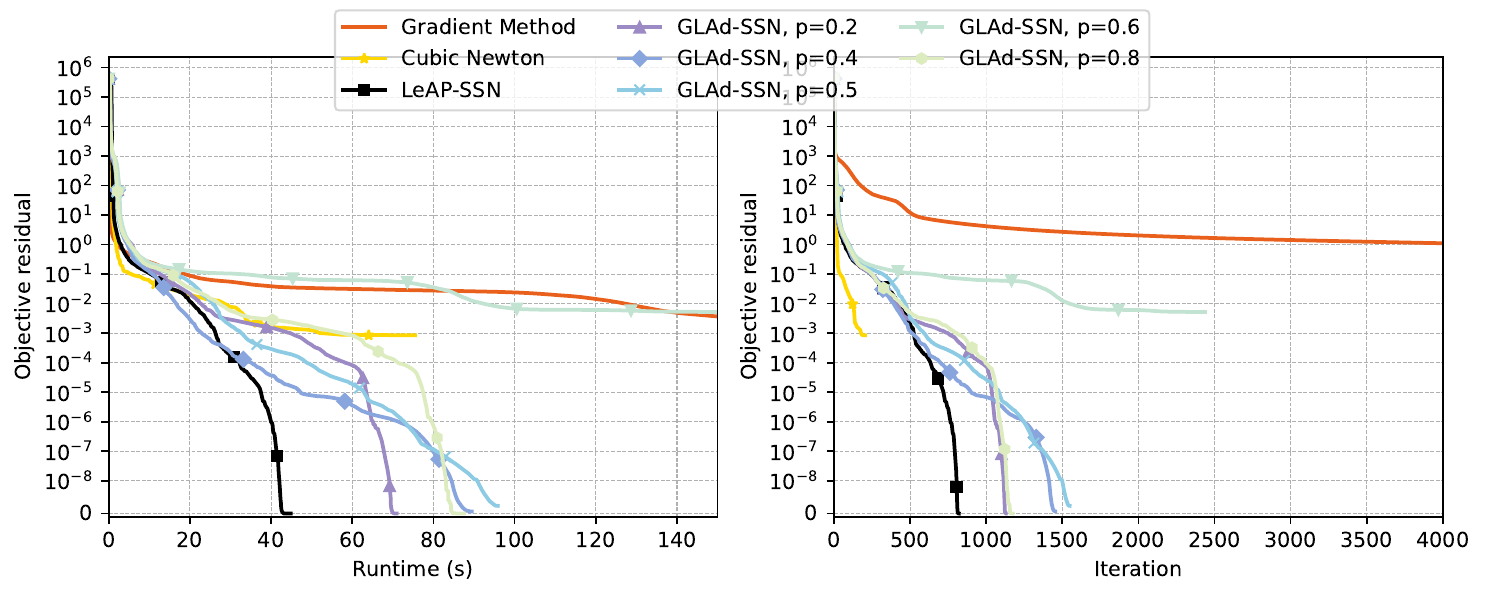}
\caption{Non-negative matrix factorization (\cref{sec:numericsNMF}): a plot of function residuals shows similar qualitative behaviour to the gradient norm (see \cref{fig:MF_grad}). }
\label{fig:MF_func}
\end{figure}

\begin{figure*}%
\centering
 \includegraphics[width =1\textwidth]{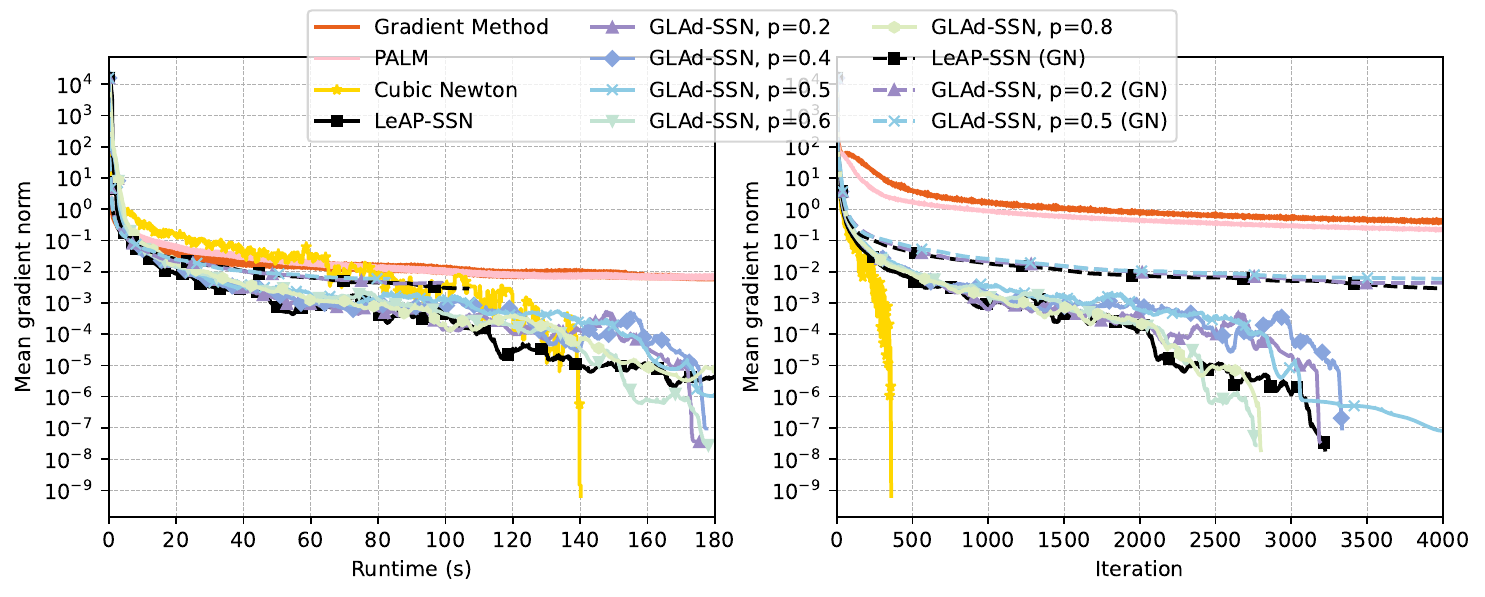}
 
\caption{Non-negative matrix factorization (\cref{sec:numericsNMF}): 
\rev{gradient norms averaged over 20 random seeds.}}
\label{fig:NMF_seeds}
\end{figure*}

\begin{figure*}%
\centering
 \includegraphics[width = 1\textwidth]{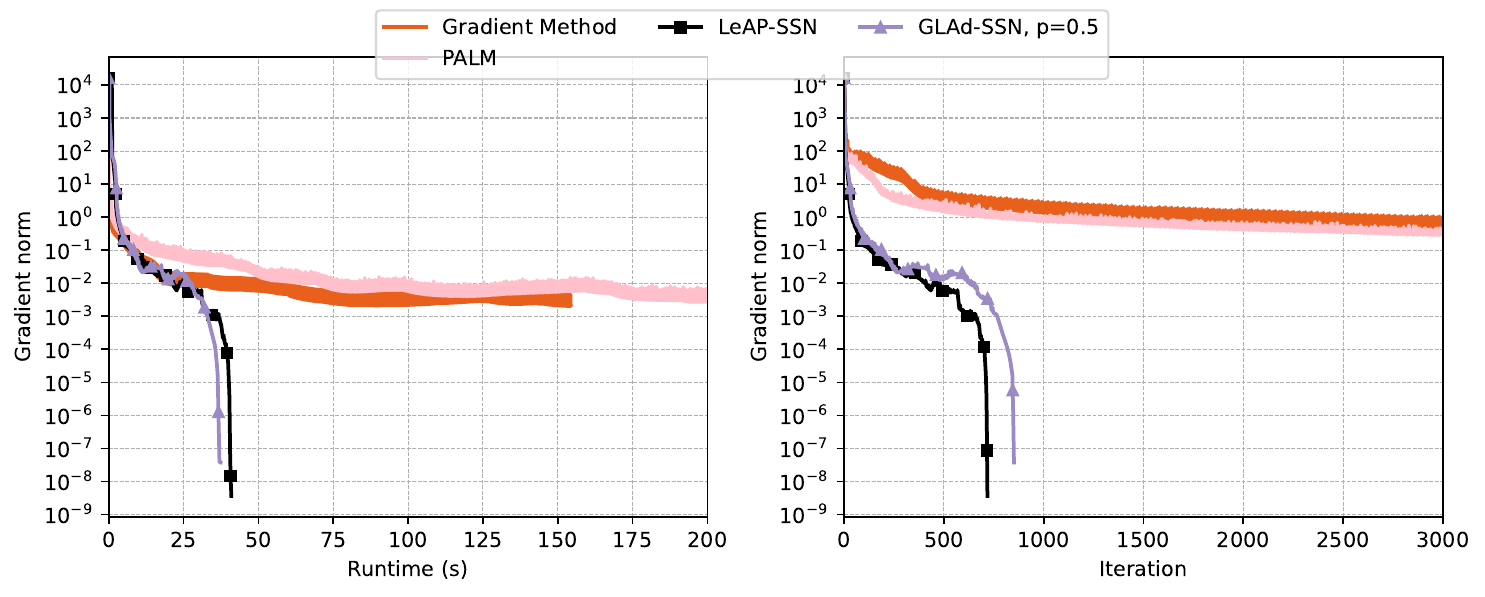}
\hspace{-0.2cm}
\includegraphics[scale=0.6]{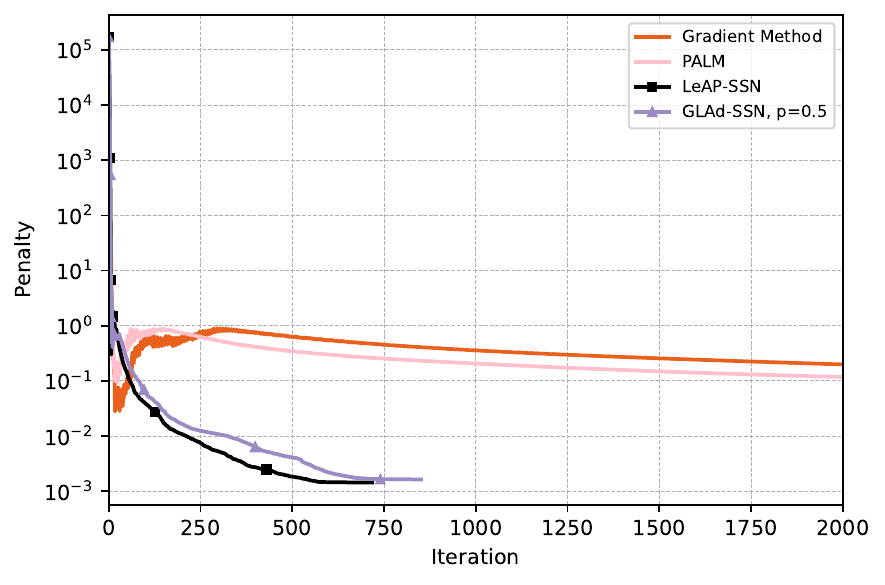}
\caption{\rev{Non-negative matrix factorization (\cref{sec:numericsNMF}): comparison with PALM on a separate illustrative run.
Gradient norms are shown above and the non-negativity penalty below.}}
\label{fig:NMF_PALM}
\end{figure*}
\subsection{Support vector classification}

\paragraph{Semismooth differentiability.}
The regularization term $\frac12\|\omega\|_{\ell^2}^2$ is a polynomial in the components of $\omega$ and hence $C^\infty$.
It therefore remains to treat the squared max term. Define the affine residuals
\[
r_i(\omega,b) \coloneqq 1 - y_i(\omega^\top x_i + b), \qquad i=1,\dots,\ell.
\]
Using again the fact that $\max(\cdot,0)^2$ is $C^1$ with gradient $\rev{2}\max(\cdot,0) = \rev{2}(\cdot)_+$ component-wise,  we may write the gradient of the objective $\nabla F$ as
\[
\nabla_\omega F(\omega,b)=\omega - 2\gamma\sum_{i=1}^\ell y_i x_i\, r_i(\omega,b)_+,
\qquad
\partial_b F(\omega,b)= - 2\gamma\sum_{i=1}^\ell y_i\, r_i(\omega,b)_+.
\]
The function $(\cdot)_+$ is piecewise affine and hence $1$-order semismooth, again see \cite{cui2021modern}.
\paragraph{Experimental data}
The results are shown for the choice $n=1000$, $\ell=10^4$ and $\gamma=10^4$. The data $(x_i)$ and labels $(y_i)$ are generated by a reproducible classification problem generator in the package \texttt{scikit-learn} \footnote{\url{https://scikit-learn.org/}, License: BSD-3-Clause.}.

\paragraph{Optimization and parameters}
We apply Algorithm~\ref{alg:proximal_newton} with $\Lambda_0 = 1$ and  gradient exponent $p = 0.5$ and compare different Hessian update frequencies $m \in \{1,2,4,5,10\}$. As before, we have $\psi \equiv 0$, so each proximal Newton subproblem reduces to the solution of a linear system which we solve using Cholesky decomposition. Since the overall problem is convex, the systems are well conditioned. We benchmark our method against \leapssn proposed in \cite{LeapSSN}, following the experimental protocol suggested therein, and against standard gradient descent equipped with an Armijo backtracking line search.
\paragraph{Results.}
 Our method \gladssn reaches a gradient norm below $10^{-5}$ faster than LeAP-SSN, cubic Newton and SVRG-BB in the experiment \cref{fig:SVM}. Moderate Hessian reuse further reduces runtime: $m=4$ requires approximately $5.5$ seconds and six Hessian evaluations, compared with $7.1$ seconds and seventeen evaluations for $m=1$. Larger reuse periods require more iterations, illustrating the tradeoff between Hessian evaluation cost and iteration count.

\section{Limitations}
\label{sec:limitations}
While our results establish a broad convergence theory for a single class-agnostic semismooth Newton method, several limitations remain. 

First, our global linear and subsequent superlinear convergence results are stated under a global Polyak--\L{}ojasiewicz (PL) condition. This is a standard and widely used assumption in machine learning theory, especially in analyses of overparameterized and nonconvex models as pointed out in the main body of the article. Nevertheless, a local PL inequality, or more generally a Kurdyka--\L{}ojasiewicz-type condition, would further broaden the scope of applicability. We expect that such a local theory is possible, but the proofs are more involved; for simplicity and clarity, we present the global PL version.

Second, our local fast-rate theory relies on Dennis--Mor\'e-type conditions, and the $\gamma$-order result in particular requires a $\gamma$-order version of such a condition. These assumptions provide useful insight into the mechanism behind local fast convergence, especially the interaction between semismoothness, gradient regularization, and lazy Hessian updates. However, a more intrinsic characterization of these conditions is still needed, for instance through active-manifold or partial-smoothness structures. Additionally, our improved local rates rely on local (strong) convexity, though it is standard in the literature on semismooth Newton methods to assume some nondegeneracy around the solution.

Third, our theoretical analysis assumes exact solutions of the proximal Newton subproblems. As pointed out in the experiments section, the method appears robust to different linear and least-squares solvers, including settings where the Hessian systems are indefinite or ill-conditioned. A detailed convergence analysis under inexact subproblem solves would be important for large-scale implementations, but is outside the scope of the present work.

Fourth, the improved $\gamma$-order local convergence result is new and clarifies how gradient regularization interacts with $\gamma$-order semismoothness to yield adaptively faster local rates. At present, however, this result is proved only in the non-lazy setting $\lzsymb=1$. Extending the proof to lazy-update regimes is not merely a technical substitution: with our current proof strategy, one would need to verify a $\gamma$-order analogue of the lazy Dennis--Mor\'e-type condition \eqref{eq:lazy_DM_condition}. Even under H\"older continuity of the generalized Hessian, this would require additional control of the distance between the current iterate and the stale Hessian point, which is not provided by our present analysis. Establishing such conditions, or replacing them by more intrinsic structural assumptions, is left for future work.

Fifth, our theoretical results quantify the number of Hessian evaluations for a fixed lazy-update period $\lzsymb$, but they do not provide a principled rule for choosing $\lzsymb$ in practice. The best value of $\lzsymb$ depends on the relative cost of Hessian evaluations, linear-system solves, and backtracking steps, as well as on the local geometry of the objective. Developing adaptive rules for selecting the Hessian-update frequency is therefore an important practical question left open by our theory.

Finally, although our experiments cover relevant machine learning problems, they remain relatively small-scale. In particular, the neural-network example is intended as a first demonstration of semismooth second-order solvers for Lipschitz-constrained training, but a more realistic imaging processing or inverse problems setup would require parametrizing high-dimensional denoisers. To the best of our knowledge, this is the first work to rigorously demonstrate semismooth Newton-type solvers with global convergence rates on such classes of problems. At present, however, direct applicability is limited to small- and medium-scale settings. Extending the approach to stochastic or randomized Newton-type methods is an important direction for making it practical in higher dimensions.

\rev{We also note that our other assumptions, such as the global minimum assumption, (strong) convexity, Lipschitz derivative, semismoothness are standard for the theoretical analysis of optimization algorithms.}
\end{document}